\documentclass[12pt, reqno]{amsart}
\usepackage{amsmath, amsthm, amssymb, amsfonts}
\usepackage[normalem]{ulem}
\usepackage{hyperref}

\usepackage{verbatim} 
\usepackage{longtable}
\usepackage{caption}
\usepackage{mathtools}
\DeclarePairedDelimiter\ceil{\lceil}{\rceil}

\allowdisplaybreaks

\usepackage{tikz}

\usepackage{tikz}
\usepackage{verbatim}
\usetikzlibrary{calc,3d}

\definecolor{tealgreen}{HTML}{1B9E77}
\definecolor{orange}{HTML}{D95F02}
\definecolor{purple}{HTML}{7570B3}
\definecolor{pink}{HTML}{E7298A}
\definecolor{grassgreen}{HTML}{66A61E}
\definecolor{goldyellow}{HTML}{E6AB02}
\definecolor{brown}{HTML}{A6761D}
\definecolor{devilgray}{HTML}{666666}

\usepackage{caption}

\usepackage{tikz-cd}
\usetikzlibrary{arrows}

\theoremstyle{plain}
\newtheorem{thm}{Theorem}[section]

\newtheorem{lemma}[thm]{Lemma}
\newtheorem{prop}[thm]{Proposition}

\newtheorem{conjecture}{Conjecture}  
\newtheorem{conj}[conjecture]{Conjecture}

\theoremstyle{definition}
\newtheorem{defn}[thm]{Definition}
\newtheorem{example}[thm]{Example}
\newtheorem{rmk}[thm]{Remark}

\theoremstyle{remark}

\newcommand{\BA}{{\mathbb{A}}}

\newcommand{\BC}{{\mathbb{C}}}

\newcommand{\BH}{{\mathbb{H}}}

\newcommand{\BP}{{\mathbb{P}}}
\newcommand{\BQ}{{\mathbb{Q}}}
\newcommand{\BR}{{\mathbb{R}}}

\newcommand{\BZ}{{\mathbb{Z}}}

\newcommand{\CE}{{\mathcal E}}
\newcommand{\CF}{{\mathcal F}}

\newcommand{\CN}{{\mathcal N}}
\newcommand{\CO}{{\mathcal O}}
\newcommand{\CP}{{\mathcal P}}

\newcommand{\CX}{{\mathcal X}}

\newcommand{\Fg}{{\mathfrak{g}}}

\newcommand{\Fp}{{\mathfrak{p}}}

\newcommand{\Fs}{{\mathfrak{s}}}

\newcommand{\Fz}{{\mathfrak{z}}}

\newcommand{\pt}{{\mathsf{p}}}

\DeclareMathOperator{\Hilb}{Hilb}

\newcommand{\ZmodN}[1]{\BZ_{#1}}
\newcommand{\vir}{\mathsf{vir}}
\newcommand{\fix}{\mathsf{fix}}
\renewcommand{\vert}{\mathsf{vert}}
\newcommand{\diag}{\mathsf{diag}}
\newcommand{\Vsf}{\mathsf{V}}
\newcommand{\Vtilde}{\widetilde{\mathsf{V}}}
\newcommand{\Quot}{\operatorname{Quot}}
\newcommand{\Elambda}{\mathsf{E}_{\lambda }}

\DeclareFontFamily{OT1}{rsfs}{}
\DeclareFontShape{OT1}{rsfs}{n}{it}{<-> rsfs10}{}
\DeclareMathAlphabet{\curly}{OT1}{rsfs}{n}{it}

\newcommand{\p}{\mathbb{P}}

\newcommand\Coker{\operatorname{Coker}}

\newcommand*\dd{\mathop{}\!\mathrm{d}}
\newcommand{\Mbar}{{\overline M}}

\newcommand{\Coh}{\mathrm{Coh}}
\newcommand{\Pic}{\mathop{\rm Pic}\nolimits}
\newcommand{\PT}{\mathsf{PT}}
\newcommand{\DT}{\mathsf{DT}}

\newcommand{\Sym}{{\mathrm{Sym}}}

\newcommand{\ev}{{\mathrm{ev}}}

\newcommand{\NS}{\mathrm{NS}}
\newcommand{\Null}{\mathrm{Null}}
\newcommand{\id}{{\mathrm{id}}}

\newcommand{\Jac}{\mathrm{Jac}}
\newcommand{\Mod}{\mathrm{Mod}}

\newcommand{\SL}{{\mathrm{SL}}}
\newcommand{\GL}{{\mathrm{GL}}}
\newcommand{\Sp}{{\mathrm{Sp}}}

\newcommand{\Ell}{\mathcal{E}ll}

\begin{document}
\baselineskip=16pt
\title[CHL Calabi--Yau threefolds]{CHL Calabi--Yau threefolds: Curve counting, Mathieu moonshine and Siegel modular forms}
\date{\today}

\author[J.~Bryan]{Jim Bryan}
\address{University of British Columbia, Department of Mathematics}
\email{jbryan@math.ubc.ca}

\author[G.~Oberdieck]{Georg Oberdieck}
\address{Mathematisches Institut, Universit\"at Bonn}
\email{georgo@math.uni-bonn.de}

\begin{abstract}
A CHL model is the quotient of $\mathrm{K3} \times E$ by an order $N$ automorphism
which acts symplectically on the K3 surface and acts by shifting by an $N$-torsion point on the elliptic curve $E$.
We conjecture that the primitive Donaldson--Thomas partition function of elliptic CHL models
is a Siegel modular form, namely the Borcherds lift of the corresponding twisted-twined elliptic genera which appear in Mathieu moonshine.
The conjecture matches predictions of string theory by David, Jatkar and Sen.
We use the topological vertex to prove several base cases of the conjecture.
Via a degeneration to $\mathrm{K3} \times \p^1$
we also express the DT partition functions as a twisted trace of an operator on Fock space.
This yields further computational evidence.
An extension of the conjecture to non-geometric CHL models is discussed.

We consider CHL models of order $N=2$ in detail. We conjecture a formula
for the Donaldson--Thomas invariants of all order two CHL models in all curve classes.
The conjecture is formulated in terms of two Siegel modular forms.
One of them, a Siegel form for the Iwahori subgroup, has to our knowledge not yet appeared in physics.
This discrepancy is discussed in an appendix with Sheldon Katz.
\end{abstract}
\maketitle

\vspace{-20pt}
\setcounter{tocdepth}{1} 
\tableofcontents
\setcounter{section}{-1}

\section{Introduction}
In this paper we conjecture a connection between Mathieu moonshine
and the enumerative geometry of algebraic curves in certain Calabi--Yau threefolds.
The connection is motivated by physics, and part of our conjecture
can be understood as a mathematical formulation of a prediction by heterotic duality.
However, the connection in general is more subtle than what has been suggested
and new input appears on the curve counting side.

\subsection{Mathieu moonshine}
Eguchi, Ooguri, and Tachi\-kawa \cite{EOT} noted
that the coefficients of the Fourier expansion of the elliptic genus of a K3 surface
\[ \Ell(K3)(\tau,z) = 8 \left[ \left( \frac{\theta_2(\tau,z)}{\theta_2(\tau,0)} \right)^2 
+ \left( \frac{\theta_3(\tau,z)}{\theta_3(\tau,0)} \right)^2 
+\left( \frac{\theta_4(\tau,z)}{\theta_4(\tau,0)} \right)^2 
\right]
\]
can be decomposed into dimensions of representations of the Mathieu group $M_{24}$
times characters of the $N{=}4$ super conformal algebra.
This observation, called Mathieu moonshine, was proven recently by Gannon \cite{Gan}.
By Gaberdiel et all \cite{GPRV} the decomposition of the elliptic genus may be used to define for every pair of commuting elements $g, h \in M_{24}$
the $g$-twisted $h$-twined elliptic genus
\[ \Ell_{g,h}(K3)(\tau,z). \]
Just as $\Ell(K3)$ the genera $\Ell_{g,h}(K3)$ are Jacobi forms \cite{EZ}. 
Roughly, twining correspond to replacing dimensions of $M_{24}$ representations by traces over $h$.\footnote{The twining genera
$\Ell_{g=\text{id},h}(K3)$ are the analogs of the McKay--Thompson series which appear in
Monster moonshine.
}
Twisting is a certain orbifolding process.
Here we treat the construction of twisted-twined elliptic genera as a blackbox and instead
will list them explicitly whenever we need them.

\subsection{CHL Calabi--Yau threefolds}
Let $S$ be a non-singular projective K3 surface endowed with a symplectic automorphism
\[ g : S \to S \]
of finite order $N$.
Let $E$ be a non-singular elliptic curve and let $e_0 \in E$ be a $N$-torsion point.
The group 
\[ \BZ_N = \BZ/ N \BZ \]
 acts on the product $S \times E$ by the map
\[ (s,e) \mapsto (gs, e + e_0).
\]
The Chaudhuri--Hockney--Lykken (CHL) model associated to $g$ is the quotient
\[ X = (S \times E)/\BZ_N. \]
Since $\BZ_N$ acts freely and preserves the Calabi--Yau form,
$X$ is a non-singular projective Calabi--Yau threefold.\footnote{
After pullback to $S \times E$,
the holomorphic symplectic form on $S$
as well as the holomorphic $1$-form on $E$ descend to $X$.}
The elliptic curve $E$ acts on the product $S \times E$ by translation in the second factor.
This action descends to an $E$-action on $X$.

\subsection{Donaldson--Thomas theory}
Let $\Hilb^n(X,\beta)$ be the Hilbert scheme of $1$-dimensional subschemes $Z \subset X$
satisfying
\[ [ Z ] = \beta \in H_2(X,\BZ), \quad \chi(\CO_Z) = n \in \BZ. \]
The action of the elliptic curve $E$ on $X$ induces an action on the Hilbert scheme.
Hence (almost) every curve or subscheme on $X$ comes in the $1$-dimensional family of its $E$-translates.
A count of these $E$-orbits is defined by integrating
with respect to the (stacky) topological Euler characteristic $e( \cdot )$
over the quotient stack: 
\[ \DT^X_{n,\beta} = \int_{\Hilb^n(X,\beta) / E} \nu \dd{e} = \sum_{k \in \BZ} k \cdot e\left( \nu^{-1}(k) \right). \]
Here 
$\nu : \Hilb^n(X,\beta) / E \to \BZ$
is the Behrend weight. The numbers
\[ \DT^X_{n,\beta} \in \BQ \]
are called the $E$-reduced Donaldson--Thomas invariants of $X$ in class~$\beta$.\footnote{
In Section~\ref{subsection:GWtheory} we conjecture a correspondence (in the usual way) between the reduced Gromov--Witten and Donaldson--Thomas theories of $X$.
Our conjectures below hence can be understood purely on the Gromov--Witten side.
}

\subsection{Homology}
Consider the averaging operator
\[ P = \frac{1}{N} \sum_{i=0}^{N-1} g^i_{\ast} \colon H^{\ast}(S,\BQ) \to H^{\ast}(S,\BQ) \]
where we let $g_{\ast}$ denote the induced action on cohomology.
By a Mayer-Vietoris argument 
there exist a canonical isomorphism
\begin{equation} H_2(X,\BZ)/\mathrm{Torsion} \, \cong \, \mathrm{Image}(P|_{H_2(S,\BZ)}) \oplus \BZ. \label{H2isomorphism} \end{equation}
The summands on the right record the degree of a class over $S/\BZ_N$ and $E/\BZ_N$ respectively.
The group of algebraic $1$-cycles on $X$ up to numerical equivalence and torsion is similarly described by
\[ N_1(X) \cong P(N_1(S)) \oplus \BZ. \]

\subsection{Elliptic CHL models} \label{sec:intro:elliptic_CHL}
By a theorem of Mukai \cite{Mukai} every symplectic automorphism $g : S \to S$
defines (up to conjugacy) an element in $M_{24}$ which we denote by $g$ as well.
Mukai's argument is lattice theoretic and it can be shown that the conjugacy class of $g \in M_{24}$ only depends on the order $N$ of the symplectic automorphism.
Let
\begin{equation} F^{(r,s)}_N = \Ell_{g^r, g^s}(K3),\ \  r,s \in \{ 0,1, \ldots, N-1 \}. \label{tw_ti} \end{equation}
be the associated $g^r$-twisted $g^s$-twined elliptic genera.
Explicit expressions for these functions can be found in Appendix~\ref{appendix_elliptic_genera}.

Heterotic duality \cite{JS,DS,DJS,DJS2} predicts that the twisted-twined elliptic genera \eqref{tw_ti} encode the Donaldson--Thomas theory of CHL models.
However, unlike the elliptic genera, the Donaldson--Thomas theory of a CHL model does \emph{not} only depend on the order $N$,
but also on more refined data. 
For fixed polarization degree on the K3 and given $N$, there can be several (but at most finitely many) distinct deformation classes of CHL models.
In the following, we connect one of these deformation classes -- the elliptic CHL models --
to the physics formula.

Let $p : S \to \p^1$ be an elliptically fibered K3 surface which admits two sections
\[ \sigma_0, \sigma_1 : \p^1 \to S. \]
We declare $\sigma_0$ to be the zero section and we assume that $\sigma_1$ is of order~$N$ with respect to $\sigma_0$.
The translation by $\sigma_1$ in the elliptic fibers
\[
g: S \to S,\ s \mapsto s + \sigma_1(p(s))
\]
is a symplectic automorphism of order $N$.
We call the CHL model $X$ associated to $g$ an \emph{elliptic CHL model}.

Consider the $N$ sections 
\[ \sigma_0, \sigma_1, \sigma_i := g(\sigma_{i-1}), i=2, \ldots, N-1 \]
and let $F \in \Pic(S)$ be the class of a fiber of $S \to \p^1$.
The classes
\[ \beta_h = \frac{1}{N}( \sigma_0 + \ldots + \sigma_{N-1} + h F ),
\ \ h \geq 0
\]
lie in the image of $P|_{H_2(S,\BZ)}$ and define curve classes on $X$ via the isomorphism \eqref{H2isomorphism}.
%
%
The primitive Donaldson--Thomas partition function of $X$ is defined by
\[ \mathsf{Z}^X(q,t,p) = \sum_{h=0}^{\infty} \sum_{d=0}^{\infty} \sum_{n \in \BZ}
\DT^X_{n, (\beta_h,d)} q^{d-1} t^{\frac{1}{2} \langle \beta_h, \beta_h \rangle} (-p)^{n} \]
where we let $\langle \alpha, \beta \rangle = \int_S\alpha \cup \beta$ denote the intersection pairing on $S$.

We have the following conjecture that relates Mathieu moonshine to Donaldson--Thomas theory.
Let
\[ Z = \begin{pmatrix} \tau & z \\ z & \sigma \end{pmatrix} \]
be the coordinate on the genus $2$ Siegel upper half plane, and write
\begin{equation} q = e^{ 2 \pi i \tau}, \quad t = e^{ 2 \pi i \sigma}, \quad p = e^{2 \pi i z}. \label{varchange} \end{equation}
Consider the Borcherds lift of the twisted-twined elliptic genera \eqref{tw_ti},
\[ \widetilde{\Phi}_N(Z) = \widetilde{\Phi}_N(q,t,p). \]
We refer to Section~\ref{sec:mult_lift} for a precise definition.
We consider here $\tilde{\Phi}_N$ as a formal power series in the variables $q,t,p$ expanded in the region
\[ 0 < |q|, |t| \ll |p| < 1. \]

\begin{conj} \label{Conj_intro_1}  \label{Conj_anyN}  Let $X$ be an elliptic CHL model of order~$N$.
Under the variable change \eqref{varchange}
the primitive Donaldson--Thomas partition function of $X$
is the negative reciprocal of the Borcherds lift of the corresponding twisted--twined elliptic genera:
\[ \mathsf{Z}^X(q,t,p) = -\frac{1}{\widetilde{\Phi}_N(Z) }. \]
\end{conj}
\vspace{5pt}

The Borcherds lift $\widetilde{\Phi}_N$ is a Siegel modular form (for a congruence subgroup of $\Sp_4(\BZ)$) of weight
\[ \left\lceil \frac{24}{N+1}\right\rceil -2. \]
In case $N=1$ the Borcherds lift is the Igusa cusp form $\chi_{10}$. Conjecture~\ref{Conj_intro_1}
then specializes to the Igusa cusp form conjecture \cite{K3xE}, proven in \cite{ObPix2, FM1}, governing curve counts in $S \times E$,
\[ \mathsf{Z}^{S \times E} = -\frac{1}{\chi_{10}}. \]

The function $\widetilde{\Phi}_N$ satisfies the symmetry:
\begin{equation} \widetilde{\Phi}_N(q,t,p) = \widetilde{\Phi}_N(t^{1/N}, q^N, p). \label{Symmetry} \end{equation}
A consequence of Conjecture~\ref{Conj_anyN} is the following remarkable non-geometric symmetry of Donaldson--Thomas invariants:
\[ \DT^X_{n,(\beta_h,d)} = \DT^X_{n,(\beta_d,h)}. \]
This symmetry should arise from a certain derived auto-equivlance of the threefold $X$,
see \cite{FM1} for a related case.

\subsection{Results}
The main mathematical result of this paper is a proof of Conjecture~\ref{Conj_intro_1} in several base cases.
%
%
Define the series $\Delta_N(q)$ by
\begin{equation} \sum_{n=0}^{\infty} q^{n-1} e\left( \Hilb^n(S/\BZ_N) \right) = \frac{1}{\Delta_N(q)} \label{Delta_N_def} \end{equation}
where $\Hilb^n(S/\BZ_N)$ is the Hilbert scheme of $0$-dimensional substacks of the quotient stack $S/\BZ_N$ of length $n$.
The function $\Delta_N$ is a cusp form for $\Gamma(N)$ of weight $\ceil{\frac{24}{N+1}}$.
Explicit expressions are listed in Table~\ref{Table_Delta_N}.\footnote{See Lemma~\ref{lem: formula for e(Hilb(Y/G))} for more details on this computation.}

\begin{table}
{\renewcommand{\arraystretch}{1.25}\begin{tabular}{| c | l | }
\hline
$N$ & $\Delta_N(\tau)$  \\
\hline
$1$ &  $\eta(\tau)^{24}$  \\
\hline
$2$ & $\eta(\tau)^8 \eta(2 \tau)^8$  \\
\hline
$3$ &  $\eta(\tau)^6 \eta(3 \tau)^6$  \\
\hline
$4$ & $\eta(\tau)^4 \eta(2 \tau)^2 \eta(4 \tau)^4$  \\
\hline
$5$ & $\eta(\tau)^4 \eta(5 \tau)^4$    \\
\hline
$6$ & $ \eta(\tau)^2 \eta(2 \tau)^2 \eta(3 \tau)^2 \eta(6 \tau)^2$   \\
\hline
$7$ & $\eta(\tau)^3 \eta(7 \tau)^3$   \\
\hline
$8$ & $\eta(\tau)^2 \eta(2 \tau) \eta(4 \tau) \eta(8 \tau)^2$  \\
\hline
\end{tabular}}
\caption{The series $\Delta_N(\tau)$ for all possible orders $N$.
Here $\eta(\tau) = q^{1/24} \prod_{n \geq 1} (1-q^n)$, where $q = e^{2 \pi i \tau}$, is the Dedekind function. 
}
\label{Table_Delta_N}
\vspace{-15pt}
\end{table}
Define also the Jacobi theta function
\[ \Theta(q,p) 
= -i(p^{1/2} - p^{-1/2}) \prod_{m \geq 1} \frac{ (1-pq^m) (1-p^{-1}q^m)}{ (1-q^m)^2 }. \]

\begin{thm} \label{thm:intro1} Let $X$ be an elliptic CHL model of order~$N$. Then
\begin{align*}
\left[  \mathsf{Z}^X(q,t,p) \right]_{t^{-1/N}} & = \frac{1}{\Theta(q^N,p)^2 \Delta_N(q)}  \\
\left[  \mathsf{Z}^X(q,t,p) \right]_{q^{-1}} & = \frac{1}{\Theta(t,p)^2 \Delta_N(t^{1/N})}
\end{align*} 
In particular, Conjecture~\ref{Conj_intro_1} holds after taking coefficients $t^{-1/N}$ or $q^{-1}$.
\end{thm}
\vspace{5pt}

The theorem determines the first coefficient in both the $t$ and $q$ direction of $Z^X$.
The coefficient of $t^{-1/N}$ correspond to curve classes which are of genus $0$ (in a certain sense) in the K3 direction.
The coefficient $q^{-1}$ correspond to curves of degree $0$ over the elliptic curve.
The symmetry between the first $q$ and $t$ coefficient in Theorem~\ref{thm:intro1}
is a special case of the $t \leftrightarrow q^N$ symmetry \eqref{Symmetry}.

The proof of Theorem~\ref{thm:intro1} relies on two approaches. For the $t^{-1/N}$ term we use the
topological vertex method of \cite{Bryan-Kool, Bryan-Kool-Young} to stratify the moduli space and calculate directly.
For the $q^{-1}$ term we use a degeneration to $K3 \times \p^1$
and results of Garbagnati, van--Geemen and Sarti \cite{GS, GarS, GarS2} on elliptic K3 surfaces with $N$-torsion section.
Here the appearence of $\Delta_N$ 
may be viewed as a consequence of the McKay correspondence.

The vertex method also yields the second coefficient in the $t$-expansion of Conjecture~\ref{Conj_intro_1}.
The result requires a technical assumption concerning the Behrend function.
\begin{thm} \label{thm:intro2} Assume Conjecture~21 from \cite{Bryan-Kool} on the Behrend function. Then $\left[  \mathsf{Z}^X(q,t,p) \right]_{t^{0}}
$ is
\begin{multline*} 
\frac{2\phi_1(N)}{\Delta_N (q) \phi_2(N) } 
\left( -12\wp(q^N,p) +  \widetilde{\mathcal{E}}_N(q) - \frac{1}{\phi_1(N)} \sum_{m|N} \widetilde{\mathcal{E}}_m(q) \mu(m)
\right)
\end{multline*}
where $\wp(q,p)$ is the Weierstra{\ss} elliptic function, 
\[ 
\widetilde{\mathcal{E}}_m(q) = E_2(q^m)-\frac{1}{m}E_2(q) 
\]
is a holomorphic weight 2 modular form for $\Gamma(m)$ (see section~\ref{subsec:modular_forms}),
$\mu(m)$ is the M\"obius function, and $\phi_d(N) $ is the number of $N$-torsion points in $\ZmodN{N}^d$ so that $\phi_1 = \phi $ is the usual Euler phi function and $\phi_2(N)$ is the number of $N$-torsion points on an elliptic curve. Explicitly,  
\[
\phi_d(N) = N^d\prod_{p|N} \left( 1-\frac{1}{p^d}\right). 
\]
\end{thm}

\subsection{Order two CHL models} \label{subsec:order_two_CHL}
We consider the Donaldon--Thomas theory of CHL models 
which come from a symplectic involution on the K3 in general.
The reduced Donaldson--Thomas invariants $\DT^X_{n,(\gamma,d)}$
are invariant
under deformations which preserve the Hodge type of the curve class $(\gamma,d) \in H_2(X,\BZ)$.
In case $N=2$ such deformation correspond to deformation of triples
\[ (S, L, \iota : S \to S) \]
where $S$ is a K3 surface, $L \in \Pic(S)$ is an invariant primitive ample class and $\iota$ is the involution.
By the Torelli theorem a K3 surface admits an involution if and only if $E_8(-2) \subset \Pic(S)$.
Hence the moduli space of such triples for fixed degree of $L$ can be described as follows \cite{GS}.
If $L^2 \equiv 2$ mod $4$ there is one connected component corresponding to
K3 surfaces polarized by the lattice
\begin{equation} E_8(-2) \oplus \BZ L . \label{dfsdf} \end{equation}
If $L^2 \equiv 0$ mod $4$ there are \emph{two} connected components: Either the K3 surface is
polarized by the lattice \eqref{dfsdf} or by the degree $2$ overlattice
obtained by adjoining a vector $(L/2, v/2)$ for some $v \in E_8(-2)$,
\begin{equation} \mathrm{Span}_{\BZ}\Big( \BZ L \oplus E_8(-2), \, (L/2, v/2)  \Big). \label{dfsdf2} \end{equation}
In particular, the Donaldson--Thomas invariant does not only depend on the degree of a primitive $\gamma$,
but also on lattice data.

Concretely, define
the \emph{divisibility} of a class $\gamma \in \mathrm{Image}(P|_{N_1(S)})$ to be the maximal integer $m \geq 1$ such that
\[ \frac{\gamma}{m} \in \mathrm{Image}(P|_{N_1(S)}) \subset \frac{1}{2} H_2(S,\BZ). \]
The class $\gamma$ is \emph{primitive} if it is of divisibility $1$.
A primitive class $\gamma$ is
\begin{itemize}
\item \emph{untwisted} if $\gamma \in H_2(S,\BZ)$,
\item \emph{twisted} if $\gamma \in \frac{1}{2} H_2(S,\BZ) \setminus H_2(S,\BZ)$.
\end{itemize}
The untwisted and twisted cases correspond to lattice polarizations by \eqref{dfsdf} and \eqref{dfsdf2} respectively
(in the twisted case, we take $\gamma = L/2$).

Let now $X$ be a $N=2$ CHL model and consider a curve class
\[ \beta = (\gamma, d) \in N_1(X) \subset H_2(X,\BZ), \]
such that $\gamma$ is non-zero and primitive with self-intersection
\[ \langle \gamma, \gamma \rangle 
= 2 s,
\quad
s \in 
\begin{cases}
\BZ & \text{ if } \gamma \text{ untwisted} \\
\frac{1}{2} \BZ & \text{ if } \gamma \text{ twisted}.
\end{cases}
\]
By deformation invariance
the Donaldson--Thomas invariant $\DT^X_{n, (\gamma,d)}$ only depends on $n,s,d$, and whether $\gamma$ is untwisted or twisted.
We write
\[ \DT^X_{n, (\gamma,d)} = 
\begin{cases}
\DT^{\text{untw}}_{n, s, d} & \text{ if } \gamma \text{ is untwisted},\\
\DT^{\text{tw}}_{n, s, d}  & \text{ if } \gamma \text{ is twisted}.
\end{cases}
\]
Form the partition functions of twisted and untwisted primitive invariants:
\begin{equation} \label{defn_part_fn}
\begin{aligned}
\mathsf{Z}^{\text{untw}}(q,t,p) & = \sum_{\substack{s \in \BZ \\ s \geq -1}} \sum_{d \geq 0} \sum_{n \in \BZ} \DT^{\text{untw}}_{n, s, d} q^{d-1} t^{s} (-p)^{n} \\
\mathsf{Z}^{\text{tw}}(q,t,p) & = \sum_{\substack{s \in \frac{1}{2} \BZ \\ s \geq -1/2}} \sum_{n \in \BZ} \sum_{d \geq 0} \DT^{\text{tw}}_{n, s, d} q^{d-1} t^{s}  (-p)^{n}.
\end{aligned}
\end{equation}

The twisted series $\mathsf{Z}^{\text{tw}}$ is precisely the primtive DT partition function of the $N=2$ elliptic CHL.
Hence by Conjecture~\ref{Conj_intro_1} the twisted series is conjecturally determined by
\[
\mathsf{Z}^{\textup{tw}}(q,t,p) = -\frac{1}{\widetilde{\Phi_2}(Z)}.
\]

The untwisted series is new and more interesting.
The following conjecture gives a precise formula.
We refer to Section~\ref{subsec:Examples_siegel_forms} for a precise definition of the modular forms.
\begin{conj} \label{Conj_intro_2} The untwisted series for order two CHL models is determined by
\[
\mathsf{Z}^{\textup{untw}}(q, t,p) =
\frac{-8 F_4(Z) + 8 G_4(Z) - \frac{7}{30} E_{4}^{(2)}(2Z)}{\chi_{10}(Z)}.
\]
\end{conj}
\vspace{5pt}

The function in the denominator is the Igusa cusp form which appears in curve counting on $S \times E$.
The numerator is a sum of two different kinds of modular forms.
The series $G_4(Z)$ and $E_4^{(2)}(2Z)$ are Siegel modular forms of weight $4$
for the level two subgroup $\Gamma^{(2)}_0(2) \subset \Sp_4(\BZ)$.
The function $F_4(Z)$ is a Siegel paramodular form of degree $2$
(these correspond to sections of a line bundle on the moduli space of $(1,2)$ polarized abelian surfaces).
Hence the conjecture implies that $\mathsf{Z}^{\textup{untw}}$ is a Siegel modular form (of weight $-6$) for the level $2$ Iwahori subgroup
\[ B(2) = \Sp_4(\BZ) \cap 
\begin{pmatrix}
\BZ & \BZ & \BZ& \BZ \\ 
2 \BZ & \BZ & \BZ& \BZ \\ 
2 \BZ & 2 \BZ & \BZ& 2 \BZ \\ 
2 \BZ & 2 \BZ & \BZ& \BZ
\end{pmatrix}.
\]

Conjectures~\ref{Conj_intro_1} and~\ref{Conj_intro_2} describe the primitive Donaldson--Thomas invariants of all order $2$ CHL models.
The invariants for imprimtive classes are determined from the primitive ones by a multiple cover formula,
see Conjecture~\ref{Conj_Multiple_Cover} in Section~\ref{dfgadgs} for a precise statement.

\subsection{Open questions and further directions}
(1) The multiplicative lift of the twisted--twined elliptic genera only matches
the Donald\-son--Thomas theory of one of the deformation classes of CHL models.
It would be interesting to connect (as we have done in case $N=2$) the other 
deformation classes to Siegel modular forms as well.
It is natural to expect paramodular or Iwahori Siegel forms here as well.
However, for higher $N$ the number of deformations families is not always known, see for example \cite[6.1]{GarS}.

\vspace{2pt}
\noindent (2) 
Let $G$ be a finite (not necessarily cyclic) group of symplectic automorphisms of a K3 surface $S$,
and assume $G$ embeds into the group of torsion points of an elliptic curve $E$.
The quotient
\[ X = (S \times E)/G \]
is called a \emph{generalized CHL model}.
Our conjectures should have analogs also for these models.
Since $G$ embedds into the torsion points, it is abelian and generated by two elements $g, h \in G$.
A connection between the Donaldson--Thomas partition function and the $g$-twisted $h$-twined elliptic genus $\Ell_{g,h}(K3)$
can be expected \cite[1.2]{PV}.

\vspace{2pt}
\noindent (3) Let $g : D^b(S) \to D^b(S)$ be a derived auto-equivalence that is symplectic and preserves a Bridgeland stability conditions
(in physics, $g$ is called a automorphism of a K3 non-linear sigma model).
Then Gaberdiel, Hohenegger, Volpato \cite{GHV2}, and Huybrechts \cite{Huy2} prove that $g$ yields an element in the Conway group unique up to conjugation.
Moreover a Conway moonshine has been proposed in \cite{DM2}.
It would be interesting to find a $g$-equivariant counting theory that correspond to this moonshine phenomenon.
Following a suggestion by Shamit Kachru a slightly adhoc definition of $g$-equivariant invariants is proposed and discussed in Section~\ref{subsection:non-geometric-CHL}.

\vspace{2pt}
\noindent (4) The Pandharipande--Thomas theory of the relative geometry
\[ S \times \p^1 / \{ S_0 , S_{\infty} \} \]
defines a matrix \cite{K3xE, HilbK3} acting on the Fock space
\[ \bigoplus_{n = 0}^{\infty} H^{\ast}( \Hilb^n(S) ). \]
By the degeneration formula the Donaldson--Thomas partition function of a CHL model $X$ 
can be written as the $g$-twined $g$-twisted trace of this matrix (the formula involves a sum over
coinvariant classes on $S$ which may be interpreted as twisting, see Section~\ref{subsec:degeneration} for details).
Hence the Fock space matrix controls the Donaldson--Thomas theory of all CHL models.
It would be interesting to establish a more direct connection between the matrix and Mathieu moonshine.
We will come back to this question in the future.

\subsection{Plan of the paper}
In Section~\ref{Section:CHL_models} we recall some background on symplectic automorphisms, 
CHL models and their curve counting theories. 
We also discuss the degeneration formula to $K3 \times \p^1$
and define invariants for non-geometric CHL models.
In Section~\ref{section:modular_forms} we discuss Jacobi and modular forms,
and define the modular forms which are relevant to Conjecture~\ref{Conj_intro_2}.
Section~\ref{sec: vertex} contains the proof of the $t^{-1/N}$ coefficient of Theorem~\ref{thm:intro1}, and the proof of Theorem~\ref{thm:intro2}.
Section~\ref{first_coefficient_in_q} contains the proof of the $q^{-1}$ coefficient of Theorem~\ref{thm:intro1}.
In Section~\ref{dfgadgs} we generalize the conjectures on order two CHL models to imprimitive curve classes and provide some evidence.
In the Appendix~\ref{appendix_elliptic_genera} we list explicitly the twisted-twined elliptic genera we use in this paper
and define their multiplicative lift.
In Appendix~\ref{appendix2} (by Sheldon Katz and the second author) we discuss the discrepency between our results for order two CHL models and the string theory predictions.

\subsection{Acknowledgements}
The project originated from discussions at the conference ``Number Theory, Geometry, Moonshine \& Strings II'' organized by
Jeffrey Harvey at the Simons foundation in NYC in March 2018.
Conversations with all participants, but especially John Duncan, Shamit Kachru, Sheldon Katz, and Albrecht Klemm played an important role in our understanding.
We thank them, the organizers and the Simons foundation for support.
We would also like to thank Matthew Dawes, Gebhard~Martin, Greg Martin, Arnav Tripathy, and Max Zimet for useful discussions.

G.~O.~was supported by the National Science Foundation Grant DMS-1440140 while at MSRI, Berkeley in the Spring of 2018.


\section{CHL Models} \label{Section:CHL_models}
In Section~\ref{subsection_symplectic_auto} we review
basic facts on symplectic automorphisms.
An introduction to the subject is Chapter~15 of \cite{Huy}.
We then discuss several topics related to CHL models: their homology and curve classes,
the equality of Donaldson--Thomas and Pandharipande--Thomas invariants,
the degeneration formula to $\mathrm{K3} \times \p^1$,
a computation scheme for the curve counting invariants,
and a conjectural Gromov--Witten/Donald\-son--Thomas correspondence.
A definition of counting invariants for non-geometric CHL models is proposed in Section~\ref{subsection:non-geometric-CHL}.

\subsection{Symplectic automorphisms} \label{subsection_symplectic_auto}
Let $S$ be a complex projective K3 surface
with holomorphic-symplectic form $\sigma \in H^0(S, \Omega_S^2)$.
Let
\[ g : S \to S \]
be an automorphism 
which is \emph{symplectic}, i.e. that satisfies $g^{\ast} \sigma = \sigma$.
We assume that $g$ has finite order $N$.\footnote{ 
See \cite[15.2.5(i)]{Huy} for a projective K3 surface with a symplectic automorphism of infinite order.} 

By the global Torelli theorem
the symplectic automorphism $g$ is uniquely determined by
its induced action on $H^2(S,\BZ)$.
Moreover, by 
\cite[Thm.15.3.13]{Huy}
the action of $g$ on the abstract lattice $H^2(S,\BZ)$ depends
up to an orthogonal transformation of the lattice 
only on the order $N$.
By \cite[15.1]{Huy} the order of $g$ can take every value in the range
\[ 1 \leq N \leq 8. \]

Let $U = \binom{0\ 1}{1\ 0}$ by the hyperbolic lattice. Recall that
\[  \Lambda = H^2(S,\BZ) \cong U^3 \oplus E_8(-1)^2. \]
The \emph{invariant lattice} with respect to $g$ is
\[ \Lambda^g = \{ v \in \Lambda\, | \, g v = v \}. \]
The \emph{coinvariant lattice} of $g$ is the orthogonal complement of the invariant lattice:
\[ \Lambda_g = \left( \Lambda^g \right)^{\perp} \subset H^2(S,\BZ). \]
In particular, $\Lambda_g \otimes \BC$ is the sum of all eigenspaces of the $\BC$-linear extension of $g$ corresponding to eigenvalues different from $1$.
Let $v \in \Lambda_g \otimes \BC$ be an eigenvector to eigenvalue $\lambda \neq 1$. Then
\[ \langle v , \sigma \rangle = \langle g_{\ast} v, g_{\ast} \sigma \rangle = \lambda \langle v, \sigma \rangle \]
and therefore $\langle v , \sigma \rangle = 0$. We conclude $\Lambda_g \subset \mathrm{NS}(S)$.

\begin{figure}
\begin{longtable}{| c | c | c | c | c | c | c | c |}
\hline
$N$ & $2$ & $3$ & $4$ & $5$ & $6$ & $7$ & $8$ \\
\hline
$|\mathrm{Fix}(g)|$ & $8$ & $6$ & $4$ & $4$ & $2$ & $3$ & $2$ \\
\hline
$|\Lambda_g|$ & $8$ & $12$ & $14$ & $16$ & $16$ & $18$ & $18$ \\
\hline
$\rho(S) \geq$ & $9$ & $13$ & $15$ & $17$ & $17$ & $19$ & $19$ \\
\hline
\caption{The number of fixed points,
the rank of the coinvariant lattice and the Picard rank
of a non-trivial symplectic automorphism $g$ of finite order $N$ on a complex projective K3 surface (taken from \cite[15.1]{Huy}).}
\label{Table1}
\end{longtable}
\end{figure}

Consider the projection operator onto the invariant part,
\[ P = \frac{1}{N} \sum_{i=0}^{N-1} g^i \colon H^2(S,\BZ) \to \frac{1}{N} H^2(S,\BZ). \]
Since the image under $P$ of an ample class is ample, there exist an ample invariant class $L \in \mathrm{NS}(S)$.
By the Hodge index theorem $\Lambda_g$ is therefore negative-definite.
Moreover, since $L$ is ample and orthogonal to $\Lambda_g$, the lattice $\Lambda_g$ contains no $(-2)$-classes.

The number of fixpoints, the rank of the co-invariant lattice and a bound for the Picard rank for non-trivial $g$ are listed in Table~\ref{Table1}.

For $N \in \{ 2, \ldots, 8 \}$ the action of the automorphism $g$ and the co-invariant lattices
were explicitly determined in the series of papers \cite{GS, GarS, GarS2}. 
If the K3 surface $S$ is of minimal Picard rank, then its Neron--Severi group is of one of the types listed in \cite[Prop.6.2]{GarS2}.
In the following example we recall the case $N=2$.

\begin{example} 
If $N=2$ then $g : S \to S$ is called a \emph{Nikulin involution}.
Its action on $\Lambda$ is trivial on $U^3$ and interchanges the two copies of $E_8(-1)$.
The invariant and co-invariant lattices are
\[ \Lambda^g = U^3 \oplus E_8(-2), \quad \Lambda_g = E_8(-2), \]
where we have written $E_8(-2)$ for the diagonal and the anti-diagonal in $E_8(-1)^2$ respectively.

Suppose now that $S$ is of minimal Picard rank~$9$ with invariant ample class $L$.
Then by \cite[Prop.2.2]{GS} its Neron-Severi group can be described by one of the following two cases.
In the first case we have
\[ \NS(S) = \BZ L \oplus E_8(-2). \]
We call this case the \emph{untwisted} case.

In the second case, $\NS(S)$ is a finite overlattice of $\BZ L\oplus E_8(-2)$ of degree $2$ obtained
by adjoining a vector $(L/2, v/2)$ for some $v \in E_8(-2)$:
\[ \NS(S) 
= \mathrm{Span}_{\BZ}\Big( \BZ L \oplus E_8(-2), \, (L/2, v/2)  \Big). \]
In particular, since $\NS(S)$ is even, we have $L^2 \equiv 0$ modulo $4$. We call the second case the \emph{twisted} case.

By the Torelli theorem for K3, the moduli space of triples $(S, L, \iota)$ can be described as follows (see also \cite{GS} for details).
If $L^2 \neq 0$ mod $4$ then the moduli space has a single connected component with an open subset parametrizing the untwisted case.
If $L^2 = 0$ mod $4$, then the moduli space has two connected components, corresponding to the untwisted and twisted case respectively.
The moduli space is of dimension $11$.
\end{example}

\begin{rmk} \label{gsdgsd}
The finite groups which act symplectically and faithfully on a given K3 surface $S$ were classified
by Mukai in terms of the Mathieu group, see \cite[Thm.15.3.1]{Huy}.
Aside from the cyclic groups which were discussed above, the following Abelian groups can appear:
\begin{gather*} 
\ZmodN{2}^2,\quad \ZmodN{2}^3, \quad \ZmodN{2}^4, \quad \ZmodN{3}^2, \quad  \ZmodN{4}^2, \\
\ZmodN{2} \times \ZmodN{4}, \quad \ZmodN{2} \times \ZmodN{6}.
\end{gather*}
As before the action of the finite Abelian groups on $H^2(S,\BZ)$ is
unique up to an orthogonal transformation of the lattice; the corresponding (co)invariant lattices have been determined in \cite{GarS2}
and the Neron--Severi groups for minimal Picard rank are listed in \cite[Prop.6.2]{GarS2}.
\end{rmk}

\subsection{Definition} \label{subsec:chl}
Let $g : S \to S$ be a symplectic automorphism of finite order $N$,
let $E$ be a non-singular elliptic curve and let $t \in E$ be a torsion point of order $N$.
Let
\[ X = (S \times E)/ \BZ_N. \]
be the associated CHL Calabi--Yau threefold. We let
\[ \pi : S \times E \to X \]
denote the degree $N$ quotient map, and let
\[ p_1 : X \to S' := S/\BZ_N, \quad p_2 : X \to E' := E / \BZ_N \]
be the maps induced from the projection.
Here $\BZ_N$ acts on $E$ by translation by $t$.

\subsection{Cohomology and $1$-cycles} \label{Subsection_cohomology_and_Neron_Severi}
We describe the cohomology and Neron--Severi group of a CHL model $X$. 
Let $s \in S$ be a fixpoint of $g$ (which exists by Table~\ref{Table1})
and consider the subscheme
\[ (E \times s)/\ZmodN{N} = E' \times s = E'_s \subset X. \]
We often drop the subscript $s$ in $E'_s$.
For any $e \in E$ let also
\[ D_e = \pi(S \times e). \]
We often drop the subscript $e$. By a Mayer-Vietoris argument we have
\begin{align*}
H^2(X,\BZ) & = \mathrm{Ker}(1-g : H^2(S,\BZ) \to H^2(S,\BZ)) \oplus \BZ [D] \\
H^4(X,\BZ) & = \mathrm{Coker}(1-g : H^2(S,\BZ) \to H^2(S,\BZ)) \oplus \BZ[E']
\end{align*}
and
\[ H^4(X,\BZ) = H_2(X,\BZ). \]
In particular, $H_2(X,\BZ)$ might contain torsion.

Consider the
projection operator
\[ P = \frac{1}{N} \sum_{i=0}^{N-1} g^i \colon H^2(S,\BZ) \to \frac{1}{N} H^2(S,\BZ)^g. \]
\begin{lemma} \label{LemmaProj} We have
\[ 
 \mathrm{Coker}\left(1-g : H^2(S,\BZ) \to H^2(S,\BZ)\right)
/ \textup{Torsion} \, \cong \, \mathrm{Im}(P).
 \]
\end{lemma}
\begin{proof}
Since $g$ is of order $N$ we have $P \circ (1-g) = 0$. Hence $P$ factors through the cokernel of $1-g$.
Since the image has no torsion, we get a natural map
\[ \Coker(1-g) / \textup{Torsion} \to \mathrm{Im}(P). \]
A non-zero element in the left hand side lifts to an element $\alpha \in H^2(S,\BZ)$ which does not lie in $\mathrm{Im}(1-g) \otimes \BQ$.
Since
\[ H^2(S,\BZ) \otimes \BQ = \Null(1-g) \oplus \Null(P) = \mathrm{Im}(P) \oplus \mathrm{Im}(1-g) \]
we have $\mathrm{Im}(1-g) \otimes \BQ = \Null(P) \otimes \BQ$, so $P(\alpha) \neq 0$.
\end{proof}

From now on we will work only with integral (co)homology modulo torsion and will write
$H_k(X,\BZ)$ for $H_k(X,\BZ)/\text{Torsion}$, etc.
With this convention by Lemma~\ref{LemmaProj} we therefore have
\begin{equation} H_2(X,\BZ) \cong \mathrm{Im}(P) \oplus \BZ[E']. \label{34sdfsd} \end{equation}
Explicitly, the isomorphisms sends $\beta \in H_2(X,\BZ)$ to
$(\frac{1}{n} \alpha, d)$ where
\[ \pi^{\ast} \beta = \alpha + d [E] \]
where we have used the K\"unneth theorem to identify
\[ H_2(S \times E, \BZ) = H_2(S,\BZ) \oplus H_2(E,\BZ). \]
%
The inverse of \eqref{34sdfsd} is
\[
\mathrm{Im}(P) \oplus \BZ \mapsto H_2(X,\BZ),\ (\gamma,d) \mapsto \pi_{\ast} \iota_{S \ast} \gamma + d [E']
\]
where $\iota_S : S \to S \times E$ is the inclusion of a fiber of $p_2$.
We often identify elements in $H_2(X,\BZ)$ with their image under \eqref{34sdfsd}.

The group of $1$-cycles $N_1(X)$ on $X$ up to numerical equivalence and torsion described as follows.
\begin{lemma} Under the identification \eqref{34sdfsd} we have
\[ N_1(X) = P(N_1(S)) \oplus \BZ [E']. \]
\end{lemma}
\begin{proof} The inclusion $\supset$ follows from $\pi_{\ast} N_1(S \times E) \subset N_1(X)$ and the existence of $E'$.
For the other direction, if $\alpha \in N_1(X)$ then $\pi^{\ast} \alpha = (\alpha_1, \alpha_2)$ with $\alpha_1 \in  N_1(S)^g \subset \Lambda^g$, so $\frac{1}{n} \alpha_1 \in \Lambda^g$.
\end{proof}

\subsection{Pandharipande--Thomas theory} \label{subsection_PT_theory}
A stable pair $(\mathcal{F},s)$ on $X$
is a coherent sheaf $\mathcal{F}$ supported in dimension $1$ together with a section
$s \in H^0(X, \mathcal{F})$ satisfying the following stability conditions:
\begin{enumerate}
\item[(i)] the sheaf $\mathcal{F}$ is pure
\item[(ii)] the cokernel of $s$ is $0$-dimensional.
\end{enumerate}
Let $P_n(X, \beta)$ be the moduli space of stable pairs
with Euler characteristic and
the class of the support $C$ of $\mathcal{F}$ satisfying
\[
\chi(\mathcal{F})=n\in \mathbb{Z}, \quad \ [C]= \beta \in H_2(X,\mathbb{Z})\,.
\]

Consider a curve class
\[ \beta = (\gamma,d) \in H_2(X,\BZ). \]
The elliptic curve $E$ acts on the moduli space $P_n(X,\beta)$ by translation.
If $\gamma > 0$ or $n \neq 0$ this action has finite stabilizers and we define reduced Pandharipande--Thomas invariants by
\[ \PT^X_{n,\beta} = \int_{ P_n(X,\beta) / E} \nu \dd{e} \]
where $\nu : P_n(X,\beta) / E \to \BZ$ is the Behrend function.

\begin{prop}
If $\gamma > 0$, then $\PT^X_{n,\beta} = \DT^X_{n,\beta}$.
\end{prop}
\begin{proof}
This follows by the argument of \cite[4.11]{Red} from the $C$-local DT/PT correspondence
by integrating over the quotient of the Chow variety of curves by $E$.\footnote{
If $\gamma = 0$ then the proposition is false, and
$\PT_{n,\beta}^X$ and $\DT_{n,\beta}^X$ differ by a non-zero wall-crossing contribution,
see \cite{OS} for the case $S \times E$.}
\end{proof}

If $\gamma > 0$ then the moduli space $P_n(X,\beta)$ also carries a reduced virtual fundamental class
\[ \big[ P_n(X,\beta) \big]^{\text{red}} \in H_{\ast}( P_n(X,\beta) ) \]
obtained from reducing the perfect obstruction theory of the moduli space
by the holomorphic $2$-form pulled back from $S'$.\footnote{
The holomorphic $2$-form produces a reduced virtual class by the cosection localization method of Kiem--Li \cite{KL}.
But the argument of \cite[Prop.1]{Red} and using that the automorphism $g : S \to S$
extends to the twister family \cite[15.1.2, Footnote 2]{Huy}
even yields a reduced perfect obstruction theory (a strictly stronger statement).
}
We will relate the invariants defined by cutting down the reduced virtual class by an insertion
with the Pandharipande--Thomas invariants $\PT^X_{n,\beta}$.
Let 
\[ q : S \to S' = S / \BZ_n \]
be the projection and let $\gamma^{\vee} \in H^2(S', \BQ)$ be any class such that
\[ \int_S \gamma \cup q^{\ast}(\gamma^{\vee}) = 1. \]
Recall also the divisor $D = D_e = \pi(S \times e)$. We have
\[ p_2^{\ast} [\pt] = n [D]. \]
Define the reduced incidence Pandharipande--Thomas invariant
\[ \widetilde{\PT}^X_{n,\beta} = \int_{ [ P_n(X,\beta) ]^{\text{red}} } \tau_0( [D] \cup p_1^{\ast}(\gamma^{\vee}) ), \]
where the insertion operator $\tau_0( \cdot )$ is defined in \cite{PT1}.

By arguments parallel to \cite{Red} we have the following comparision. 

\begin{prop} \label{Prop:quotient=incidence}
If $\gamma>0$, then $\PT^X_{n,\beta} = \widetilde{\PT}^X_{n,\beta}$.
\end{prop}

By deformation invariance of the reduced virtual class the $\widetilde{\PT}^X_{n,\beta}$ are invariant under deformations of $(X,\beta)$ which
keep the class $\beta$ algebraic.
Hence Proposition~\ref{Prop:quotient=incidence} implies the deformation invariance of $\PT^X_{n,\beta}$.

\subsection{Rubber invariants}
We relate the Pandhari\-pande--Thomas invariants of $X$ to rubber invariants on $\mathrm{K3} \times \p^1$.
These are defined as follows.
Consider the relative geometry
\begin{equation} S \times \p^1 / \{ S_0, S_{\infty} \} \label{relative} \end{equation}
where $S_0, S_{\infty}$ are the fiber over $0, \infty \in \p^1$ respectively. 
Let 
\[ P_{n}^{\sim}(S \times \p^1 / \{ S_0, S_{\infty} \}, (\gamma,d)) \]
be the moduli space of stable pairs on the relative geometry \eqref{relative}
modulo the $\BC^{\ast}$-scalling on $\p^1 / \{ 0, \infty \}$ (this is
also called the moduli space of stable pairs on the rubber of \eqref{relative}, see \cite{K3xE}).
The moduli space is of reduced virtual dimension $2d$ (assuming $\gamma > 0$)
and admits evaluation maps
\[ \ev_0, \ev_{\infty} : P_{n}^{\sim}(S \times \p^1 / \{ S_0, S_{\infty} \}, (\gamma,d)) \to \Hilb^d(S) \]
over the points $0, \infty \in \p^1$ respectively.

Consider cohomology classes
\[ \mu, \nu \in H^{\ast}(\Hilb^d(S), \BQ). \]
We define the rubber Pandharipande--Thomas invariants by
\[
\PT^{S \times \p^1}_{n, (\gamma,d)}( \mu, \nu )
=
\int_{ 
\left[ P_{n}^{\sim}(S \times \p^1 / \{ S_0, S_{\infty} \}, (\gamma,d)) \right]^{\text{red}}
} 
\ev_0^{\ast}(\mu) \cup \ev_{\infty}^{\ast}(\nu).
\]

\subsection{Degeneration to $\mathrm{K3} \times \p^1$} \label{subsec:degeneration}
Let 
\[ \CE \to \Delta \]
be a non-singular elliptically fibered surface over a disk $\Delta \subset \BC$ such that the following conditions hold:
\begin{enumerate}
\item The fiber over $1 \in \Delta$ is isomorphic to $E$.
\item The fiber over $0 \in \Delta$ is isomorphic to a cycle of $N$ copies of $\p^1$ (a $I_n$ fiber in Kodaira's classification)
\item There exist two sections $s_0, s_1$. We take $s_0$ to be the zero section,
and we require the section $s_1$ to be of order $N$ with respect to the group law defined by $s_0$.
\item The induced action of $s_1$ on the fiber over $0$ sends the $i$-th copy to the $(i+1)$-th copy (modulo $N$).
\end{enumerate}
Consider the order $N$ automorphism on the product $S \times \CE$
which acts by $g$ on the first, and by addition by $s_1$ on the second factor.
The quotient by this free action is a non-singular $4$-fold
\[ \CX = (S \times \CE)/ \ZmodN{N}. \]
Let
$f : \CX\to \p^1$
be the fibration induced by $\CE \to \p^1$.
We have
\[ f^{-1}(1) = X, \quad f^{-1}(0) = (S \times \p^1)/\sim \]
where the cylinder $S \times \p^1$ is glued to itself via the monodromy relation
\[ (s,0) \sim (g(s), \infty) \text{ for all } s \in S. \]
Hence $\CX$ is the total space of a degeneration
\begin{equation} X \, \rightsquigarrow \, (S \times \p^1)/\sim. \label{degeneration} \end{equation}

By Proposition~\ref{Prop:quotient=incidence} the reduced Pandharipande--Thomas invariant
is expressed in terms of an integral over the reduced class.
Hence we may apply the degeneration formula to the degeneration \eqref{degeneration}.
The result, after a de-rigidification argument, is as follows.
Let
\[ g : \Hilb^d(S) \to \Hilb^d(S) \]
be the automorphism induced by $g$. Consider its graph
\[ \Gamma_g = \{\, (z, gz) \, | \, z \in \Hilb^d(S) \, \} \subset \Hilb^d(S) \times \Hilb^d(S). \]
Let $H_2(S,\BZ)_{>0}$ be the set of effective curve classes on $S$.
Then
\begin{equation} \label{deg_formula}
\PT^X_{n, (\gamma,d)}
=
\frac{1}{N} \sum_{ \substack{ \tilde{\gamma} \in H_2(S,\BZ)_{>0} \\ P(\tilde{\gamma}) = \gamma }}
\PT^{S \times \p^1}_{n + d, (\tilde{\gamma},d)}( \Gamma_g ).
\end{equation}

%

\subsection{Computation scheme} \label{subsec:computation_scheme}
Modulo known conjectures, the degeneration formula \eqref{deg_formula} yields a computation scheme for the invariants $\PT^X_{n,(\gamma,d)}$ of any CHL model $X$ as follows.
By \eqref{deg_formula} to determine $\PT_{n,(\gamma,d)}$
it is enough to know the rubber invariants $\PT_{n,(\gamma,d)}(\mu,\nu)$.
These are known conjecturally known as follows.

First, the rubber Pandharipande--Thomas invariants
are related by a (conjectural) GW/PT correspondence to rubber Gromov--Witten invariants of $\mathrm{K3} \times \p^1$ \cite{K3xE, K3xP1}.
After applying the product formula on the Gromov--Witten side,
Conjecture~C2 of \cite{K3xE} then expresses invariants for imprimitive classes $\gamma$
in terms of invariants where $\gamma$ is primitive.
Hence we are reduced to the case where $\gamma$ is primitive.

Second, by the conjectural PT/Hilb correspondence of \cite[Sec.5]{K3xE} the rubber invariants of $\mathrm{K3} \times \p^1$
for primitive $\gamma$
are determined by two-point genus $0$ Gromov--Witten invariants of the Hilbert scheme of points $\Hilb^d \mathrm{K3}$.
An effective conjectural formula for these invariants was presented in \cite{HilbK3} (see also \cite{ObPix} for a more explicit presentation).
This completes the scheme.

The scheme we described is effective, i.e. for any given $n$ and $(\gamma,d)$ the invariant $\PT_{n, (\gamma,d)}$
can be computed in finite time.
For us this was one important source of computational evidence for the conjectures in the paper.
However, at present it appears difficult to prove any implications or explicit formulas from this algorithm.\footnote{
This is not unlike the case of the quintic threefold which has been 'algorithmically' solved a long time ago \cite{MP}.
However, explicit formulas for the quintic are known only in low genus.}

\subsection{Gromov--Witten theory}
\label{subsection:GWtheory}
Let $\Mbar^{\bullet}_{h,n}(X,\beta)$ be the moduli space of stable maps $f : C \to X$ from possibly disconnected
$n$-marked curves $C$ of genus $h$ representing the curve class
\[ f_{\ast} [C] = \beta = (\gamma,d) \in H_2(X,\BZ). \]
If $\gamma > 0$ the moduli space
carries a reduced virtual fundamental class
\[ \big[ \Mbar^{\bullet}_{h,n}(X,\beta) \big]^{\text{red}} \in H_{2(n+1)}(\Mbar^{\bullet}_{h,n}(X,\beta)). \]
Reduced Gromov--Witten invariants of $X$ are defined by
\[ \mathsf{N}_{h,\beta} = \int_{ [ \Mbar^{\bullet}_{g,1}(X,\beta) ]^{\text{red}} } \ev_1^{\ast}( [D] \cup p_1^{\ast}( \gamma^{\vee} )) \]
where $\ev_1 :  \Mbar^{\bullet}_{g,1}(X,\beta) \to X$ is the evaluation map at the first marking.

By arguments parallel to \cite[Sec.4]{Red} the formal Laurent series
\[ \sum_{n \in \BZ} \PT_{n,\beta} y^n \]
is the expansion of a rational function in $y$. Hence the variable change
$y = e^{iu}$ is well-defined.

\begin{conj} If $\gamma > 0$, then the GW/PT correspondence holds:
\[ \sum_{h \in \BZ} \mathsf{N}_{h,\beta} u^{2h-2} = \sum_{n \in \BZ} \PT_{n,\beta} y^n \]
under the variable change $y = - e^{iu}$.
\end{conj}

\subsection{Non-geometric CHL models}
\label{subsection:non-geometric-CHL}
The Mukai lattice is the group $H^{\ast}(S,\BZ)$ 
together with the Mukai pairing defined by
\[ \big( (r_1, D_1, n_1) , (r_2, D_2, n_2) \big) = r_1 n_1 + n_1 r_2 - \int_{S} D_1 \cup D_2 \]
for all 
\[ (r_i, D_i, n_i) \in H^{\ast}(S,\BZ) = H^0(S,\BZ) \oplus H^2(S,\BZ) \oplus H^4(S,\BZ) \]
where we have identified $H^0(S,\BZ) = \BZ$ and $H^4(S,\BZ) = \BZ$.

A derived auto-equivalence
\[ g : D^b(S) \to D^b(S) \]
induces an isometry of the Mukai lattice:
\[ g_{\ast} : H^{\ast}(S,\BZ) \to H^{\ast}(S,\BZ). \]
We say $g$ is \emph{symplectic} if
\[ g_{\ast}|_{H^{2,0}(S)} = \id. \]

Let $g : D^b(S) \to D^b(S)$ be a symplectic auto-equivalence which is of finite order $N$ and preserves a Bridgeland stability condition.
Let $e_0 \in E$ be a $N$-torsion point on an elliptic curve $E$,
let $t_{e_0} : E \to E$ be the translation by $e_0$,
and let $t_{e_0 \ast} : D^b(E) \to D^b(E)$ be the induced action on the derived category.
Tensoring the kernel of $g$ with the kernel of $t_{e_0 \ast}$ induces a order $N$ derived auto-equivalence
\[ \tilde{g} = g \boxtimes t_{e_0 \ast} : D^b(S \times E) \to D^b(S \times E) \]
which defines an action of $\BZ_N$ on $D^b(S \times E)$.
We call the pair
\begin{equation} \left(D^b(S\times E), \tilde{g} \right) \label{non_geom_chl} \end{equation}
a \emph{non-geometric} or \emph{non-commutative CHL model}.

We would like to define invariants which count stable sheaves on the non-commutative CHL model \eqref{non_geom_chl},
i.e. some form of $\BZ_N$-equivariant stable complexes in $D^b(S \times E)$.
These invariants should correspond, via an analog of Conjecture~\ref{Conj_intro_1}, to the
$g$-twined elliptic genera which appear in Mathieu or Conway moonshine \cite{DM2}
(by \cite{GHV2,Huy2} $g$ induces an element in the Conway group).
We do not address this task here directly.
Instead following a proposal of Shamit Kachru we define invariants
which should be equivalent to such a count.
The idea is to start with the degeneration formula \eqref{deg_formula}.
The right hand side in \eqref{deg_formula} only depends on
the action on cohomology which a symplectic automorphism induces and not on the symplectic automorphism itself.
We will define Donaldson--Thomas invariants of a non-commutative CHL by
the right hand side of \eqref{deg_formula} but using the induced action $g_{\ast}$
of a derived auto-equivalence $g$.
Intuitively this corresponds to ''gluing'' the cylinder $S \times \p^1$ with respect to the auto-equivalence $g$.
A careful definition (to make sure everything is well-defined) proceeds as follows.

The definition requires a conjectural invariance property of the rubber invariant.
Let $\gamma \in H_2(S,\BZ)$ be an non-zero curve class.
Let 
\[ \varphi : H^{\ast}(S,\BR) \to H^{\ast}(S, \BR) \]
be any orthogonal map (defined over $\BR$, orthogonal with respect to the Mukai lattice) 
which satisfies $\varphi(\gamma) = \gamma$. We let
\[ \varphi : H^{\ast}( \Hilb^d S, \BR) \to H^{\ast}( \Hilb^d S, \BR) \]
be the induced map\footnote{
Concretely, for $\alpha \in H^{\ast}(S)$ and $i \geq 0$
let 
\[ \Fp_{-m}(\alpha) : H^{\ast}(\Hilb^d S) \to H^{\ast}( \Hilb^{d+m} S) \]
 be the Nakajima creation operator
that geometrically adds the cycle of $m$-fat subschemes located on the locus Poincar\'e dual to $\alpha$.
Define the modified creation operator
\[ \tilde{\Fp}_m(\alpha) =
\begin{cases}
(-m)^{-1} \Fp_m(\alpha) & \text{ if } \alpha \in H^2(S) \\
\Fp_m(\alpha) & \text{ if } \alpha \in H^2(S) \\
(-m) \Fp_m(\alpha) & \text{ if } \alpha \in H^4(S).
\end{cases}
\]
Then the induced map $\varphi$ acts by
$\phi\left( \prod_{i} \Fp_{-m_i}(\alpha_i) v_{\varnothing} \right)
=
\prod_{i} \Fp_{-m_i}( \phi(\alpha_i) ) v_{\varnothing}$,
where $v_{\varnothing}$ is the vacuum vector.
}.
We require the following conjecture.
\begin{conj} \label{Conj_invariance} For any $\mu, \nu \in H^{\ast}(\Hilb^d(S))$ we have
\[
\PT^{S \times \p^1}_{n, (\gamma,d)}( \mu, \nu )
=
\PT^{S \times \p^1}_{n, (\gamma,d)}( \varphi(\mu), \varphi(\nu) ).
\]
\end{conj}
The conjecture is a consequence of the conjectural formula for the rubber invariants proposed in \cite{HilbK3}.
If $\varphi$ acts by the identity on $H^0(S)$ and $H^4(S)$, the conjecture specializes to \cite[Conj.C1]{K3xE}.
Let now
\[ \gamma \in H^{\ast}(S,\BZ) \]
be any non-zero Hodge class. Let
\[ \varphi : H^{\ast}(S,\BR) \to H^{\ast}(S, \BR) \]
be an isometry such that $\varphi(\gamma)$ is a Hodge class, lies in $H^2(S,\BZ)$ and is positive with respect to an ample class.
We define the extended rubbber invariants by
\[
\widetilde{\PT}^{S \times \p^1}_{n, (\gamma,d)}( \mu, \nu ) :=
\PT^{S \times \p^1}_{n, (\varphi(\gamma),d)}( \varphi(\mu), \varphi(\nu) ).
\]
By Conjecture~\ref{Conj_invariance} the definition is independent of the choice of $\varphi$.

We are now ready to define the invariants of the non-commutative CHL \eqref{non_geom_chl}. Let 
\[ \widetilde{\Lambda}^g = H^{\ast}(S,\BZ)^g, \quad \widetilde{\Lambda}_g = \big( \widetilde{\Lambda}^g \big)^{\perp} \subset H^{\ast}(S,\BZ) \]
be the invariant and coinvariant lattice, and let
\[ P = \frac{1}{N} \sum_{i=0}^{N-1} g_{\ast}^i \]
be the projection operator.
Let $\gamma \in P(H^{\ast}(S,\BZ))$
be a Hodge class. 
Let
\[ g_{\ast} : H^{\ast}( \Hilb^d S, \BZ) \to H^{\ast}( \Hilb^d S, \BZ). \]
be the action induced by $g_{\ast} : H^{\ast}(S,\BZ) \to H^{\ast}(S,\BZ)$.
Let
\[ \Gamma_{g_{\ast}} \in H^{\ast}( \Hilb^d S, \BZ)^{\otimes 2} \]
be its graph (where the Poincare duality is taken with respect to the Mukai pairing).
We define the Donaldson--Thomas invariant of the non-commutative CHL model \eqref{non_geom_chl} to be
\[ \DT^{\tilde{g}}_{n, (\gamma,d)}
:=
\frac{1}{N} \sum_{ \substack{ \tilde{\gamma} \in H^{\ast}(S,\BZ) \\ P(\tilde{\gamma}) = \gamma }}
\widetilde{\PT}^{S \times \p^1}_{n + d, (\tilde{\gamma},d)}( \Gamma_{g_{\ast}} ).
\]
If $g$ arises from a symplectic automorphism, then this definition specializes
by \eqref{deg_formula} to the Donaldson--Thomas invariant of the CHL model in class $(\gamma,d)$.
The relationship between this set of invariants and the Mathieu moonshine
conjecture will be pursued in future work.

\section{Modular forms} \label{section:modular_forms}
\subsection{Variables}
Let $\BH = \{ x + iy \in \BC \, | \, x,y \in \BR, y > 0\}$ be the upper half plane.
Consider variables $\tau \in \BH$ and $z \in \BC$, and let
\[ q = e^{2 \pi i \tau}, \quad p = e^{2 \pi i z}. \]
We make the following convention:
If a function $f(\tau, z)$ is invariant under $z \mapsto z+1$ and $\tau \mapsto \tau + 1$
we often write $f(q,p)$ instead of $f(\tau,z)$. Sometimes we will omit the argument $z$ or $p$.
Sometimes we will also omit $\tau$ or $q$. 
If an argument is modified it is always written out. For example, for the functions $f(\tau,z)$, $f(2 \tau, z)$ and $f(2 \tau, 2 z)$
we may write
\[ f = f(q) = f(q,p), \quad f(q^2) = f(q^2,p), \quad f(q^2, p^2) \]
respectively.

\subsection{Modular forms} \label{subsec:modular_forms}
A subgroup $\Gamma \subset \SL_2(\BZ)$ is a \emph{congruence subgroup}
if $\Gamma(N) \subset \Gamma$ for some $N \geq 1$,
where
\[ \Gamma(N) = \left\{ g \in \SL_2(\BZ)\, \middle|\, g \equiv \begin{pmatrix} 1&0\\0&1 \end{pmatrix} \text{ mod } N \right\}. \]
Let $\Mod_k(\Gamma)$ be the space of modular forms of weight $k$ for a congruence subgroup 
$\Gamma \subset \SL_2(\BZ)$ (more generally
$\Gamma \subset \GL_2^+(\BQ)$ is conjugate to a congruence subgroup) \cite{Koblitz}.
The algebra of modular forms is defined by
\[ \Mod(\Gamma) = \bigoplus_{k} \Mod_k(\Gamma). \]
If $\Gamma = \SL_2(\BZ)$ we often omit $\Gamma$ from the notation. Let also
\[ \Gamma_0(N) =
\left\{
\begin{pmatrix} a&b\\c&d \end{pmatrix} \, \middle| \, c \equiv 0 \text{ mod } N.
\right\}
\]

Define the weight $2k$ Eisenstein series
\[ E_{k}(\tau) = 1 - \frac{2k}{B_{k}} \sum_{m \geq 1} \sum_{d |m} d^{k-1} q^m, \quad k=2,4,6,\ldots \]
where $B_k$ are the Bernoulli numbers. If $k \geq 4$ the $E_k$ are modular forms for $\SL_2(\BZ)$.
For all $N \geq 2$ the functions defined by
\[
\CE_N(\tau) = \frac{1}{N-1} \left[NE_2(N\tau) - E_2(\tau)\right] = 1+O(q)
\]
and 
\[
\widetilde{\CE}_N(\tau) = \frac{N-1}{N} \CE _N
\]
are modular forms of weight $2$ for $\Gamma_0(N)$. For example,
\[ \CE_2(\tau) = \vartheta_{D4}(\tau) = 1 + 24 q + \ldots \]
is the theta function of the $D_4$ lattice. We have
\[ \Mod(\SL_2(\BZ)) = \BC[E_4, E_6], \quad \Mod(\Gamma_0(2)) = \BC[\vartheta_{D_4}, E_4]. \]
%

\subsection{Jacobi forms}
Let $\Gamma \subset \mathrm{SL}_2(\BZ)$ be a congruence subgroup (or congujate to one), and let
$\Jac_{k,m}(\Gamma)$ be the space of \emph{weak} Jacobi forms of weight $k$ and index $m \geq 0$
for the Jacobi group $\Gamma \rtimes \BZ^2$, see \cite{EZ}.
The $\Mod(\Gamma)$-algebra of weak Jacobi forms is
\[ \Jac(\Gamma) = \bigoplus_{k} \bigoplus_{m} \Jac_{k,m}(\Gamma). \]
The subspace of weak Jacobi forms of even weight is
\[ \Jac_{\text{even}}(\Gamma) = \bigoplus_{k \text{ even}} \bigoplus_{m} \Jac_{k,m}(\Gamma). \]

Define functions $K(\tau,z), \Theta(\tau,z), \wp(\tau,z)$ by
\begin{align*}
K(\tau,z) & = i \Theta(\tau, z) = (p^{1/2} - p^{-1/2}) \prod_{m \geq 1} \frac{ (1-pq^m) (1-p^{-1}q^m)}{ (1-q^m)^2 } \\
\wp(\tau,z) & = \frac{1}{12} + \frac{p}{(1-p)^2} + \sum_{d \geq 1} \sum_{k | d} k (p^k - 2 + p^{-k}) q^{d}. 
\end{align*}
With respect to the standard Jacobi theta functions $\theta_i(\tau,z)$ we have
\[ \Theta(\tau,z) = \frac{i \theta_1(\tau,z)}{\eta^3(\tau)}. \]
Define the weak Jacobi forms
\begin{align*}
\phi_{-2,1}(\tau,z) & = - K^2 = (-p^{-1} + 2 - p) + O(q) \\
\phi_{0,1}(\tau,z) & = 12 K^2 \wp = (p^{-1} + 10 + p) + O(q).
\end{align*}
In particular, the elliptic genus of a K3 surface is $2 \phi_{0,1}$. 
The algebra of weak Jacobi forms of even weight for group $\Gamma$ satisfies
\[ \Jac_{\text{even}}(\Gamma) \cong \BC[ \phi_{-2,1} , \phi_{0,1}] \otimes \Mod_{\ast}(\Gamma). \]
%

Every Jacobi form $F \in \Jac_{k,m}(\Gamma(N))$ has a Fourier expansion
\[
F(\tau, z) =
\sum_{b \in \{0,1, \ldots, 2m-1 \}} \sum_{\substack{n \in \BZ/N \\ j \in 2m \BZ+b}} c_b^{(r,s)}(4mn-j^2) q^n p^j.
\]

\subsection{Siegel modular forms}
Let $\BH_2$ be the Siegel upper half space. The
standard coordinates are
\[ Z =
\begin{pmatrix} \tau & z \\ z & \sigma \end{pmatrix} \in \BH_2\, ,
\]
where
$\tau, \sigma \in \BH$,
$z \in \BC$, and $\text{Im}(z)^2 < \text{Im}(\tau) \text{Im}(\sigma)$. Let
\[ q = e^{2 \pi i \tau}, \quad p = e^{2 \pi i z}, \quad t = e^{2 \pi i \sigma}. \]
The group $\Sp_4(\BR)$ acts on the Siegel space $\BH_2$ by
\[ gZ := (AZ+B) (CZ + D)^{-1}, \quad g = \begin{pmatrix}A&B \\C&D \end{pmatrix}. \] 
%
A Siegel modular form of weight $k$ for congruence subgroup $\Gamma \subset \Sp_4(\BZ)$
is a holomorphic function $f : \BH_2 \to \BC$ such that
\[ f(g Z) = \det(C Z + D)^k f(Z) \]
for all $g = \binom{A\ B}{C\ D} \in \Gamma$.
We let $\Mod_k^{(2)}(\Gamma)$ be the space of Siegel modular forms of weight $k$ for $\Gamma$.
The $\BC$-algebra of Siegel modular forms for $\Gamma$ is denoted by
\[ \Mod^{(2)}(\Gamma) = \bigoplus_{k} \Mod_k^{(2)}(\Gamma). \]

We will work with the several congruence subgroups in this paper.
For any $N \geq 1$ consider
\[ \Gamma_0^{(2)}(N) =
\left\{
\begin{pmatrix} A&B\\C&D \end{pmatrix} \, \middle| \, C \equiv 0 \text{ mod } N
\right\}.
\]
For any prime $p \geq 1$ define the paramodular subgroup \cite{IO, GH} (or rather a conjugate thereof)
\[ K(p) = \Sp(2,\BQ) \cap 
\begin{pmatrix}
\BZ & \BZ & p^{-1} \BZ& \BZ \\ 
p \BZ & \BZ & \BZ& \BZ \\ 
p \BZ & p \BZ & \BZ& p \BZ \\ 
p \BZ & \BZ & \BZ& \BZ
\end{pmatrix}.
\]
as well as the Iwahori subgroup
\[ B(p) = K(p) \cap \Gamma_0^{(2)}(p)
=
 \Sp(2,\BZ) \cap 
\begin{pmatrix}
\BZ & \BZ & \BZ& \BZ \\ 
p \BZ & \BZ & \BZ& \BZ \\ 
p \BZ & p \BZ & \BZ& p \BZ \\ 
p \BZ & p \BZ & \BZ& \BZ
\end{pmatrix}.
\]

\subsection{Examples of Siegel modular forms} \label{subsec:Examples_siegel_forms}
We discuss examples of Siegel modular forms for several congruence subgroups.

\subsubsection{The full group $\Sp_4(\BZ)$}
Consider the Fourier expansion of the elliptic genus of K3,
\[ 2 \phi_{0,1}(\tau,z) = \sum_{n \geq 0} \sum_{k \in \BZ} c(4n - k^2) p^k q^n. \]
The Igusa cusp form is a weight $10$ Siegel modular form for $\Sp_4(\BZ)$
which by a result of Gritsenko and Nikulin \cite{GN} can be defined by
\begin{equation} \chi_{10}(\Omega) = p q \tilde{q} \prod_{\substack{k \in \BZ \\ h,d \geq 0 \\ k<0 \text{ if } h=d=0}} ( 1 - p^k q^h \tilde{q}^d )^{c(4 h d - k^2)}. \label{Igusa} \end{equation}
Alternatively, $\chi_{10}$ is the additive lift of the Jacobi form
\[ -\phi_{10,1}(\tau,z) = -\phi_{-2,1} \Delta = \sum_{n \geq 0} \sum_{r \in \BZ} \mathsf{a}(n,r) p^k q^n \]
as in \cite[Sec.6]{EZ}, i.e.
\[ \chi_{10}(Z) = \sum_{(m,n,r)\neq 0} q^m t^n p^r \sum_{a|(m,n,r)} a^{9} \mathsf{a}\left( \frac{mn}{a^2}, \frac{r}{a} \right). \]
%

The second example we consider is the weight $4$ Eisenstein series 
\[ E_4^{(2)} = 1 + O(q,t,p). \]
We give two descriptions.
The first is as additive lift of the Jacobi form
\[ E_{4,1}(\tau,z) = K(\tau,z)^2 \left( E_4(\tau) \wp(\tau,z) - \frac{1}{12} E_6(\tau) \right) = \sum_{n,r} \mathsf{b}(n,r) q^n p^r. \]
We have
\[
E_4^{(2)}(Z)
=
1 + 
240 \sum_{(m,n,r) \neq 0} q^m t^n p^r
\sum_{a | (m,n,r)} a^{3} \mathsf{b}\left( \frac{mn}{a^2}, \frac{r}{a} \right).
\]
The second description is as the Siegel theta series of the $E_8$ lattice. 
Let $r_{E_8}(T)$ is the number of embeddings of $T = \binom{2m\ \phantom{r}r}{r\phantom{r}\ 2n}$ into $E_8$. Precisely,
\[ r_{E_8}(T) = \left| \left\{ (x,y) \in E_8^2 \,\middle|\, \langle x, x \rangle = 2m, \ \langle x,y \rangle = r, \ \langle y,y \rangle = 2n \right\} \right|. \]
Then by \cite[Sec.7]{EZ} we have
\[
E_4^{(2)}(Z) = \sum_{T = \binom{m\phantom{/2}\ r/2}{r/2\ \phantom{/2}n}} r_{E_8}(T) e( \mathrm{tr}\ TZ) \]

The forms $E_4^{(2)}$ and $\chi_{10}$ are generators of the ring of Siegel modular forms for $\Sp_4(\BZ)$.
By a result of Igusa \cite[v.d.G.,Thm.6]{123} we have
\[ \Mod^{(2)}_{\text{even}} = \BC[ E_4^{(2)}, E_6^{(2)}, \chi_{10}, \chi_{12} ] \]
where $E_6^{(2)}$ is an Eisenstein series and $\chi_{12}$ is a cusp form. 

\subsubsection{The group $\Gamma_0^{(2)}(2)$} \label{subsubsec:Gamma2}
Let $E_4^{(2)}(Z)$ be the Eisenstein series defined above.
The first Siegel modular form for $\Gamma_0^{(2)}(2)$ we consider is
\[ E_4^{(2)}(2Z). \]
[To see this is modular with respect to $\Gamma_0^{(2)}(2)$ we may argue as follows.
If $f(Z)$ is a modular form of weight $k$ for group $\Gamma$, then $f|_k g$ is a modular form
for group $g^{-1} \Gamma g$. Here 
\[ (f|_k g)(Z) = \det(C Z+D)^{-k} f(gZ) \]
with $g = (A,B;C,D)$ is the slash operator.
Apply this fact to $E_4^{(2)}(Z)$ and $g_0 = (2I_2, 0; 0, I_2)$ and use
\[ g_0^{-1} \Sp_4(\BZ) g_0 \cap \Sp_4(\BZ) = \Gamma_0^{(2)}(2).] \]

The second form for $\Gamma_0^{(2)}(2)$ is an additive lift.
Let $\psi \in \Jac_{k,1}(\Gamma_0(2))$ be Jacobi form of weight $k$ index $1$ and consider the Fourier expansion
\[ \psi = \sum_{n,r} \mathsf{c}(n,r) q^n p^r. \]
We define the $m$-th Hecke lift of $\psi$ as in \cite[App.A]{JS} by
\[
\psi|_k V_{m} = \sum_{n,r} q^n p^r \sum_{\substack{ a | (m,n,r) \\ a \text{ odd}}} a^{k-1} \mathsf{c}\left( \frac{mn}{a^2}, \frac{r}{a} \right).
\]
The function $\psi|_k V_{m}$ is a Jacobi form of weight $k$ and index $m$ for $\Gamma_0(2)$. 
The series
\[ \Psi = \sum_{m = 0}^{\infty} t^m (\psi|_k V_{m})(\tau,z) \]
defines a Siegel modular form for $\Gamma_0^{(2)}(2)$ (as in \cite[Sec.4]{EZ} we need to pick an appropriate re-normalization for the constant term of $\Psi$ here).

We apply this lifting construction to the weight $4$ Jacobi form
\[ G_{4,1}(\tau,z)
=
K(\tau,z)^2 \left( \wp(\tau) E_4(2 \tau) - \frac{1}{8} \theta_{D4}^3(\tau) + \frac{1}{24} E_4(\tau) \theta_{D4}(\tau) \right).
\]
Define its Fourier coefficients:
\[ G_{4,1}(\tau,z)
=
\sum_{n,r} \mathsf{c}_G(n,r) q^n p^r. \]
Then let $G_4(Z)$ be the Siegel modular form for $\Gamma^{(2)}_0(2)$ defined as the additive lift of $G_{4,1}$,
\[
G_4(Z) =
-\frac{7}{240} +
\sum_{0 \neq (m,n,r)}
q^m t^n p^r
\sum_{\substack{a | (m,n,r)\\ a \text{ odd}}}
a^3 \mathsf{c}_G \left( \frac{mn}{a^2}, \frac{r}{a} \right).
\]

We give a description of the algebra of modular forms for $\Gamma_0^{(2)}(2)$ and express
the function $E_4(2Z)$ and $G_4(Z)$ in terms of the standard generators.
For $m', m'' \in \BZ^2$ (considered as column vectors) and $m=\binom{m'}{m''}^T$ consider the genus 2 theta functions
\[
\theta_m(Z)
=
\sum_{x \in \BZ^2} e\left( \frac{1}{2} \left(x + \frac{1}{2} m'\right)^t Z \left(x + \frac{1}{2} m'\right)
+ \left(x + \frac{1}{2} m'\right)^t \frac{m''}{2} \right).
\]
where $e(z) = \exp(2 \pi i z)$ for all $z \in \BC$. Following \cite{AI} define
\begin{align*}
X & = \left( \theta_{0000}^4 + \theta_{0001}^4 + \theta_{0010}^4 + \theta_{0011}^4 \right)/4 \\
Y & = \left( \theta_{0000} \theta_{0001} \theta_{0010} \theta_{0011} \right)^2 \\
Z & = (\theta_{0100}^4 - \theta_{0110}^4)^2/16384 \\
W & = (\theta_{0100} \theta_{0110} \theta_{1000} \theta_{1001} \theta_{1100} \theta_{1111})^2 / 4096
\end{align*}
The functions $X,Y,Z,W$ are Siegel modular forms for $\Gamma_0^{(2)}(2)$ of weight $2,4,4,6$ respectively. Moreover,
\[ \Mod_{\text{even}}(\Gamma_0^{(2)}(2)) = \BC[ X,Y,Z,W]. \]
In these generators we have explicitly
\begin{align*}
\chi_{10}(Z) & = Y W \\
E_4(Z) & = 4 X^2 - 3Y + 12288 Z \\
E_4(2Z) & = \frac{1}{4} X^2 + \frac{3}{4} Y - 192 Z \\
G_4(Z) & = \frac{1}{120} X^2 - \frac{3}{80} Y - \frac{12}{5} Z.
\end{align*}

\subsubsection{The paramodular group $K(2)$} \label{subsubsec:paramodular}
Let $y \in E_8$ be a vector of length $\langle y, y \rangle = 4$ where we let $\langle -,- \rangle$ denote the pairing on $E_8$.
Consider the theta function
\[ \Theta_{E_8, y}
=
\sum_{x \in E_8} q^{\frac{1}{2} \langle x,x \rangle} p^{\langle x, y \rangle}.
\]
By \cite[Thm.7.1]{EZ} and since $E_8$ is unimodular the function $\Theta_{E_8, y}$
is a Jacobi form for $\SL_2(\BZ) \otimes \BZ^2$
of weight $4$ and index $2$.  
Concretely,
\[
\Theta_{E_8,y}
=
K^4 \left( \wp^2 E_4 - \frac{1}{6} \wp E_6 + \frac{1}{144} E_4^2 \right).
\]
As explained in \cite[Proof of Thm.2.1]{GH} the paramodular lift of $\Theta_{E_8,y}$ is
\[
F_4(Z)
=
\frac{1}{240} 
+
\sum_{\substack{0 \neq (m,n,r) \in \BZ_{\geq 0}^3 \\ m \text{ even}}} q^m t^n p^r
\sum_{a | \left( m/2, n, r \right)}
a^3 \Theta_{E_8,y}\left[ \frac{m/2 \cdot n}{a^2}, \frac{r}{a} \right].
\]
By \cite{IO} $F_4(Z)$ is the unique modular form for $K(2)$ of weight $4$,
\[ \Mod_4( K(2) ) = \BC F_4(Z). \]

We have the following alternative description. As in \cite{IO} let
\[ T = (\theta_{0100} \theta_{0110})^4/256. \]
Then we have
\[ F_4 = \frac{1}{960} ( X^2 + 3 Y + 3072 Z + 960 T ). \]

Because the matrix
\[
\begin{pmatrix}
0&1&0&0\\
1&0&0&0\\
0&0&0&1\\
0&0&1&0
\end{pmatrix}
\]
does not lie in $K(2)$, the function $F_4(Z) = F_4(q,t,p)$ does not have to be symmetric in $q$ and $t$ and in fact it is not.
For example,
the Fourier-Jacobi coefficients of $F_4$ in each direction have the form
\begin{align*}
\left[ F_4 \right]_{t^m}
& = K(\tau,z)^{2m} P_{m}\left(\wp(\tau,z), \wp(2 \tau,z), \vartheta_{D4}(\tau), E_4(\tau) \right) \\
\left[ F_4 \right]_{q^m}
& = K(\sigma,z)^{2m} Q_{m} \left(\wp(\sigma,z), \vartheta_{D4}(\sigma,z), E_4(\sigma) \right)
\end{align*}
where $P_m$ and $Q_m$ are polynomials of weight $2m+4$ (and the degree of $P_m$ in $\wp(2 \tau)$ is non-zero in general).

\section{Vertex computations}\label{sec: vertex}
\newcommand{\Ztfirstcoeff}{\left[  \mathsf{Z}^X(q,t,p) \right]_{t^{-1/N}}}
\newcommand{\Ztsecondcoeff}{\left[  \mathsf{Z}^X(q,t,p) \right]_{t^{0}}}

In this section we use the topological vertex method to compute  the first
two terms in the $t$ expansion of $\mathsf{Z}^X(q,t,p) $.
This proves the first part of Theorem~\ref{thm:intro1}, and Theorem~\ref{thm:intro2}.

\subsection{Preliminaries.
}

For any $\BC$-scheme $S$ of finite type, let $e(S)$ denote the
topological Euler characteristic of $S$, taken with the analytic
topology. More generally, if $\mu :S\to R$ is a constructible function
valued in a ring $R$, let
\[
e(S,\mu ) = \sum_{r\in R} r\cdot e(\mu^{-1}(r))
\]
be the $\mu$-weighted Euler characteristic.

We use the following standard facts:

\begin{itemize}
\item The Euler characteristic defines a ring homomorphism
\[ e:K_{0}(\operatorname{Var}_{\BC})\to \BZ, \]
i.e. it is additive under the
decomposition of a scheme into an open set and its complement, and it
is multiplicative on Cartesian products.
\item For any constructible morphism\footnote{A constructible morphism
is a map which is regular on each piece of a decomposition of its
domain into locally closed subsets.} $f:Y\to Z$ we have (see
\cite{MacPherson-Annals74})
\begin{equation}\label{eqn: e(Y,mu)=e(Z,f_*(mu))}
e(Y,\mu ) = e(Z,f_{*}\mu )
\end{equation}
where $f_{*}\mu $ is the constructible function given by 
\[
(f_{*}\mu )(x) = e(f^{-1}(x),\mu ).
\]
\item Let $g:\BZ_{\geq 0}\to \BZ$ be any function with $a(0)=1$. Let
\[
G_{d}:\Sym^{d}(Z)\to \BZ
\]
be the constructible function defined by
\[
G\left(\sum_{i}k_{i}z_{i} \right) = \prod_{i}g(k_{i}).
\]
Then (see
\cite[Lemma~32]{Bryan-Kool}) 
\begin{equation}\label{eqn: formula for euler char of multiplicative
constructible functions on symmetric products}
\sum_{d=0}e(\Sym^{d}(Z),G_{d})q^{d} = \left(\sum_{k=0}^{\infty}g(k)q^{k} \right)^{e(Z)}.
\end{equation}
\item If $\BC^{*}$ acts on a scheme $Y$ with fixed point locus
$Y^{\BC^{*}}\subset Y$, then (see \cite{Bialynicki-Birula})
\[
e(Y) = e(Y^{\BC^{*}}). 
\]
\item Let $G$ be an algebraic group acting on a scheme $Y$. Let $\mu$
be a $G$-invariant constructible function on $Y$. Suppose that each
$G$-orbit has zero Euler characteristic.
Then $e(Y,\mu )=0$. \cite{Banana}
\end{itemize}

For any $\BC$-scheme $S$, let $\nu_{S}:S\to \BZ$ denote the Behrend
function. We define the \emph{virtual Euler characterisitic} 
\[
e_{\vir}(S) = e(S,\nu_{S})
\]
to be the Behrend function weighted Euler characteristic.

The Behrend function depends on a scheme formally locally and
consequently, the virtual Euler characteristic is motivic in the
following sense. Let $Z\subset S$ be a closed subscheme, let
$U=S\setminus Z$, and let $\widehat{Z}$ be the formal neighborhood of $Z$
in $S$. Then
\begin{equation}\label{eqn: evir is motivic with formal nghds}
e_{\vir} (S) = e_{\vir}(U) + e_{\vir}(\widehat{Z}).
\end{equation}

We adopt the convention that replacing an index with a bullet denotes
a sum over the index multiplied by the appropriate variable raised to
the index. For example
\[
\Hilb^{\sigma +\bullet E' ,\bullet}(X ) = \sum_{d,n} \Hilb^{\sigma
+dE' ,n}(X )\, q^{d}\, y^{n}
\]
where we regard the right hand side as a formal power series in $q$,
Laurent in $y$, and whose coefficients are schemes.

With these conventions in hand, we may write
\begin{align*}
\Ztfirstcoeff &= \sum_{d,n} e_{\vir} \left(\Hilb^{\sigma +dE' ,n}(X ) /
E  \right)\, q^{d-1}  \, y^{n}\\
&= q^{-1} e_{\vir}\left(\Hilb^{\sigma +\bullet E' ,\bullet} (X )/E  \right)
\end{align*}
and
\begin{align*}
\Ztsecondcoeff &= \sum_{d,n} e_{\vir} \left(\Hilb^{\sigma +\frac{1}{N}F +dE' ,n}(X  \right) /
E )\, q^{d-1}  \, y^{n}\\
&= q^{-1} e_{\vir}\left(\Hilb^{\sigma +\frac{1}{N}F +\bullet E' ,\bullet} (X )/E  \right)
\end{align*}

\subsection{Decomposing Hilbert schemes via cycle support to compute $\Ztfirstcoeff$}

Subschemes of $X $ correspond to $\ZmodN{N}$ invariant subschemes
of $S\times E$, in particular we have
\[
\Hilb^{\sigma +dE' ,n}(X )\cong \Hilb^{\sigma_{0}+\dotsb
+\sigma_{N-1}+dE ,Nn}(S\times E)^{\ZmodN{N}} .
\]

The 1-cycle corresponding to any subscheme of $S\times E$ in the class $\sigma_{0}+\dotsb
+\sigma_{N-1}+dE $ must be of the form
\[
\sum_{i=0}^{N-1} \sigma_{i}\times \{x_{i} \} \, + \, \{y_{1}+\dotsb +y_{d}
\} \times E
\]
where $x_{i}\in E$ and $y_{1}+\dotsb +y_{d}$ is a length $d$ 0-cycle
on $S$. Consequently, such subschemes are uniquely determined by their
restrictions to the subschemes
\[
U\times E, \quad \widehat{\sigma}_{1}\times E,\quad \dotsc \quad
,\widehat{\sigma}_{N-1}\times E
\]
where $\widehat{\sigma}_{i}$ is the formal neighborhood of $\sigma_{i}$ in
$S$ and $U$ is the complement of the union of the sections. 

This leads to the following decomposition:
\[
\Hilb^{\sigma_{0}+\dotsb +\sigma_{N-1}+\bullet E ,N\bullet}(S\times
E)\, \cong \,\Hilb^{\bullet E ,N\bullet}(U\times E) \prod_{i=0}^{N-1}
\Hilb^{\sigma_{i}+\bullet E ,\bullet}(\widehat {\sigma}_{i} \times E)
\]
which should be understood as giving constructible
isomorphisms\footnote{A constructable isomorphism is a bijective map
which is regular on some decomposition of the domain into locally
closed subsets.} among the coefficients and consequently equality of
the coefficients in the Grothendieck group of varieties. 

Since the $\ZmodN{N}$ action on the Hilbert scheme permutes the last
$N$ factors in the above decomposition, we get an isomorphism
\[
\Hilb ^{\sigma_{0}+\dotsb +\sigma_{N-1}+\bullet E ,N\bullet}(S\times
E)^{\ZmodN{N}}\, \cong\,\Hilb^{\bullet E ,N\bullet}(U\times E)^{\ZmodN{N}}
\times \Hilb^{\sigma_{0}+\bullet E ,\bullet}(\widehat {\sigma}_{0} \times E).
\]
Moreover, we may fix a slice for the $E$ action on the Hilbert scheme
by defining 
\[
\Hilb_{\fix}^{\sigma_{0}+\bullet
E ,\bullet}(\widehat{\sigma}_{0}\times E) \subset
\Hilb^{\sigma_{0}+\bullet E ,\bullet}(\widehat{\sigma}_{0}\times E)
\]
to be the locus of subschemes containing $\sigma_{0}\times \{x_{0} \}$
as a component, where $x_{0}\in E$ is the origin. Combining this with
the previous discussion we arrive at
\[
\Hilb^{\sigma +\bullet E' ,\bullet} (X )/E \cong \Hilb^{\bullet
E ,N\bullet}(U\times E)^{\ZmodN{N}} \times \Hilb_{\fix}^{\sigma_{0}
+\bullet E ,\bullet} (\widehat{\sigma}_{0}\times E) 
\]
which again is understood as giving constructible isomorphisms of the
coefficients. 

To compute the DT partition function $\Ztfirstcoeff$, we need to
take the virtual Euler characteristic of $\Hilb^{\sigma +\bullet
E' ,\bullet} (X )/E$. Recall that for the virtual Euler
characteristic to respect the above decomposition, we must retain the
formal neighborhood of each stratum in the decomposition (see
equation~\eqref{eqn: evir is motivic with formal nghds}). The result
is
\[
q\Ztfirstcoeff  = - e_{\vir}(\Hilb^{\bullet E ,N\bullet} (U\times
E)^{\ZmodN{N}} )\cdot
e_{\vir}\left(\widehat{\Hilb}_{\fix}^{\sigma_{0}+\bullet
E ,\bullet}(\widehat{\sigma}_{0}\times E) \right) 
\]
where $\widehat{\Hilb}_{\fix}(\widehat{\sigma}_{0}\times E)$ is the
formal neighborhood of $\Hilb_{\fix}(\widehat{\sigma}_{0}\times E)$ in
$\Hilb (S\times E )$. The overal minus sign arises as the difference
between the Behrend function of $\Hilb (S\times E)$ and  $\Hilb
(S\times E)/E$.  

Only the fixed points of the $E$ action on $\Hilb (U\times
E)^{\ZmodN{N}}$ contribute to the virtual Euler characteristic. These
fixed points correspond to $E$-invariant subschemes. Such subschemes
cannot have zero dimensional components and are determined by their
intersection with $U\times \{x_{0} \}$. Thus
\[
e_{\vir}\left(\Hilb^{\bullet E ,N\bullet}(U\times E)^{\ZmodN{N}}
\right) = e_{\vir}\left(\Hilb^{\bullet}(U)^{\ZmodN{N}} \right). 
\]
The Hilbert scheme of points on $S$ (and hence on $U$) is a smooth
holomorphic symplectic variety and consequently, so are the fixed
points of the $\ZmodN{N}$ action. The Behrend function on smooth even
dimensional varieties is 1 and so the virtual Euler characteristic and
the usual Euler characteristic coincide. Moreover, the $\ZmodN{N}$
fixed locus can be identified with the Hilbert scheme of substacks of
the stack quotient $[U/\ZmodN{N}]$. Thus
\[
e_{\vir}\left(\Hilb^{\bullet}(U)^{\ZmodN{N}} \right) =
e\left(\Hilb^{\bullet}(U)^{\ZmodN{N}} \right) =
e\left(\Hilb^{\bullet}([U/\ZmodN{N}]) \right).
\]

The usual motivic methods for computing the generating function of
Hilbert schemes of points on smooth orbifold surfaces work also for
orbifolds given by quotients by cyclic groups.
\begin{lemma}\label{lem: formula for e(Hilb(Y/G))}
Let $Y$ be a smooth surface with a $\ZmodN{N}$ action. For $d|N$ let
$F_{d}\subset Y$ be the locus of points whose $\ZmodN{N}$ stabilizer
has order $d$, and let $e_{d} = e(F_{d}/\ZmodN{N/d})$. Then
\[
\sum_{n=0}^{\infty} e(\Hilb^{n}([Y/\ZmodN{N}])\, q^{n} = \prod_{d|N}
\prod_{k=1}^{\infty} (1-q^{\frac{N}{d}k})^{-e_{d}}. 
\]
\end{lemma}
\begin{proof}
The standard method of computing the Euler characteristic of the Hilbert scheme of points in terms of the punctual Hilbert scheme applies in this setting.
The punctual Hilbert scheme at a point in the quotient stack is easily expressed in terms of punctual Hilbert schemes of the etal\'e cover. 
\end{proof}

For  $S$ a $K3$ surface with a symplectic $\ZmodN{N}$ action, the
numbers $e_{d}$ are given in the following table:

\smallskip

\begin{center}
\begin{tabular}{|c|c|}
\hline 
$N$& $e_{d}$\\ \hline 
1& $e_{1}=24$\\ \hline 
2& $e_{1}=8$, $e_{2}=8$ \\ \hline 
3& $e_{1}=6$, $e_{3}=6$ \\ \hline 
4& $e_{1}=4$, $e_{2}=2$, $e_{4}=4$, \\ \hline 
5& $e_{1}=4$, $e_{5}=4$ \\ \hline 
6& $e_{1}=e_{2}=e_{3}=e_{6}=2$ \\ \hline 
7& $e_{1}=3$, $e_{7}=3$ \\ \hline 
8& $e_{1}=e_{8}=2$, $e_{2}=e_{4}=1$ \\ \hline 
\end{tabular}
\end{center}

\smallskip

By inspection, we see that 
\[
\sum_{n=0}^{\infty}e(\Hilb^{n}([S/\ZmodN{N}]) q^{n}  = q \Delta_{N}^{-1}
\]
where $\Delta_{N}$ is the modular form given in table~\ref{Table_Delta_N}. Since
\[
[S/\ZmodN{N}] = [U/\ZmodN{N}] \cup \BP^{1}
\]
we see that the values of $e_{d}$ for $U$ are the same as those for
$S$ except $e_{1}$ is smaller by 2. Thus we find that
\[
e(\Hilb^{\bullet}([U/\ZmodN{N}])) = q\Delta_{N}(q)^{-1}
\prod_{k=1}^{\infty } (1-q^{Nk})^{2}
\]
and so
\[
\Ztfirstcoeff= -\Delta_{N} (q)^{-1}
\prod_{k=1}^{\infty } (1-q^{Nk})^{2} \cdot 
e_{\vir}\left(\widehat{\Hilb}_{\fix}^{\sigma_{0}+\bullet
E ,\bullet}(\widehat{\sigma}_{0}\times E) \right) 
\]
The last factor in the above product is independent of $N$
and hence can be evaluated by specializing to $N = 1$
and using that we know the left hand side by the proof of the
the Igusa conjecture
\cite{ObPix2}.\footnote{
The extra factor of $N$ results
in the substitution $q\mapsto q^{N}$.
The last factor was also computed earlier
(modulo a conjecture on the Behrend function) in \cite[eqns~(4),(5),Lemma~2]{Bryan-K3xE}.}
The result is
\begin{align*}
e_{\vir}\left(\widehat{\Hilb}_{\fix}^{\sigma_{0}+\bullet
E ,\bullet}(\widehat{\sigma}_{0}\times E) \right)  =&
\frac{-p}{(1-p)^{2}}\prod_{k=1}^{\infty}
\frac{(1-q^{Nk})^{2}}{(1-pq^{Nk})^{2}(1-p^{-1}q^{Nk})^{2}} \\
=&\frac{-1}{\Theta(q^{N},p)^{2}\prod_{k=1}^{\infty}(1-q^{Nk})^{2}} 
\end{align*}
and consequently we have completed the proof that
\[
\Ztfirstcoeff = \frac{1}{\Theta(q^{N},p)^{2}\Delta_{N}(q)}.
\]

\smallskip

\subsection{Decomposing Hilbert schemes via cycle support to compute $\Ztsecondcoeff$.} \label{sec:z0}

$\quad$
\smallskip
\smallskip

We first observe that if an automorphism\footnote{Here $ \operatorname{Aut}(H^{2}(S,\BZ
))^{+}\subset \operatorname{Aut}(H^{2}(S,\BZ ))$ is the index 2
subgroup generated by reflections through $-2$-vectors.} $\phi \in \operatorname{Aut}(H^{2}(S,\BZ
))^{+}$ commutes
with the $\ZmodN{N}$ action, it is realized by a
$\ZmodN{N}$-equivariant monodromy deformation of $X$
\cite[Chapt 7., Prop.~5.5]{Huy} and hence gives equality of Donaldson-Thomas
invariants:
\[
\DT^{X }_{n,\beta +dE' } = \DT^{X }_{n,\phi (\beta) +dE' } .
\]
Let $s_{\sigma_{i}}(v) = v+\left\langle \sigma_{i},v \right\rangle
\sigma_{i}$ be the reflection about $\sigma_{i}$. Then $\sigma_{i}$
commutes with $\sigma_{j}$ and consequently the automorphism
\[
\phi  = s_{\sigma_{0}}\circ \dotsb \circ s_{\sigma_{N-1}}
\]
commutes with the action $\ZmodN{N}$.
Since 
\[ \phi (\sigma
+\frac{1}{N}F) = \frac{1}{N}F \]
we therefore have 
\[
\Ztsecondcoeff = q^{-1} e_{\vir}(\Hilb^{\frac{1}{N}F +\bullet
E',\bullet}(X )/E). 
\]

Second, by writing the elliptic fibration $S \to \p^1$ as a Weierstra{\ss} model
and deforming the coefficients,
we may assume that the fibration $S \to \p^1$ is generic among all elliptic fibrations with an $N$-torsion section.
In particular all reducible fibers are of type $I_{n}$.
We also take the elliptic curve $E$ to be generic.

Because any subscheme in $X$ in the class
$\frac{1}{N}F+dE'$ corresponds uniquely to a $\ZmodN{N}$-invariant
subscheme $Z\subset S\times E$ in the class $F+dE$ we may write
\[
\Hilb^{\frac{1}{N}F +\bullet E',\bullet}(X  ) = \Hilb^{F +\bullet
E,N\bullet}(S\times E)^{\ZmodN{N}} .
\]

\begin{defn}\label{defn: horizontal, vertical, and diagonal curves}
We say an irreducible curve component $C\subset Z$ is
\emph{horizontal} if $\pi_{S}(C)$ is zero dimensional, \emph{vertical}
if $\pi_{E}(C)$ is zero dimensional, and \emph{diagonal} otherwise. 
\end{defn}

\begin{lemma}\label{lem: horizontal and diagonal curves and slice
conditions for Hilb(SxE)^{ZN}} Assume that $N>1$. Let $Z\subset S\times E$ be a
$\ZmodN{N}$-invariant subscheme corresponding to a point $[Z]\in
\Hilb^{F+\bullet E,N\bullet}(S\times E)^{\ZmodN{N}}$. Then the curve
support of $Z$ must be a union of horizontal components along with either
\begin{enumerate}
\item exactly $N$ vertical components given by 
\[
\bigcup_{g\in \ZmodN{N}} g(C\times x)
\]
where $C$ is an irreducible component of an $I_{N}$ fiber and $x\in
E$, or
\item a single diagonal component contained in $F_{t}\times E$ and
given by the graph of a map $f :F_{t}\to E$ where $F_{t}\subset S$
is a smooth fiber. 
\end{enumerate}
Moreover, we may define a slice for the $E$ action on
$\Hilb^{F+\bullet E,\bullet}(S\times E)^{\ZmodN{N}}$ by imposing that
\begin{enumerate}
\item [$(1)'$] the curve $C\times x$ is given by $C_{0}\times x_{0}$ where
$C_{0}$ is the component meeting the zero section and $x_{0}\in E$ is
the origin, or
\item  [$(2)'$] the diagonal curve in $F_{t}\times E$ contains the point
$\sigma_{0}(t)\times x_{0}$ so that the map $f  :F_{t}\to E$ is
either a homomorphism and anti-homomorphism, i.e. $\pm f 
(y_{1}+y_{2})=f  (y_{1})+f  (y_{2})$.  
\end{enumerate}
\end{lemma}

The lemma gives us the decomposition
\[
\Hilb^{F+\bullet E,N\bullet} (S\times E)^{\ZmodN{N}}/E\cong
\,\,\Hilb^{F+\bullet E,N\bullet}_{\vert } \,\, +  \,\, \Hilb^{F+\bullet E,N\bullet}_{\diag}
\]
where
\[
\,\,\Hilb^{F+dE,Nn}_{\vert } \,\, ,  \,\, \Hilb^{F+d E,Nn}_{\diag}  \,\,
\subset  \,\, \Hilb^{F+dE,Nn}(S\times E)^{\ZmodN{N}} 
\]
parameterize subschemes of type (1) and (2) respectively satisfying
the slice conditions $(1)'$ and $(2)'$ respectively (see figure~\ref{fig: vert and diag curves}).

\begin{rmk}
In the case of $N=1$, the lemma holds as stated with the additional caveat that in case (1), the vertical component can be any smooth fiber. 
\end{rmk}

\begin{figure}
\begin{tikzpicture}[
                    z  = {-15},
                    scale = 0.85]

\begin{scope}[yslant=-0.35,xslant=0]


\begin{scope} [canvas is yz plane at x=0]
\draw [black](0,0) rectangle (3,12);
\end{scope}
\begin{scope} [canvas is xz plane at y=0]
\draw [black](0,0) rectangle (4,12);
\end{scope}
\draw [black](0,0) rectangle (4,3);

\draw [ultra thick,orange,opacity=1.0] 
(0.5  ,0   ,12)--(0.5,2,12);
\draw [ultra thick,orange,opacity=0.2] 
(0.5  ,0   ,8)--(0.5,2,8);
\draw [ultra thick,orange,opacity=0.2] 
(0.5  ,0   ,4)--(0.5,2,4);
\draw [ultra thick,orange,opacity=0.2] 
(0.1  ,1.7   ,12)--(1.5,0.8,12);
\draw [ultra thick,orange,opacity=0.2] 
(0.1  ,1.7   ,8)--(1.5,0.8,8);
\draw [ultra thick,orange,opacity=1.0] 
(0.1  ,1.7   ,4)--(1.5,0.8,4);
\draw [ultra thick,orange,opacity=0.2] 
(0.1  ,0.3   ,12)--(1.5,1.2,12);
\draw [ultra thick,orange,opacity=1.0] 
(0.1  ,0.3   ,8)--(1.5,1.2,8);
\draw [ultra thick,orange,opacity=0.2] 
(0.1  ,0.3   ,4)--(1.5,1.2,4);

\node [below] at (2,0,13) {$S$};
\node [right] at (5,0,7) {$E$};
\node [right] at (4.2,0,12) {$x_{0}$};
\node [right] at (4.2,0,8) {$gx_{0}$};
\node [right] at (4.2,0,4) {$g^{2}x_{0}$};

\begin{scope} [canvas is yz plane at x=4]
\draw [black](0,0) rectangle (3,12);

\draw [pink, ultra thick] (3,0)--(2,12);
\draw [pink, ultra thick] (2,0)--(1,12);
\draw [pink, ultra thick] (1,0)--(0,12);
\end{scope}

\begin{scope} [canvas is xz plane at y=3]
\draw [black](0,0) rectangle (4,12);
\end{scope}

\draw [black](0,0,12) rectangle (4,3,12);
\draw [black,fill, opacity=0.1](0,0,12) rectangle (4,3,12);
\draw [black](0,0,8) rectangle (4,3,8);
\draw [black,fill, opacity=0.1](0,0,8) rectangle (4,3,8);
\draw [black](0,0,4) rectangle (4,3,4);
\draw [black,fill, opacity=0.1](0,0,4) rectangle (4,3,4);

\node [above] at(0.6,2,12) {$\scriptscriptstyle{C_{0}\times x_{0}}$};

\end{scope}

\end{tikzpicture} \caption{ $\ZmodN{3}$-invariant curves in
$S\times E$ satisfying the slice condition of Lemma~\ref{lem:
horizontal and diagonal curves and slice conditions for
Hilb(SxE)^{ZN}}. The dark orange curves are the orbit of $C_{0}\times
x_{0}$ giving an invariant vertical configuration, while the pink
curve is an invariant diagonal curve.  Horizontal curves are also possible, but not shown. }\label{fig: vert and diag
curves}
\end{figure}
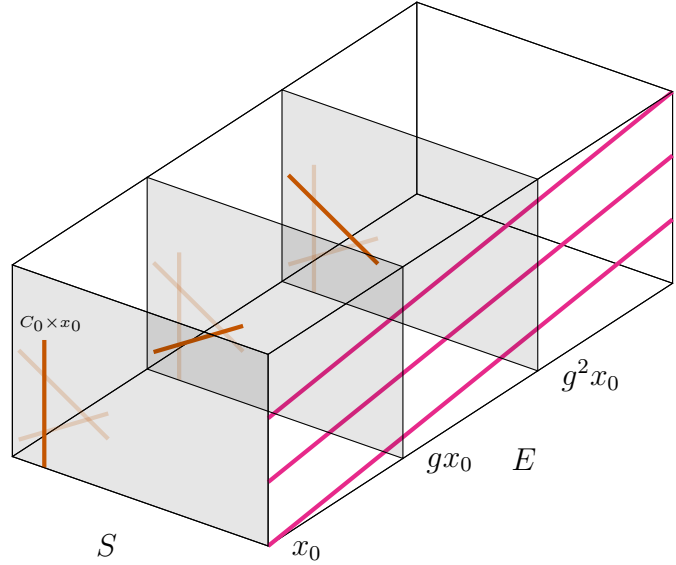

\begin{proof}
Any $\ZmodN{N}$-invariant vertical curve must project to a
$\ZmodN{N}$ orbit in $E$ and thus must be a union of orbits
$\cup_{g\in \ZmodN{N}}g(C\times x)$. In the case of the lemma,
$\sum_{g\in \ZmodN{N}}g(C)$ is in the class $F$ and so $C$ must be a
single component of an $I_{n}$ fiber (as asserted by (1)). Moreover,
since $\{g(C) \}_{g\in \ZmodN{N}}$ are all the components of an $I_{N}$
fiber, we may assume that $C=C_{0}$, the component meeting the zero
section. By $E$ translation, we may then assume that $x=x_{0}$ and
this is unique in the $E$ orbit (as asserted by $(1)'$).

Now let $C\subset S\times E$ be a $\ZmodN{N}$-invariant diagonal curve
with $\pi_{S}(C)$ in the class $F$. Such a curve cannot project to a
singular fiber since it would then give rise to a non-constant map
from a rational curve to $E$. Thus $\pi_{S}$ maps $C$ isomorphically
onto a smooth fiber $F_{t}$ and hence is the graph of a map $f 
:F_{t}\to E$ (as asserted by (2)). Moreover, by a translation by $E$,
we may assume that the map $f $ takes the origin $\sigma_{0}(t)\in
F_{t}$ to $x_{0}\in E$. Any such map is the composition of a group
homomorphism and an automorphism. Since $E$ is generic, the only
automorphisms are $\pm 1$. An so $f $ is a homomorphism or an
antihomomorphism (as asserted by $(2)'$).

Any invariant curve projecting to the class of $F$ must have some
diagonal or vertical component and by the above arguments, it can have
only one or the other.
\end{proof}

\subsection{Diagonal contributions. }

The possible diagonal curves are enumerated by the following
\begin{lemma}\label{lem: formula for number of diagonal curves}
Let $\delta (d)$ be the number of $\ZmodN{N}$-invariant diagonal curves in
$S\times E$ upto translation by $E$. Then 
\[
\sum_{d=1}^{\infty} \delta (d)q^{d} = \frac{-2}{\phi_2(N)}
\sum_{m|N}(E_{2}(q^{m})-1) \mu (m)
\]
where $\mu (m)$ is the M\"obius function, $E_{2}$ is the Eisenstein
series,  and $\phi_2(N)$
is the number of $N$-torsion points in an elliptic curve (c.f. Theorem~\ref{thm:intro2}). 
\end{lemma}

\begin{proof}\footnote{We warmly thank Greg Martin for assistance with
this proof.}
By Lemma~\ref{lem: horizontal and diagonal curves and slice
conditions for Hilb(SxE)^{ZN}} $(2)'$, we need to count maps 
\[
f  : F_{t}\to E
\]
of degree $d$ where $F_{t}\subset S$ is the fiber over some point
$t\in \BP^{1}$ and the map $f $ is a homomorphism or
anti-homomorphism such that the graph in $F_{t}\times E$ is
$\ZmodN{N}$-invariant. The graph of $f $ is invariant if and only if
$f $ is $\ZmodN{N}$-equivariant, i.e. the $N$-torsion point
$s=\sigma_{1}(t)\in F_{t}$ is mapped to the fixed $N$-torsion point
$e_{1}\in E$.

We first count all degree $d$ homomorphisms $f  :F_{t}\to E$ without
imposing the additional condition $f  (s)=e_{1}$. It is well known
that there are exactly
\[
\sigma (d) = \sum_{k|d} k 
\]
elliptic curves $F$ admitting a degree $d$ homomorphism to $E$ (unique
for $E$ generic), and since every such curve appears exactly 24 times
in the $K3$ elliptic fibration, the number of such homomorphisms
$F_{t}\to E$ is $ 24\sigma (d)$.

We can refine the count of homomorphisms $f  :F_{t}\to E$ as follows:
\begin{align*}
D(d,k,N)& = \# \left\{f  :F_{t}\to E \text{ such that $f  (s)$ has
order dividing $k$.} \right\}\\
E(d,k,N)& = \# \left\{f  :F_{t}\to E \text{ such that $f  (s)$ has
order exactly $k$.} \right\}
\end{align*}

It will suffice to compute $E(d,N,N)$ since if $\phi_2(N)$ is the number of
$N$-torsion points on $E$, then 
\[
\phi_2(N) \delta (d) = 2E(d,N,N)
\]
since the left hand side counts all homomorphisms and
anti-homomorphisms where $f  (s)$ has order exactly $N$.

We apply M\"obius inversion to
\[
D(d,k,N) = \sum_{b|k} E(d,b,N)
\]
to get
\[
E(d,k,N) = \sum_{b|k}D(d,b,N)\mu (k/b).
\]
If the order of $f  (s)$ divides $k$, then $kf  (s) = f 
(ks)=x_{0}$ and so $f  :F_{t}\to E$ factors through the map
$F_{t}\to F_{t}/\left\langle ks \right\rangle$ which is a map of
degree $N/k$. Therefore
\begin{align*}
D(d,k,N)& = \# \left\{\widetilde{f } :F_{t}/\left\langle ks
\right\rangle\to E \text{ such that the degree of $\widetilde{f  }$ is  $kd/N$} \right\}\\
&=\begin{cases}
24\sigma (kd/N)& \text{if $N|kd$,}\\
0&\text{otherwise.}
\end{cases}
\end{align*}
Hence we find 
\begin{align*}
E(d,N,N)& = \sum_{b|N} D(d,b,N)\mu (N/b)\\
&= \sum_{\substack{b|N\\N|bd}} 24 \sigma (bd/N)\mu (N/b)\\
&= \sum_{\substack{m|N\\m|d}} 24\sigma (d/m)\mu (m).
\end{align*}
Therefore we get
\begin{align*}
\sum_{d=1}^{\infty} \delta(d) q^{d} &= \frac{48}{\phi_2(N)}
\sum_{d=1}^{\infty }\sum_{\substack{m|N\\m|d}} \sigma (d/m) (q^{\frac{d}{m}})^{m}\mu (m)\\
&=  \frac{48}{\phi_(N)} \sum_{m|N} \sum_{k=1}^{\infty} \sigma (k) q^{mk} \mu (m)\\
&=  \frac{-2}{\phi_2(N)} \sum_{m|N} \left(E_{2}(q^{m})-1 \right) \mu (m). 
\end{align*}
\end{proof}

We compute the full contribution of the diagonal components to the the
DT partition function:
\begin{lemma}\label{lem: evir(Hilb_{diag})}
\[
e_{\vir}\left(\Hilb_{\diag}^{F+\bullet E,N\bullet} \right) =
\frac{-2q}{\phi_{2}(N)\Delta_{N}(q)} \sum_{m|N} (E_{2}(q^{m})-1)\mu (m).
\]
\end{lemma}
\begin{proof}
Let $C\subset F_{t}\times E$ be a $\ZmodN{N}$-invariant diagonal curve
satisfying the slice condition and let 
\[
\Hilb_{\diag ,C}^{F+\bullet E,N\bullet} \subset \Hilb_{\diag}^{F+\bullet E,N\bullet} 
\]
be the component parameterizing subschemes containing $C$. Such
subschemes are a union of $C$, horizontal components, and zero
dimensional components. Consequently, such a subscheme is a disjoint
union of components supported on $\widehat{F}_{t}\times E$ and on
$U\times E$ where $\widehat{F}_{t}$ is the
formal neighborhood of $F_{t} $ inside of $S$ and $U=S \setminus F_{t}$. Thus
\[
\Hilb_{\diag ,C}^{F+\bullet E,N\bullet} = \Hilb_{\diag,C}^{F+\bullet
E,N\bullet}(\widehat{F}_{t}\times E) \cdot \Hilb^{\bullet
E,N\bullet }(U\times E)^{\ZmodN{N}}  
\]
where the first factor is the Hilbert scheme parameterizing
$\ZmodN{N}$-invariant subschemes $Z\subset S\times E$ whose support is
contained in $\widehat{F}_{t}\times E$ and which contains the diagonal
curve $C$.

As in the previous subsection, the $E$ action on $\Hilb^{\bullet
E,N\bullet }(U\times E)^{\ZmodN{N}}$ reduces the Euler computation to
$E$-invariant subschemes, which necessarily pullback from
zero-dimensional subschemes of the stack $[U/\ZmodN{N}]$:
\begin{align*}
e_{\vir} \left(\Hilb^{\bullet E,N\bullet }(U\times E)^{\ZmodN{N}}
\right) & =
\sum_{n=0}^{\infty} e(\Hilb^{n}([U/\ZmodN{n}])) \, q^{n}\\
&= \sum_{n=0}^{\infty} e(\Hilb^{n}([S/\ZmodN{n}])) \, q^{n}\\
&=\frac{q}{\Delta_{N}(q)}. 
\end{align*}
Here we have used the fact that $e(S\setminus U) = e(F_{t})=0$.

We construct a group action on $ \Hilb_{\diag,C}^{F+\bullet
E,N\bullet}(\widehat{F}_{t}\times E)$ as follows. Let $\Delta_{t}$ be
the formal neighborhood of $t\in \BP^{1}$ and let $\widehat{F}_{t}\to
\Delta_{t}$ be the restriction of $S\to \BP^{1}$. Let $MW$ denote the
Mordell-Weil group of sections of $\widehat{F}_{t}\to
\Delta_{t}$. $MW$ acts on $\widehat{F}_{t}$ by translation and we can
define an action of $MW$ on $\widehat{F}_{t}\times E$ which preserves
the curve $C$ by composing with the appropriate action of $E$. In
particular, on closed points the action of $\sigma \in MW$ is given by
\[
\sigma (y,x) = (y+\sigma ,x+f(\sigma ))
\]
where $C$ is given by the graph of $f:F_{t}\to E$. This action induces
an action of $MW$ on $ \Hilb_{\diag,C}^{F+\bullet
E,N\bullet}(\widehat{F}_{t}\times E)$. The group $MW$ is a
pro-algebraic group whose action on $ \Hilb_{\diag,C}^{F+d
E,n}(\widehat{F}_{t}\times E)$ factors through an algebraic group. The
orbits of the $MW$ action on $\Hilb_{\diag,C}^{F+\bullet
E,N\bullet}(\widehat{F}_{t}\times E)$ have zero Euler characteristic
unless they are fixed points. Moreover, $MW$ preserves the Behrend
function since the action on the strata $\Hilb_{\diag,C}^{F+\bullet
E,N\bullet}(\widehat{F}_{t}\times E) $ extends to an action on the
formal neighborhood of this strata in the whole Hilbert
scheme\footnote{See \cite{Banana} for a careful discussion of the
action of Mordell-Weil groups on Hilbert schemes.}. But the only
$MW$-invariant subscheme is $C$ itself since no horizonal component or
zero-dimensional component is $MW$ invariant. Therefore we have
\begin{align*}
e_{\vir}(\Hilb_{\diag,C}^{F+\bullet
E,N\bullet}(\widehat{F}_{t}\times E))& =e_{\vir}(\Hilb_{\diag,C}^{F+\bullet
E,N\bullet}(\widehat{F}_{t}\times E)^{MW} )\\
& = e_{\vir}(pt)\\
&=1
\end{align*}
Therefore
\[
e_{\vir}(\Hilb_{\diag}^{F+\bullet
E,N\bullet}) = \left(\sum_{d=1}^{\infty}\delta (d)\, q^{d} \right) \frac{q}{\Delta_{N}(q)}
\]
and then lemma~\ref{lem: formula for number of diagonal curves}
completes the proof of Lemma~\ref{lem: evir(Hilb_{diag})}.
\end{proof}

\subsection{Vertical contributions.}
The contribution of subschemes containing vertical components to the
DT partition function requires a new vertex computation which is carried
out in this subsection.

\begin{prop}\label{prop: evir(Hilb_{vert})}
Assuming \cite[Conj.~21]{Bryan-Kool}, we have
\[
e_{\vir}\left(\Hilb_{\vert}^{F+\bullet E,N\bullet} \right) =
-\frac{q}{\Delta_{N}(q)}\cdot 24 \frac{\phi_{1}(N)}{\phi_{2}(N)}\cdot 
\left\{\wp (q^{N}) - \frac{1}{12} E_{2}(q^{N})+\frac{1}{12}\delta_{1,N} \right\}
\]
where $\delta_{1,N}$ is $1$ if $N=1$ and $0$ otherwise.
\end{prop}

To prove Proposition~\ref{prop: evir(Hilb_{vert})}, we begin by observing
that $24\frac{\phi_{1}(N)}{\phi_{2}(N)}$ is the number of $I_{N}$
fibers in $S$. Let $C = C_{0}\cup \dotsb \cup C_{N-1}$ be a fixed
$I_{N}$ fiber and let 
\[
\Hilb_{\vert,C}^{F+\bullet E,N\bullet} \subset \Hilb_{\vert}^{F+\bullet E,N\bullet}
\]
be the component parameterizing curves whose vertical components
project to $C$. To prove the proposition, it then suffices to prove\footnote{Our Lemma~\ref{lem: horizontal and diagonal curves and slice
conditions for Hilb(SxE)^{ZN}}, which asserts that the vertical components are supported on the $I_N$ fibers, applies only for $N>1$. As previously remarked, vertical curves in the $N=1$ case can also occur at smooth fibers. The contribution from subschemes containing vertical curves on smooth fibers is given by 
\[\frac{24q}{\Delta(q)}\cdot e(\BP^1-\{24pts\})=\frac{24q}{\Delta(q)}\left(\frac{1}{12}-1\right).\]
This accounts for the replacement of $\frac{1}{12}\delta_{1,N}$ with $\delta_{1,N}$. See the computation in \cite{Bryan-K3xE} for details.
}
\[
e_{\vir}\left(\Hilb_{\vert,C}^{F+\bullet E,N\bullet} \right) =
-\frac{q}{\Delta_{N}(q)}\cdot 
\left\{\wp (q^{N}) - \frac{1}{12} E_{2}(q^{N})+\delta_{1,N} \right\}
\]

Let $\widehat{C}$ be the formal neighborhood of $C$ in $S$ and let
$U=S\setminus C$. 

Then, as in the previous cases (using similar notation), we have
\begin{multline*}
e_{\vir}\left(\Hilb_{\vert ,C}^{F+\bullet E,N\bullet} \right) \\
 =
e_{\vir}\left( \Hilb_{\vert ,C}^{F+\bullet
E,N\bullet} (\widehat{C}\times E) \right) \cdot e_{\vir}\left( \Hilb^{\bullet E,N\bullet}(U\times E)^{\ZmodN{N}} \right)
\end{multline*}
and
\begin{align*}
e_{\vir}\left( \Hilb^{\bullet E,N\bullet}(U\times E)^{\ZmodN{N}}
\right) &= \sum_{n=0}^{\infty} e\left(\Hilb^{n}([U/\ZmodN{N}]) \right)
\, q^{n}\\
&=\frac{q}{\Delta_{N}(q)} \prod_{k=1}^{\infty} (1-q^{kN})
\end{align*}
since $[S/\ZmodN{N}]=[U/\ZmodN{N}]\cup [C/\ZmodN{N}]$ and $\ZmodN{N}$
acts freely on $C$ with $e(C/\ZmodN{N})=1$.

As in the previous case, we get an action of $MW$, the Mordell-Weil
group of sections of $\widehat{C}$, on $\Hilb_{\vert ,C}^{F+\bullet
E,N\bullet}(\widehat{C}\times E)$. Note that the group of the fiber $C$
is $\BC^{*}\times \ZmodN{N}$ and (unlike the case of smooth fibers)
the restriction of $MW$ to the group of the fiber
splits. Consequently, we get a $\BC^{*}$ action on $\Hilb_{\vert
,C}^{F+\bullet E,N\bullet}(\widehat{C}\times E)$. Moreover this action
preserves the Behrend function, so it will suffice to prove the
following:
\begin{multline}\label{eqn: formula for evir(HilbvertC)}
 e_{\vir}\left(\Hilb_{\vert ,C}^{F+\bullet
E,N\bullet}(\widehat{C}\times E)^{\BC^{*}} \right)
=\\
-\prod_{k=1}^{\infty} (1-q^{kN})^{-1}\cdot \left\{\wp(q^{N})
-\frac{1}{12}E_{2}(q^{N}) +\delta_{1,N} \right\}.
\end{multline}

The horizontal components of a $\BC^{*}$ invariant subscheme supported
on $\widehat{C}\times E$ must be supported at the fixed points of
$\BC^{*}$, namely over the nodes in $C$. There are formal local
coordinates $(x,y)$ on $S$ at each node such that $C$ is formally
locally given by $xy=0$ and $\mu \in \BC^{*}$ acts by $\mu 
(x,y)=(\mu x,\mu^{-1}y)$. Then to be $\BC^{*}$-invariant, the
horizontal components must be of the form $Z_{\lambda }\times E$ where
$\lambda$ is an integer partition and $Z_{\lambda}$ is the length
$|\lambda |$ zero dimensional subscheme supported at a node of $C$ and
given by the monomial ideal $(x^{i}y^{j})_{(i,j)\not\in  \lambda}$.

We thus see that a subscheme supported on $\widehat{C}\times E$ which
is both $\ZmodN{N}$ and $\BC^{*}$ invariant must have its one
dimensional components given by
\[
C(\lambda ) = \bigcup_{g\in \ZmodN{N}} g\left(C_{0}\times x_{0} \cup
Z_{\lambda} \times E \right)
\]
(see figure~\ref{fig: the curve C(lambda)}).

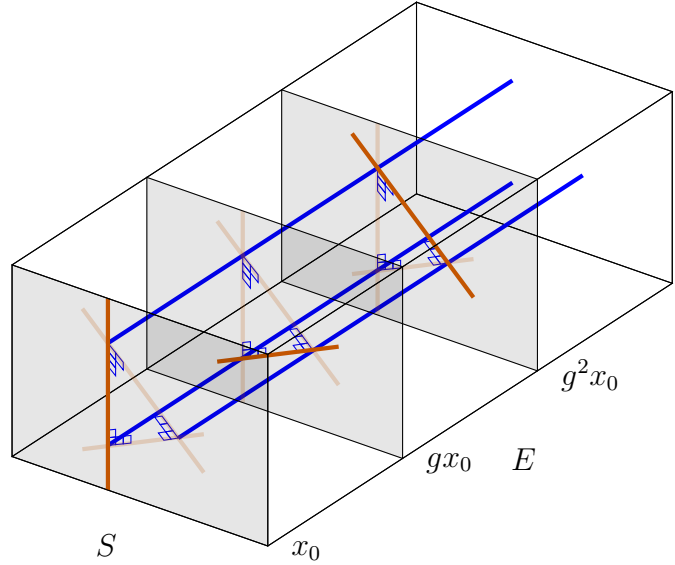
\begin{figure}
\usetikzlibrary{shapes.geometric}

\newcommand{\partitionbox}[2]{
	(#1, #2) --
	++(0.1, 0.0) --
	++(0.0, 0.1) --
	++(-0.1,0.0) --
	cycle
}

\newcommand{\partdraw}{
\draw [blue] \partitionbox{0}{0} ;
\draw [blue] \partitionbox{0}{0.1} ;
\draw [blue] \partitionbox{0.1}{0} ;
\draw [blue] \partitionbox{0.2}{0} ;
}

\begin{tikzpicture}[
                    z  = {-15},
                    scale = 0.85]

\begin{scope}[yslant=-0.35,xslant=0]

\begin{scope} [canvas is xy plane at z=12]
  \begin{scope}[shift={(1.5,.7)}, scale=1.2,yslant=0.45] 
  \partdraw;
  \end{scope}
  \begin{scope}[shift={(1.5,2.3)},scale=1.2,rotate=-90, xslant=0.9] 
  \partdraw;
  \end{scope}
  \begin{scope}[shift={(2.6,1.2)}, scale=1.2,rotate=135,xslant=0.33] 
  \partdraw;
  \end{scope}
\end{scope}
\begin{scope} [canvas is xy plane at z=8]
  \begin{scope}[shift={(1.5,.7)}, scale=1.2,yslant=0.45] 
  \partdraw;
  \end{scope}
  \begin{scope}[shift={(1.5,2.3)},scale=1.2,rotate=-90, xslant=0.9] 
  \partdraw;
  \end{scope}
  \begin{scope}[shift={(2.6,1.2)}, scale=1.2,rotate=135,xslant=0.33] 
  \partdraw;
  \end{scope}
\end{scope}
\begin{scope} [canvas is xy plane at z=4]
  \begin{scope}[shift={(1.5,.7)}, scale=1.2,yslant=0.45] 
  \partdraw;
  \end{scope}
  \begin{scope}[shift={(1.5,2.3)},scale=1.2,rotate=-90, xslant=0.9] 
  \partdraw;
  \end{scope}
  \begin{scope}[shift={(2.6,1.2)}, scale=1.2,rotate=135,xslant=0.33] 
  \partdraw;
  \end{scope}
\end{scope}

\begin{scope} [canvas is yz plane at x=0]
\draw [black](0,0) rectangle (3,12);
\end{scope}
\begin{scope} [canvas is xz plane at y=0]
\draw [black](0,0) rectangle (4,12);
\end{scope}
\draw [black](0,0) rectangle (4,3);

\draw [ultra thick, blue,opacity=1.0]
(1.5,0.7,0)--(1.5,0.7,12)
(2.6,1.2,0)--(2.6,1.2,12)
(1.5,2.3,0)-- (1.5,2.3,12)
;

\draw [ultra thick,orange,opacity=1.0] 
(1.5  ,0   ,12)--(1.5,3,12);
\draw [ultra thick,orange,opacity=0.2] 
(1.5  ,0   ,8)--(1.5,3,8);
\draw [ultra thick,orange,opacity=0.2] 
(1.5  ,0   ,4)--(1.5,3,4);

\draw [ultra thick,orange,opacity=0.2] 
(1.1  ,2.7   ,12)--(3.0,0.8,12);
\draw [ultra thick,orange,opacity=0.2] 
(1.1  ,2.7   ,8)--(3.0,0.8,8);
\draw [ultra thick,orange,opacity=1.0] 
(1.1  ,2.7   ,4)--(3.0,0.8,4);

\draw [ultra thick,orange,opacity=0.2] 
(1.1  ,0.5   ,12)--(3.0,1.4,12);
\draw [ultra thick,orange,opacity=1.0] 
(1.1  ,0.5   ,8)--(3.0,1.4,8);
\draw [ultra thick,orange,opacity=0.2] 
(1.1  ,0.5   ,4)--(3.0,1.4,4);

\node [below] at (2,0,13) {$S$};
\node [right] at (5,0,7) {$E$};
\node [right] at (4.2,0,12) {$x_{0}$};
\node [right] at (4.2,0,8) {$gx_{0}$};
\node [right] at (4.2,0,4) {$g^{2}x_{0}$};

\begin{scope} [canvas is yz plane at x=4]
\draw [black](0,0) rectangle (3,12);
\end{scope}

\begin{scope} [canvas is xz plane at y=3]
\draw [black](0,0) rectangle (4,12);
\end{scope}

\draw [black](0,0,12) rectangle (4,3,12);
\draw [black,fill, opacity=0.1](0,0,12) rectangle (4,3,12);
\draw [black](0,0,8) rectangle (4,3,8);
\draw [black,fill, opacity=0.1](0,0,8) rectangle (4,3,8);
\draw [black](0,0,4) rectangle (4,3,4);
\draw [black,fill, opacity=0.1](0,0,4) rectangle (4,3,4);


\end{scope}

\end{tikzpicture} \caption{The $\BC ^* \times\ZmodN{3}$ invariant curve $C(\lambda )$. The horizontal components (shown in blue) are thickened by the monomial ideal corresponding to the partition $\lambda =(3,1)$. The vertical components (shown in bold orange) consist of the curves $g(C_0\times \{x_0\})$ where $g\in \ZmodN{3}$. }
\label{fig: the curve C(lambda)}
\end{figure}

We let 
\[
\Hilb_{C(\lambda )}^{m}\subset \Hilb_{\vert ,C}^{F+|\lambda
|E,nN}(\widehat{C}\times E)^{\BC^{*}}
\]
be the component parameterizing $\BC^{*}\times \ZmodN{N}$-invariant
subschemes $Z$, supported on $\widehat{C}\times E$, containing
$C(\lambda)$, and such that $I_{C(\lambda )}/I_{Z}$ is a
zero-dimensional sheaf of length $mN$. In other words, the subscheme
$Z' = Z/\ZmodN{N}\subset X$ is obtained from the subscheme $C'(\lambda ) = C(\lambda
)/\ZmodN{N}\subset X$ by adding $m$ embedded points.

Let 
\[
\Hilb^{\bullet}_{C(\lambda )} = \sum_{m=0}^{\infty} \Hilb_{C(\lambda
)}^{m}\, y^{m}.
\]
We define a constructible morphism
\[
\rho :\Hilb^{\bullet}_{C(\lambda )} \to \Sym^{\bullet}(E')
\]
as follows.

Let $[Z]\in \Hilb^{m}_{C(\lambda )}$ be a closed point corresponding
to a $\ZmodN{N}\times \BC^{*}$ invariant subscheme $Z$ containing
$C(\lambda )$. Let $C'(\lambda )\subset Z'\subset X$ be the
corresponding subschemes in $X$ and let $F'_{Z}$ be the
length $m$, zero-dimensional quotient sheaf  $F'_{Z} = I_{C'(\lambda
)}/I_{Z'}$. Then define
\[
\rho :[Z]\mapsto \operatorname{supp}(\pi_{*}F_{Z}')
\]
where $\pi :X\to E'$ and support is given as a collection of points
with multiplicity.

We then have
\begin{align}\label{eqn: e(HilbC(lambda))=B(y)/A(y)}
e(\Hilb^{\bullet}_{C(\lambda )} )& = e\left(\Sym^{\bullet}E',\rho_{*}1 \right)\nonumber\\
&= e\left(\Sym^{\bullet}(E'-\{x_{0} \}),\rho_{*}1 \right)\cdot e\left(\Sym^{\bullet}\{x_{0} \},\rho_{*}1 \right)\nonumber\\
&=\left(\sum_{k=0}^{\infty}a(k)y^{k} \right)^{-1} \cdot \left(\sum_{k=0}^{\infty}b(k)y^{k} \right)
\end{align}
where
\begin{align*}
a(k) &= e\left(\rho^{-1}(kx) \right)\\
b(k) &= e\left(\rho^{-1}(kx_{0}) \right)
\end{align*}
and $x\in E'-\{x_{0} \}$. Here we have used the fact that
$\rho_{*}1$ satisfies the multiplicative property required by
equation~\eqref{eqn: formula for euler char of multiplicative
constructible functions on symmetric products} , over $\Sym
(E'-\{x_{0} \})$.

The preimages $\rho^{-1}(kx)$ and $\rho^{-1}(kx_{0})$ parameterize
subschemes obtained by adding $k$ $\BC^{*}$ invariant embedded points
to $C'(\lambda )$ in the fiber of $X\to E'$ over $x\in E'$ and
$x_{0}\in E'$ respectively. Such subschemes are determined formally
locally at the support of the embedded points and consequently can be
written in terms of the following local model.

Let $\lambda ,\mu ,\nu$ be a triple of $2D$ partitions and consider
the scheme $C_{\lambda \mu \nu}\subset \BA^{3}$ given by the ideal
$I_{\lambda \mu \nu}\subset \BC [x,y,z]$ where
\begin{align*}
I_{\lambda \mu \nu} & = I_{\lambda \emptyset \emptyset} \cap
I_{\emptyset \mu \emptyset}\cap I_{\emptyset \emptyset \nu} \\
I_{\lambda \emptyset \emptyset}&= (x^{i}y^{j})_{(i,j)\not\in \lambda }\\
I_{\emptyset  \mu \emptyset}&= (y^{j}z^{k})_{(j,k)\not\in \mu }\\
I_{\emptyset  \emptyset \nu}&= (z^{k}x^{i})_{(k,i)\not\in \nu }
\end{align*}

Let 
\[
\Quot^{m}_{\lambda \mu \nu} = \left\{I_{\lambda \mu \nu} \to Q,\quad
\operatorname{length}(Q)=m,\quad \operatorname{supp}(Q)=(0,0,0) \right\}
\]
be the Quot scheme of zero-dimensional, length $m$ quotients of
$I_{\lambda \mu \nu}$ supported at $(0,0,0)\in
\BA^{3}$. Hence $\Quot^{m}_{\lambda \mu \nu}$ parameterizes subschemes
obtained by adding $m$ embedded points to $C_{\lambda \mu \nu}$ at the
origin.

We define
\[
\Quot^{\bullet}_{\lambda \mu \nu} = \sum_{m=0}^{\infty}
\Quot^{m}_{\lambda \mu \nu} y^{m}
\]
and we define
\[
\Vtilde_{\lambda \mu \nu}(y) = e\left(\Quot^{\bullet}_{\lambda \mu \nu} \right).
\]
$\Vtilde_{\lambda \mu \nu}(y)$ is the normalized topological vertex
(see \cite{Bryan-Kool}).

Using formal local coordinates at the point in $C'(\lambda )$ where
embedded points are added, we can describe the preimages
$\rho^{-1}(kx)$ and $\rho^{-1}(kx_{0})$ in terms of 
\[
(\Quot^{\bullet}_{\lambda \emptyset \emptyset })^{\BC^{*}},\quad (\Quot^{\bullet}_{\lambda \square \emptyset })^{\BC^{*}},\quad (\Quot^{\bullet}_{\lambda \emptyset \square })^{\BC^{*}},\quad (\Quot^{\bullet}_{\lambda \square \square })^{\BC^{*}}
\]
where $\alpha\in\BC^{*}$ acts on $(x,y,z)$ by $(\alpha
x,\alpha^{-1}y,z)$. Specifically, we get
\begin{align*}
\sum_{k=0}^{\infty} \rho^{-1}(kx)\, y^{k} &= \left(\left(
\Quot^{\bullet}_{\lambda \emptyset \emptyset }\right)^{\BC^{*}}
\right)^{N},\\
&\quad \\[-10pt]
\sum_{k=0}^{\infty} \rho^{-1}(kx_{0})\, y^{k} &=\begin{cases}
\left(\left(
\Quot^{\bullet}_{\lambda \emptyset \emptyset }\right)^{\BC^{*}}
\right)^{N-2} \left(
\Quot^{\bullet}_{\lambda \square \emptyset }\right)^{\BC^{*}}
\left(
\Quot^{\bullet}_{\lambda \emptyset \square }\right)^{\BC^{*}}
&\text{ if $N\geq 2$,}\\
&\quad \\
\quad \left(
\Quot^{\bullet}_{\lambda \square \square }\right)^{\BC^{*}}&\text{ if $N=1$.}
\end{cases}
\end{align*} 
Indeed, in each case we are adding $\BC^{*}$ invariant embedded points
to $C'(\lambda )$ at the $N$ points in the fiber of $X\to E'$
corresponding to the $N$ nodes of the $I_{N}$ fiber (the $\BC^{*}$
fixed points). Over $x\neq x_{0}$ at the $N$ points, $C'(\lambda )$ is
formally locally given by $C_{\lambda \emptyset \emptyset}$. Over
$x_{0}$, $C'(\lambda )$ has a vertical component which meets 2 of the
nodes (if $N\geq 1$) so that at these nodes, $C'(\lambda )$ is
formally locally $C_{\lambda \square\emptyset }$ and $C_{\lambda
\emptyset \square }$ respectively. In the case of $N=1$, the vertical
component itself has a node and $C'(\lambda )$ is locally $C_{\lambda
\square\square }$.

Applying Euler characteristics and equation~\eqref{eqn:
e(HilbC(lambda))=B(y)/A(y)} we get
\[
e(\Hilb^{\bullet}_{C(\lambda )}) =
\begin{cases}
\Vtilde_{\lambda \square \emptyset} \cdot \Vtilde_{\lambda \emptyset
\square }\cdot \Vtilde^{-2}_{\lambda \emptyset \emptyset}
&\quad N\geq 2\\
&\quad \\
\Vtilde_{\lambda \square \square }\cdot \Vtilde^{-1}_{\lambda \emptyset \emptyset} &\quad N=1
\end{cases}
\]

We wish to rewrite the above in terms of the vertex with the usual
normalization. This works out nicely when we reindex by our subschemes
by holomorphic Euler characteristic instead of number of embedded
points. Using the normalization exact sequence for $C'(\lambda )$, we
can compute
\[
\chi (\CO_{C'(\lambda )}) = -\lambda_{1}-\lambda_{1}'+1-\delta_{N,1}.
\]
From \cite[Lemma~17]{Bryan-Kool} we have
\[
\Vtilde_{\lambda \emptyset \emptyset} =\Vsf_{\lambda \emptyset
\emptyset},\, \, \, 
\Vtilde_{\lambda \square \emptyset} =y^{\lambda_{1}}\cdot \Vsf_{\lambda
\square \emptyset},\, \, \, 
\Vtilde_{\lambda \emptyset  \square} =y^{\lambda'_{1}}\cdot \Vsf_{\lambda
\emptyset \square},\, \, \, 
\Vtilde_{\lambda \square \square} =y^{\lambda_{1}+\lambda '_{1}-1}\cdot \Vsf_{\lambda \square \square}
\]
and so we get
\[
\sum_{m=0}^{\infty} e\left(\Hilb^{m}_{C(\lambda )} \right) \, y^{\chi (\CO_{C'(\lambda )})+m} =
\begin{cases}
y\cdot \Vsf_{\lambda \square \emptyset} \cdot \Vsf_{\lambda \emptyset
\square }\cdot \Vsf^{-2}_{\lambda \emptyset \emptyset}
&\quad N\geq 2\\
&\quad \\
y\cdot \Vsf_{\lambda \square \square }\cdot \Vsf^{-1}_{\lambda \emptyset \emptyset} &\quad N=1
\end{cases}
\]
We define
\[
\Elambda (y) =y^{\frac{1}{2}} \frac{\Vsf_{\lambda \square \emptyset
}}{\Vsf_{\lambda \emptyset \emptyset }} = \sum_{i=1}^{\infty}
y^{-\lambda_{i}+i-\frac{1}{2}}.  
\]
Then using \cite[Proof of Lemma 5]{Bryan-Kool-Young} we may express
$y\cdot \Vsf_{\lambda \square \emptyset} \cdot \Vsf_{\lambda \emptyset
\square }\cdot \Vsf^{-2}_{\lambda \emptyset \emptyset} $ and $y\cdot
\Vsf_{\lambda \square \square }\cdot \Vsf^{-1}_{\lambda \emptyset
\emptyset}$ in terms of $\Elambda$. Namely
\[
\sum_{m=0}^{\infty} e\left(\Hilb^{m}_{C(\lambda )} \right)\, y^{\chi (\CO_{C'(\lambda
)})+m} =\delta_{N,1} - \Elambda (y)\Elambda (y^{-1}). 
\]

We now use 
Conjecture 21 in \cite{Bryan-Kool}, applied to the subscheme
$C'(\lambda )\subset X$, to convert the Euler characteristic to
virtual Euler characteristic. In our context, the conjecture says that
\[
e_{\vir}(\Hilb^{m}_{C(\lambda )}) = - \nu (C'(\lambda ))
(-1)^{m} e(\Hilb^{m}_{C(\lambda )})
\]
where $\nu  (C'(\lambda ))$ is the value of the Behrend
function of $\Hilb (X)/E$ at the subscheme $ C'(\lambda )\subset X$
(the extra sign is because of the quotient by $E$).

\begin{lemma}\label{lem: behrend function at C(lambda)}
The value of the Behrend function of $\Hilb (X)/E$ at the point
$C'(\lambda )$ is given by
\[
\nu  (C'(\lambda )) = -(-1)^{\chi (\CO_{C'(\lambda )})}.
\]
\end{lemma}
\begin{proof}
Using the methods of \cite[Sec~9]{Bryan-Kool}, one can show that
$C'(\lambda )$ is a smooth point of $\Hilb (X)/E$ of
dimension $2|\lambda |-\lambda_{1}-\lambda_{1}'+\delta_{1,N}$.
\end{proof}

Assuming the conjecture then we get
\begin{align*}
\sum_{m=0}^{\infty} e_{\vir }\left(\Hilb^{m}_{C(\lambda )} \right)\, y^{\chi (\CO_{C'(\lambda
)})+m} &=-\sum_{m=0}^{\infty} e\left(\Hilb^{m}_{C(\lambda )} \right)\, (-y)^{\chi (\CO_{C'(\lambda
)})+m}\\
&=-\delta_{N,1} + \Elambda (-y)\Elambda (-y^{-1})\\
&=-\delta_{N,1} + \Elambda (p)\Elambda (p^{-1})
\end{align*}
where $p=-y$.

So then
\begin{align*}
e_{\vir}\left(\Hilb_{\vert ,C}^{F+\bullet
E,N\bullet}(\widehat{C}\times E)^{\BC^{*}} \right) &=
\sum_{\lambda} Q^{N|\lambda|} \sum_{m-0}^{\infty}
e_{\vir}\left(\Hilb^{m}_{C(\lambda )} \right) \, y^{\chi
(\CO_{C'(\lambda )})+m} \\
&= \sum_{\lambda} q^{N|\lambda |} \left(-\delta_{N,1} + \Elambda (p)\Elambda (p^{-1}) \right)\\
&=-\prod_{k=1}^{\infty} (1-q^{kN})^{-1}\cdot \left(\delta_{N,1} - F(p,p^{-1};q^{N}) \right)
\end{align*}
where $F(x_{1},\dotsc ,x_{n};q)$ is Block-Okounkov's $n$-point
function. $F(p,p^{-1};q)$ is evaluated in
\cite[\S~4]{Bryan-Kool-Young} and is given by
\begin{align*}
 F(p,p^{-1};q^{N}) & = \frac{-p}{(1-p)^2} - \sum_{d=1}^\infty \sum_{k|d}k(p^k+p^{-k})q^{Nd}\\
  &= -\wp(p,q^{N}) + \frac{1}{12} E_{2}(q^{N}).    
\end{align*}

This completes the proof of equation \eqref{eqn: formula for
evir(HilbvertC)} and hence the proof of Propostion~\ref{prop: evir(Hilb_{vert})}.

\subsection{Putting vertical and diagonal contributions together.}
We have
\begin{align*}
    \Ztsecondcoeff &= q^{-1} e_{\vir} \left( \Hilb^{\frac{1}{N}F+\bullet E',\bullet}(X)/E\right)\\
    &=q^{-1} e_{\vir} \left( \Hilb^{F+\bullet E,N\bullet}(S\times E)^\ZmodN{N}/E\right)\\
    &=q^{-1} e_{\vir} \left( \Hilb^{F+\bullet E,N\bullet}_\vert   + \Hilb^{F+\bullet E,N\bullet}_\diag\right).
\end{align*}

Then using Lemma~\ref{lem: evir(Hilb_{diag})} and Proposition~\ref{prop: evir(Hilb_{vert})} we get

\begin{multline*}
 \Ztsecondcoeff 
    = \frac{-1}{\Delta_N(q)}\cdot 24 \cdot \frac{\phi_1(N)}{\phi_2(N)} \cdot \left\{\wp (q^N) - \frac{1}{12}E_2(q^N) +\frac{1}{12} \delta_{1,N} \right\} +\\
     \ \quad \ \frac{-2}{\Delta_N(q)\phi_2(N)} \cdot \sum_{m|N} (E_2(q^m)-1)\mu (m)\\
    =\frac{2\phi_1(N)}{\Delta_N(q)\phi_2(N)} \cdot \left\{ -12\wp (q^N)+E_2(q^N)-\delta_{1,N} - \frac{1}{\phi_1(N)} \sum_{m|N} (E_2(q^m)-1)\mu(m)\right\}
\end{multline*}

Substituting $E_2(q^m) = \widetilde{\CE }_m(q)+\frac{1}{m}E_2(q)$ and using the facts that
\begin{align*}
    \sum_{m|N}\mu(m) &= \delta_{1,N}\\
    \sum_{m|N}\frac{\mu(m)}{m} &= \frac{\phi_1(N)}{N}
\end{align*}
we conclude
\begin{equation*}
    \Ztsecondcoeff = \frac{2\phi_1(N)}{\Delta_N(q)\phi_2(N)} \left\{-12\wp(q^N)+\widetilde{\CE }_N - \frac{1}{\phi_1(N)} \sum_{m|N} \widetilde{\CE }_m(q) \mu(m) \right\}
\end{equation*}
which completes the proof of Theorem~\ref{thm:intro2} \qed. 

\section{Lattice computations} \label{first_coefficient_in_q} We
present the proof of the following part of Theorem~\ref{thm:intro1}.
\begin{thm} \label{thm:first_coeff_in_q} Let $X$ be an order $N$ elliptic CHL model. Then
\[ \left[ \mathsf{Z}^X(q,t,p) \right]_{q^{-1}} = \frac{1}{ \Theta(t,p)^2 \cdot f_N(t^{1/N}) }. \]
\end{thm}

For the proof we apply the degeneration formula 
to reduce to the computation of a theta function of the coinvariant lattice.
We discuss the connection with the McKay correspondence in a remark.

\subsection{Proof of Theorem~\ref{thm:first_coeff_in_q}} \label{subsection:first_coefficient_in_q}
Applying the degeneration formula \eqref{deg_formula} to the left hand side of Theorem~\ref{thm:first_coeff_in_q} gives
\[
\left[ \mathsf{Z}^X \right]_{q^{-1}}
=
\frac{1}{N} \sum_{h \geq 0} \sum_{n \in \BZ} \sum_{ \substack{ \tilde{\beta} \in H_2(S,\BZ)_{>0} \\ P(\tilde{\beta}) = \beta_h }} t^{ \langle \beta_h, \beta_h \rangle/2 } (-p)^n
\PT^{S \times \p^1}_{n, (\tilde{\beta},0)}( 1 \otimes 1 ).
\]
Since the curve class $\beta_h$ is indivisible in $P(N_1(S))$
the classes $\tilde{\beta}$ in the third sum on the right are primitive.
By deformation invariance, we may hence evaluate the rubber invariant
$\PT^{S \times \p^1}_{n, (\tilde{\beta},0)}( 1 \otimes 1 )$ by deforming $(S,\tilde{\beta})$
to a pair $(S',\beta')$ such that $\beta'$ is an irreducible curve class.
If $\beta'$ is irreducible, every curve in $S' \times \p^1$ of class $(\beta',0)$ is contained in a fiber over a single point in $\p^1$.
Moreover, the moduli space of rubber stable pairs is isomorphic to the moduli space
of stable pairs on $S'$ in class $\beta'$:
\[ P_{n}^{\sim}(S' \times \p^1 / \{ S_0, S_{\infty} \}, (\beta',0)) \cong P_n(S',\beta'). \]
Since $P_n(S',\beta')$ is non-singular, its Behrend function takes the constant value $(-1)^{\dim P_n(S',\beta')}$. Hence
\[
\PT^{S \times \p^1}_{n, (\tilde{\beta},0)}( 1 \otimes 1 ) = \int_{P_n(S,\beta')} (-1)^{ \dim(P_n(S,\tilde{\beta})) } \dd{e}.
\]
By the Kawai--Yoshioka formula \cite{KY} (see also \cite{MPT}) we conclude\footnote{
Alternatively, \eqref{evaluation} follows from applying the rigidification lemma, the degeneration formula and localization.
See \cite{K3xP1} for similar arguments.}
\begin{equation} \label{evaluation}
\sum_n (-p)^n \PT^{S \times \p^1}_{n, (\tilde{\beta},0)}( 1 \otimes 1 )
=
\left[ \frac{1}{\Delta(t) \Theta(t,p)^2} \right]_{t^{\langle \tilde{\beta}, \tilde{\beta} \rangle}/2}. 
\end{equation}
Inserting this into the degeneration formula we obtain
\begin{equation} \label{evaluation2}
\left[ \mathsf{Z}^X \right]_{q^{-1}}
=
\frac{1}{N} \sum_{h \geq 0}
 \sum_{ \substack{ \tilde{\beta} \in H_2(S,\BZ)_{>0} \\ P(\tilde{\beta}) = \beta_h }}
t^{ \langle \beta_h, \beta_h \rangle/2 }
\left[ \frac{1}{\Delta(t) \Theta(t,p)^2} \right]_{t^{\langle \tilde{\beta}, \tilde{\beta} \rangle}/2}.
\end{equation}

We consider the set of effective curve classes $\tilde{\beta}$ of $S$ with $P(\tilde{\beta}) = \beta_h$ for some $h \geq 0$.
As in Section~\ref{sec:z0} we may assume that the fibration $S \to \p^1$ 
is generic and in particular all reducible fibers are of type $I_n$.
Let
\[ C_{i}^{(j)} \in \Pic(S),\ \  i=0, \ldots, n_j \]
be the classes of irreducible components of the $j$-th fiber (which is of type $I_{n_j+1}$).
We order the $C_{i}^{(j)}$
such that $C_0^{(j)}$ is the component which meets the zero section $\sigma_0$, and the matrix
\[ \left( \, \langle C_{i}^{(j)}, C_{k}^{(j)} \rangle \, \right)_{i,j=1}^{n_j}
\]
is the Cartan matrix of the negative $A_{n_j}$ lattice.
In particular, if $L$ is the lattice spanned by all $C_{i}^{(j)}$ for $i\geq 1$ then we have
\[ L \cong \bigoplus_{j} A_{n_j}(-1). \]

Every effective class $\tilde{\beta}$ satisfying $P(\tilde{\beta}) = \beta_h$ is of degree $1$ over the base of the elliptic fibration
and hence of the form
\[ \tilde{\beta} = \sigma_i + aF + \sum_{j} \sum_{i=1}^{n_j} d_{i,j} C_{i}^{(j)} \]
for some $i=0, \ldots, N-1$, $a \geq 0$ and $d_{i,j} \in \BZ$.
By translating by $\sigma_{i}$ (or rather the inverse operation) we can assume that $\tilde{\beta}$ is of the form
\begin{equation} \tilde{\beta} = \sigma_0 + aF + \sum_{j} \sum_{i=1}^{n_j} d_{i,j} C_{i}^{(j)}. \label{gdfgdf} \end{equation}

We make one more simplification. 
Consider any, not necessarily effective class $\tilde{\beta}$ of the form \eqref{gdfgdf}.
We have
\begin{equation}
\label{normsddas}
 \langle \tilde{\beta}, \tilde{\beta} \rangle = 2a - 2 + \langle \alpha, \alpha \rangle, \quad \alpha =
\sum_{j} \sum_{i=1}^{n_j} d_{i,j} C_{i}^{(j)}.
\end{equation}
If $\langle \tilde{\beta}, \tilde{\beta} \rangle \geq -2$, then the class $\tilde{\beta}$ is effective by the Riemann--Roch formula.
If $\langle \tilde{\beta}, \tilde{\beta} \rangle < - 2$
then we have
\[ \left[ \frac{1}{\Delta(t) \Theta(t,p)^2} \right]_{t^{\langle \tilde{\beta}, \tilde{\beta} \rangle/2}} = 0.
 \]
Hence in \eqref{evaluation2} we may drop the effectivity condition on the classes $\tilde{\beta}$
and sum over all $a \in \BZ$ and $\alpha \in L$. 
Putting both simplifications into \eqref{evaluation2} and accounting for overcounting $N$ times the classes \eqref{gdfgdf} by
canceling the $1/N$ factor we get
\begin{equation} \label{evaluation33}
\left[ \mathsf{Z}^X \right]_{q^{-1}}
=
\sum_{a \in \BZ} \sum_{\alpha \in L}
t^{ \langle P(\beta_{a,\alpha})), P(\beta_{a,\alpha}) \rangle/2 }
\left[ \frac{1}{\Delta(t) \Theta(t,p)^2} \right]_{t^{\langle \beta_{a,\alpha}, \beta_{a,\alpha} \rangle/2}}
\end{equation}
where $\beta_{a,\alpha} = \sigma_0 + aF + \alpha$.

Using
\[
\quad P(C_i^{(j)}) = \frac{1}{n_j+1} F
\]
for all $i$ and $j$ we have
\[ \frac{1}{2} \langle P(\beta_{a,\alpha}), P(\beta_{a,\alpha}) \rangle
=
-\frac{1}{N} + a + \sum_j \frac{1}{n_j+1} \sum_{i} d_{ij}. \]
Inserting this and \eqref{normsddas} into \eqref{evaluation33} then yields
\begin{align*}
\left[ \mathsf{Z}^X \right]_{q^{-1}}
& =
\sum_{\alpha \in L} t^{-\frac{1}{2} \langle \alpha, \alpha \rangle + \sum_j \frac{1}{n_j+1} \sum_i d_{ij} + \frac{N-1}{N}}
\sum_{a \in \BZ} t^{\frac{1}{2} \langle \alpha, \alpha \rangle + a - 1} 
\left[ \frac{1}{\Delta \Theta^2} \right]_{t^{\frac{1}{2} \langle \alpha, \alpha \rangle + a - 1}} \\
& =
\frac{1}{\Delta(t) \Theta(t,p)^2}
\cdot
t^{\frac{N-1}{N}}
\prod_j \sum_{d_j=(d_{ij})_{i}} t^{\frac{1}{n_j+1} \sum_{i} d_{ij}} t^{d_j^T C_{A_{n_j}} d_j}
\end{align*}
where we let $C_{A_n}$ denote the Cartan matrix of the $A_n$ Dynkin diagram.

Consider the following theta series of the $A_n$ lattice:
\[
\vartheta_{A_n,v_n}(q)
= \sum_{m = (m_1, \ldots, m_n) \in \BZ^n} q^{\frac{1}{2} (m+v_n)^T C_{A_n} (m + v_n)}
\]
where
\[ v_n = \frac{1}{n+1} C_{A_n}^{-1} \begin{pmatrix} 1 \\ \vdots \\ 1 \end{pmatrix}. 
\]
By a direct check using Table~\ref{Table_ell_fibers} we have\footnote{
A small calculation shows $\frac{1}{2} v_n^T C_{A_n} v_n = \frac{1}{24} \frac{n (n+2)}{n+1}$.
}
\[ \frac{N-1}{N} - \frac{1}{2} \sum_j v_{n_j}^T C_{A_{n_j}} v_{n_j} = 0. \]
So we conclude
\[
\left[ \mathsf{Z}^X \right]_{q^{-1}}
=
\frac{1}{\Delta(t) \Theta(t,p)^2}
\prod_j \vartheta_{A_{n_j}, v_{n_j}}(t).
\]
The claim now follows from Table~\ref{Table_ell_fibers} and the following identity
which is a special case of \cite[Thm.1.3]{GNS} (set $q_i = q^{1/(N+1)}$ for all $i$):
\[
\vartheta_{A_n,v_n}(q) = 
\frac{\eta(q)^{n+1}}{\eta(q^{1/(n+1)})}.
\]
\qed
\begin{figure}
\begin{longtable}{|c|c|}
\hline 
$N$& singular fibers \\ \hline 
$1$& $24I_1$ \\ 
$2$& $8 I_1 + 8 I_2$ \\ 
$3$& $6 I_1 + 6 I_3$ \\ 
$4$& $4 I_4 + 2 I_2 + 4 I_1$ \\ 
$5$& $4 I_5 + 4 I_1$ \\ 
$6$& $2 I_6 + 2 I_3 + 2 I_2 + 2 I_1$ \\
$7$& $3 I_7 + 3 I_1$ \\  
$8$& $2 I_8 + I_4 + I_2 + 2 I_1$ \\
\hline
\caption{The number and type of singular fibers of
a generic elliptic K3 surface $S \to \p^1$ with order $N$ section.
Taken from \cite[Table 1]{GarS2}.}
\label{Table_ell_fibers}
\end{longtable}
\end{figure}

\begin{rmk}
Under the variable change $p=e^{iu}$ consider the $u^{-2}$ coefficient of \eqref{evaluation33},
\begin{equation}
\left[ \mathsf{Z}^X \right]_{q^{-1}u^{-2}}
=
\sum_{a \in \BZ} \sum_{\alpha \in L}
t^{ \langle P(\beta_{a,\alpha})), P(\beta_{a,\alpha}) \rangle/2 }
\left[ \frac{1}{\Delta(t)} \right]_{t^{\langle \beta_{a,\alpha}, \beta_{a,\alpha} \rangle/2}}
\end{equation}
By the GW/DT correspondence (Section~\ref{subsection:GWtheory}) this is precisely the genus $0$ contribution to the series
$[\mathsf{Z}^X]_{q^{-1}}$.
The sum on the right hand side arises also naturally from the McKay correspondence as follows.

Let $S$ be the elliptic K3 surface on which $G = \BZ_N$ acts by translating by an order $N$ section.
Let $S'$ be the crepant resolution of the coarse quotient $S/G$.
The lattice spanned by exceptional classes on $S'$ is isomorphic to $L$.
(This is a consequence of the equality of the
values in Table~\ref{Table_ell_fibers} and the table after Lemma~\ref{lem: formula for e(Hilb(Y/G))})
One can now show that under the McKay correspondence \cite{KV}
\[ \Phi : D^b([S/G]) \equiv D^b(\Coh_G(S)) \to D^b(S'), \]
where $[S/G]$ is the quotient stack, the generating series of Euler charistics
\[ \sum_{n=0}^{\infty} q^{n-1} e\left( \Hilb^n([S/G]) \right) \]
precisely corresponds to the right hand side in \eqref{evaluation33}.
Using \eqref{Delta_N_def} we hence recover the claim without proving
identities for the theta functions of the $A_n$ lattice. \qed
\end{rmk}

\section{Order two CHL models} \label{dfgadgs}
We expand the conjectures on order two CHL models by including also imprimitive classes.
A few base cases are discussed.

\subsection{Definition}
Let $g : S \to S$ be a symplectic involution of 
a non-singular projective K3 surface $S$,
and let
\[ X = (S \times E)/ \BZ_2 \]
be the associated CHL model.
Recall the projection operator
\[ P = \frac{1}{2} (1 + g_{\ast}) : H^{2}(S,\BQ) \to \frac{1}{2} H^2(S,\BQ). \]
In Section~\ref{subsec:order_two_CHL} we defined
the divisibility $\mathrm{div}(\gamma)$ of a class $\gamma \in P(N_1(S))$ to be the maximal positive integer $m$ such that
\[ \frac{\gamma}{m} \in P(N_1(S)). \]
Let $\tilde{\gamma} = \frac{\gamma}{\mathrm{div}(\gamma)}$. Then we say the class $\gamma$ is
\begin{itemize}
\item \emph{untwisted} if $\tilde{\gamma} \in H_2(S,\BZ)$,
\item \emph{twisted} if $\tilde{\gamma} \in \frac{1}{2} H_2(S,\BZ) \setminus H_2(S,\BZ)$.
\end{itemize}

Consider a curve class
\[ \beta = (\gamma, d) \in H_2(X,\BZ), \quad \gamma \neq 0. \]
If $\gamma$ is primitive,
then $\DT^X_{n, (\gamma,d)}$ only depends on
$n,d, s := \frac{1}{2} \langle \gamma, \gamma \rangle$ and whether $\gamma$ is twisted or not;
we have written
\[ \DT^X_{n, (\gamma,d)} = 
\begin{cases}
\DT^{\text{untw}}_{n, s, d} & \text{ if } \gamma \text{ is untwisted},\\
\DT^{\text{tw}}_{n, s, d}  & \text{ if } \gamma \text{ is twisted}.
\end{cases}
\]
A conjectural formula for the generating series of these primitive invariants was presented in Section~\ref{subsec:order_two_CHL}
as follows:
\begin{align*}
\mathsf{Z}^{\textup{tw}}(q,t,p) & = \frac{1}{\widetilde{\Phi_2}(Z)} \\
\mathsf{Z}^{\textup{untw}}(q, t,p) & = \frac{-8 F_4(Z) + 8 G_4(Z) - \frac{7}{30} E_{4}^{(2)}(2Z)}{\chi_{10}(Z)}.
\end{align*}

Here we present a multiple cover formula which expresses the Donald\-son--Thomas invariants
for imprimitive classes $\gamma$ in terms of the primitive invariants.
The conjecture is a direct consequence of
the multiple cover rule for K3 surfaces proposed in \cite[Conj.C]{K3xE}
and the computation scheme of Section~\ref{subsec:computation_scheme}.

\begin{conj} \label{Conj_Multiple_Cover} Let $\beta = (\gamma,d) \in H_2(X,\BZ)$ be a curve class.
\begin{enumerate}
\item If $\gamma$ is untwisted, then
\[ \DT_{n,(\gamma,d)} = \sum_{k | (n,\mathrm{div}(\gamma))} \frac{1}{k} \DT^{\textup{untw}}_{n/k, \frac{1}{2} \langle \gamma/k, \gamma/k \rangle, d} \]
\item If $\gamma$ is twisted, then
\[ 
\DT_{n,(\gamma,d)} 
= \sum_{\substack{k | (n,\mathrm{div}(\gamma)) \\ \mathrm{div}(\gamma)/k \textup{ even}}} \frac{1}{k} \DT^{\textup{untw}}_{\frac{n}{k}, \frac{1}{2} \langle \gamma/k, \gamma/k \rangle, d} 
+
\sum_{\substack{k | (n,\mathrm{div}(\gamma)) \\ \mathrm{div}(\gamma)/k \textup{ odd}}} \frac{1}{k} \DT^{\textup{tw}}_{\frac{n}{k}, \frac{1}{2} \langle \gamma/k, \gamma/k \rangle, d}
\]
\end{enumerate}
\end{conj}

\subsection{Evidence}
We work with the following model. Let 
\[ R \to \p^1 \] be a rational elliptic surface
with 12 rational nodal fibers. Let 
\[ f : \p^1 \to \p^1 \]
be a degree $2$ map, branched away from the base points of singular fibers.
Consider the elliptic K3 surface $S \to \p^1$ defined by the fiber diagram
\[
\begin{tikzcd}
S \ar{r} \ar{d} & R \ar{d} \\
\p^1 \ar{r} & \p^1.
\end{tikzcd}
\]
The $E_8$ lattice of sections of $R$ induces an $E_8$ lattice of sections of $S$.
Let $\sigma_0 : \p^1 \to S$ be a fixed section which we declare as the zero section, and let $F \in \Pic(S)$ be the class of a fiber. 
The Picard lattice of $S$ is
\[ \Pic(S) \cong \begin{pmatrix} -2 & 1 \\ 1 & 0 \end{pmatrix} \oplus E_8(-2) \]
where the first summand corresponds to the lattice spanned by $\sigma_0$ and $F$,
and $E_8(-2)$ is the image of the section classes under orthogonal projection away from the first summand.

Let $\iota_1$ be the involution of $S$ which acts fiberwise by multiplication by $-1$.
Switching the two fibers of the degree $2$ covering $f : \p^1 \to \p^1$
induces another involution $\iota_2 : S \to S$.
The involutions $\iota_1$ and $\iota_2$ commute and their composition
\[ g = \iota_1 \circ \iota_2 : S \to S \]
is symplectic. The invariance and coinvariant lattices are
\[ \Lambda^g = \mathrm{Span}_{\BZ}(\sigma_0, F), \quad \Lambda_g = E_8(-2). \]
The curve classes
\[ \beta_h = \sigma_0 + h F, \quad h \geq 0 \]
are invariant, primitive and untwisted. Hence
\[ \DT^{X}_{n, (\beta_h, d)} = \DT^{\text{untw}}_{n, h-1, d} \]
where $X$ is the CHL model associated to $(S, g)$.

Both the vertex methods of Section~\ref{sec: vertex} and the lattice argument of Section~\ref{first_coefficient_in_q}
can be applied in a parallel way to the model $X$.
The vertex computation for class $\beta_0$ yields the evaluation
\[
\sum_{d=0}^{\infty} \sum_{n \in \BZ} \DT^{X}_{n, (\beta_0, d)} q^{d-1} (-p)^n
=
\frac{1}{2} \frac{1}{\Theta(q,p)^2 \Delta_2(q)}.
\]
Using the degeneration formula and the fact that the theta function of the $E_8$ lattice is the Eisenstein series $E_4(q)$
yields
\[
\sum_{h=0}^{\infty} \sum_{n \in \BZ} \DT^{X}_{n, (\beta_h, 0)} t^{d-1} (-p)^n
= 
\frac{1}{2} \frac{E_4(t^2)}{\Theta(t,p)^2 \Delta(t)}.
\]
Both computations match Conjecture~\ref{Conj_intro_2}.
Further evidence for Conjecture~\ref{Conj_intro_2} can be obtained from the computation scheme of Section~\ref{subsec:computation_scheme}.

\appendix

\section{Twisted-twined elliptic genera} \label{appendix_elliptic_genera}
We list the twisted-twined elliptic genera associated to symplectic automorphisms of K3,
and define their multiplicative lift.
This provides the necessary background for Conjecture~\ref{Conj_intro_1}.

\subsection{List}
Let $g : S \to S$ be a symplectic automorphism of a K3 surface $S$ of order $N$.
By Mukai \cite{Mukai} the automorphism defines (up to conjugacy) an element $g \in M_{24}$.
The conjugacy class of $g$ only depends on the order $N$.
Let
\[ F^{(r,s)}_N = \Ell_{g^r, g^s}(K3),\ \  r,s \in \{ 0,1, \ldots, N-1 \}. \]
denote the $g^r$-twisted $g^s$-twined elliptic genera in the sense of \cite{GPRV}.
We usually drop the subscript $N$ from notation and we take the indices $r,s$ modulo $N$.
The functions $F^{(r,s)}$ are Jacobi forms of weight $0$ and index $1$ for the group $\Gamma(N) \rtimes \BZ^2$.\footnote{
For a general element $g \in M_{24}$
the associated twisted-twined elliptic genera might have a character.
However, if $g$ lies in $M_{23}$,
for example it arises as in our case from a symplectic automorphism, the character is trivial.}
The functions satisfy
\[
F^{(r,s)}\left( \frac{a \tau + b}{c \tau + d}, \frac{z}{c \tau + d} \right)
=
\exp\left( 2 \pi i \frac{c z^2}{c \tau + d} \right) F^{(cs + ar, ds + br)}(\tau, z)
\]
for all 
$\binom{a\ b}{c\ d} \in \mathrm{SL}_2(\BZ)$.
Hence the vector $( F^{(r,s)} )_{r,s}$ is a vector-valued Jacobi form for the full group Jacobi group $\SL_2(\BZ) \rtimes \BZ^2$.

Below is a list of the elliptic genera which we have taken from \cite{CD}.
The conjugacy classes corresponding to an automorphism of order $N=1, \ldots, 8$ are denoted
by 
\[ 1A,\, 2A,\, 3A,\, 4B,\, 5A,\, 6A,\, 7A,\, 8A \]
in \cite[Table 1]{CD} respectively.
As explained in \cite[2.4]{CD} the computations of \cite{CD} match the construction of Gaberdiel et all in \cite{GPRV}.
We have also checked the matching of \cite{CD} with \cite[Table~3]{NMW}.
For the modular and Jacobi forms we will follow the notation of Section~\ref{section:modular_forms}.
We will also use
\[ A = \frac{1}{4} \phi_{0,1} = \sum_{i=2}^{4} \frac{ \vartheta_i(\tau, z)^2 }{\vartheta_i(\tau,0)^2},
\quad B = \phi_{-2,1} = \frac{\vartheta_1(\tau,z)^2}{\eta(\tau)^6}.
\]

\vspace{7pt} \noindent 
\textbf{Case $N \in \{ 1, 2,3,5,7 \}$.}
For all $1 \leq s, r \leq N-1$ and $0 \leq k \leq N-1$,
\begin{align*}
F^{(0,0)}(\tau,z)  & = \frac{8}{N} A(\tau,z) \\
F^{(0,s)}(\tau,z)  & = \frac{8}{N (N+1)} A(\tau,z) - \frac{2}{N+1} B(\tau,z) \CE_N(\tau) \\
F^{(r,rk)}(\tau,z) & = \frac{8}{N (N+1)} A(\tau,z) + \frac{2}{N (N+1)} B(\tau,z) \CE_N\left( \frac{\tau + k}{N} \right)
\end{align*}


\vspace{7pt} \noindent 
\textbf{Case $N=4$}. For all $s \in \{ 0,1,2,3 \}$,
\begin{align*}
F^{(0,0)}(\tau, z) &=2A(\tau, z) , \\ 
 F^{(0,1)}(\tau, z) &= F^{(0,3)}(\tau, z) =
 \frac{1}{4}\left[\frac{4A}{3}-B \left(-\frac{1}{3}\CE_2(\tau)+2\CE_4(\tau) \right)\right], \\ 
 F^{(1,s)}(\tau, z) &= F^{(3,3s)}=\frac{1}{4}\left[\frac{4A}{3}+
 B \left(-\frac{1}{6}\CE_2(\frac{\tau+s}{2})+\frac{1}{2}\CE_4(\frac{\tau+s}{4}) \right)\right], \\ 
 F^{(2,1)}(\tau, z) &= F^{(2,3)}=\frac{1}{4}\left(\frac{4A}{3}  -\frac{B}{3} (5 \CE_2(\tau)-6\CE_4(\tau)\right) ,\\ 
 F^{(0,2)}(\tau, z) &=\frac{1}{4}\left(\frac{8A}{3}-\frac{4B}{3}\CE_2(\tau)\right) ,\\ 
 F^{(2,2s)}(\tau, z) &=\frac{1}{4}\left(\frac{8A}{3}+\frac{2B}{3}\CE_2(\frac{\tau+s}{2})\right) .
\end{align*}

\vspace{7pt} \noindent
\textbf{Case $N=6$}.
\begin{align*}
F^{(0,0)}&=\frac{4}{3}A\\ 
F^{(0,1)} &=F^{(0,5)}= \frac{1}{6}\left[\frac{2A}{3}-B \left(-\frac{1}{6}\CE_2(\tau)
-\frac{1}{2}\CE_3(\tau)+\frac{5}{2}\CE_6(\tau) \right)\right],\\ 
 F^{(0,2)}&=F^{(0,4)}= \frac{1}{6}\left[2A-\frac{3}{2}B \CE_3(\tau)\right], \\ 
 F^{(0,3)}&= \frac{1}{6}\left[\frac{8A}{3}-\frac{4}{3}B \CE_2(\tau)\right]. \\
F^{(1, k)}&=F^{(5, 5k)}=\frac{1}{6}\left[\frac{2A}{3}+B \left(-\frac{1}{12}\CE_2(\frac{\tau+k}{2})-\frac{1}{6}\CE_3(\frac{\tau+k}{3})+\frac{5}{12}\CE_6(\frac{\tau+k}{6}) \right)\right], \\
%
F^{(2, 2k+1)}&= \frac{A}{9}+\frac{B}{36}\left[ \CE_3(\frac{\tau+2+k}{3}) + \CE_2(\tau) - \CE_2\left( \frac{\tau+k+2}{3} \right) \right], \\
F^{(4, 4k+1)}&= \frac{A}{9}+\frac{B}{36}\left[ \CE_3(\frac{\tau+1+k}{3}) + \CE_2(\tau) - \CE_2\left( \frac{\tau+k+1}{3} \right) \right], \\
F^{(3, 1)}&=F^{(3, 5)}= \frac{A}{9}-\frac{B}{12}\CE_3(\tau)-\frac{B}{72}\CE_2(\frac{\tau+1}{2})+\frac{B}{8}\CE_2(\frac{3\tau+1}{2}), \\ 
F^{(3, 2)}&=F^{(3, 4)}= \frac{A}{9}-\frac{B}{12}\CE_3(\tau)-\frac{B}{72}\CE_2(\frac{\tau}{2})+\frac{B}{8}\CE_2(\frac{3\tau}{2}), \\ 
 F^{(2r,2rk)}&= \frac{1}{6}\left[2A+\frac{1}{2}B \CE_3(\frac{\tau+k}{3})\right],   \\
  F^{(3,3k)}&= \frac{1}{6}\left[\frac{8A}{3}+\frac{2}{3}B \CE_2(\frac{\tau+k}{2})\right].
\end{align*}

\vspace{7pt} \noindent 
\textbf{Case $N=8$}. In case $N=8$ we have
\begin{align*}
F^{(0, 0)} &= A, \\ 
 F^{(0,1)}&= F^{(0,3)}= F^{(0,5)}= F^{(0,7)},\\ 
&=\frac{1}{8}
\left[\frac{2A}{3}-B \left(-\frac{1}{2}\CE_4(\tau)+\frac{7}{3}\CE_8(\tau) \right)\right]. \\
F^{(r, rk)} & =\frac{1}{8}\left[\frac{2A}{3}+\frac{B}{8} \left(-\CE_4(\frac{\tau+k}{4})+\frac{7}{3}\CE_8(\frac{\tau+k}{8})\right)\right], r=1,3,5,7 \\
 F^{(2,1)}&= F^{(6,3)}= F^{(2,5)}= F^{(6,7)},\\ 
&= \frac{1}{8}\left[\frac{2A}{3}+\frac{B}{3} \left(-\CE_2(2\tau)+\frac{3}{2}\CE_4(\frac{2\tau+1}{4}) \right)\right];\\ 
 F^{(2,3)}&= F^{(6,5)}= F^{(2,7)}= F^{(6,1)},\\ 
&= \frac{1}{8}\left[\frac{2A}{3}+\frac{B}{3} \left(-\CE_2(2\tau)+\frac{3}{2}\CE_4(\frac{2\tau+3}{4}) \right)\right]. \\
 F^{(0,2)}&= F^{(0,6)}=\frac{1}{8}\left(\frac{4A}{3}-B \left(-\frac{1}{3}\CE_2(\tau)+2\CE_4(\tau) \right)\right),\\ 
 F^{(0,4)}&=\frac{1}{8}\left(\frac{8A}{3}-\frac{4B}{3}\CE_2(\tau)\right),\\ 
 F^{(2,2s)}&= F^{(6,6s)}=\frac{1}{8}\left(\frac{4A}{3}+B \left(-\frac{1}{6}\CE_2(\frac{\tau+s}{2})+\frac{1}{2}\CE_4(\frac{\tau+s}{4}) \right)\right),\\ 
 F^{(4,4s)}&=\frac{1}{8}\left(\frac{8A}{3}+\frac{2B}{3}\CE_2(\frac{\tau+s}{2})\right),\\ 
 F^{(4,2)}&= F^{(4,6)}=\frac{1}{8}\left(\frac{4A}{3}-\frac{B}{3} (3\CE_2(\tau)-4\CE_2(2\tau)\right),\\ 
 F^{(4,2k+1)}&=\frac{1}{8}\left(\frac{2A}{3}+B \left(\frac{4}{3}\CE_2(4\tau)-\frac{2}{3}\CE_2(2\tau)-\frac{1}{2}\CE_4(\tau) \right)\right).
\end{align*}

%

%

\subsection{Multiplicative lift} \label{Section:twisted-twined} \label{sec:mult_lift}
We define the Borcherds or multiplicative lift of the twisted-twined elliptic genera
$F^{(r,s)}, r,s=0,\ldots, N-1$.

Consider the discrete Fourier transform
\[
\hat{F}^{(r,\ell)}(\tau,z)
=
\sum_{s=0}^{N-1}
e^{-2 \pi i s \ell/N}
F^{(r,s)}(\tau, z).
\]
Since $\hat{F}^{(r,s)}$ are Jacobi forms of index $1$ we have the expansion
\[
\hat{F}^{(r,\ell)}(\tau,z)
=
\sum_{b \in \{0,1 \}} \sum_{\substack{n \in \BZ/N \\ j \in 2 \BZ+b}} \hat{c}_b^{(r,\ell)}(4n-j^2) q^n p^j
\]

\begin{rmk}
By a direct check we have
\[ \hat{c}^{(0,0)}_{1}(-1) = 2, \quad \hat{c}^{(0,0)}_{0}(0) = 20 - |\Lambda_g|. \]
In particular the first coefficient of $\hat{F}^{(0,0)}$ is
\begin{align*}
\big[ \hat{F}^{(0,0)} \big]_{q^0} 
& = \hat{c}^{(0,0)}_{1}(-1) (p + p^{-1}) + \hat{c}^{(0,0)}_{0}(0) \\
& = \sum_{i,j} h^{i,j}(S/\BZ_N) (-1)^{i+j} p^i
\end{align*}
is the $\chi_{y}$ genus of the quotient $S/G$.
This generalizes the corresponding property of the elliptic genus $\Ell(K3)$. \qed
\end{rmk}

Let
\[ Z = \begin{pmatrix} \tau & z \\ z & \sigma \end{pmatrix} \]
be the standard coordinates on the Siegel upper half space and write
\[ q = e^{ 2 \pi i \tau}, \quad t = e^{ 2 \pi i \sigma}, \quad p = e^{2 \pi i z}. \]

\begin{defn}[\cite{DJS2}] \label{defn_multiplicative_lift}
The multiplicative lift of the twisted-twined elliptic genera
$( F^{(r,s)} )_{r,s = 0, \ldots, N-1}$ is defined by
\[
\widetilde{\Phi}_N(Z)
=
q t^{1/N} p
\prod_{b=0,1} \prod_{r=0}^{N-1} \prod_{ \substack{ k \in \BZ+ \frac{r}{N}, \ell \in \BZ \\ j \in 2 \BZ+b \\ k, \ell \geq 0, j<0 \text{ if } k=\ell=0}}
(1 - q^{\ell} t^k p^{j})^{\hat{c}_b^{(r,\ell)}(4 k \ell - j^2) }.
\]
\end{defn}
\vspace{5pt}

By \cite[Sec.3]{DJS2}, see also \cite{PV}, $\widetilde{\Phi}_N(Z)$ is a Siegel modular form for a certain congruence subgroup of $\Sp(4,\BZ)$ of weight
\[ \frac{1}{2} \hat{c}_0^{(0,0)}(0) \, = \, 10 - \frac{1}{2} | \Lambda_g | \, =\,  \left\lceil \frac{24}{N+1}\right\rceil -2.
\]
%
A consequence of the modularity is the $t \leftrightarrow q^N$ symmetry
\[ \widetilde{\Phi}_N(t^{1/N}, q^N, p) = \widetilde{\Phi}_N(q,t,p). \]
This may also seen directly as follows. 
By the explicit values in Appendix~\ref{appendix_elliptic_genera} the $\hat{F}^{(r,\ell)}$ are symmetric in $(r,\ell)$,
\[ \hat{F}^{(r,\ell)} = \hat{F}^{(\ell,r)} 
\ \  \text{ for all } r,\ell = 0, \ldots, N-1.
\]
Hence
\begin{equation} \label{sym_chat}
\hat{c}_b^{(r,\ell)}(D) = \hat{c}_b^{(\ell,r)}(D)
\ \  \text{ for all } r,\ell,b,D.
\end{equation}
This implies the symmetry by definition.

\section{Heterotic string and the duality group} \label{appendix2}
\vspace{5pt}
\begin{center}
By \textbf{Sheldon Katz}\footnote{
University of Illinois at Urbana-Champaign, Department of Mathematics,
Email: \url{katz@math.uiuc.edu}}
and \textbf{Georg Oberdieck}
\end{center}
\vspace{8pt}

In this appendix, our main goal is to explain the difference between the twisted and untwisted primitive invariants of order two CHL models in the context of physics. A secondary goal is to provide a cursory explanation of some of the physics background.  

We start with a discussion of several relevant ideas about dual string models and the duality group of the CHL model.  Although many of these ideas are necessarily relegated to a ``black box," we strive to formulate some of the ideas in precise mathematical language in the hopes that other mathematicians will be able to benefit from the ideas of physics as we have.

We will adopt the device of initially describing relevant concepts from physics in italics.  We will then selectively give some precise mathematical properties that these structures are supposed to have.  

\emph{String theory} is a 10-dimensional \emph{physical theory}, with variants including \emph{Type IIA} string theory, \emph{Type IIB} string theory, and \emph{Heterotic $E_8\times E_8$} string theory.  For brevity, we refer to these theories as IIA, IIB, and heterotic respectively.  String theory takes place on a 10-dimensional Lorentzian manifold $M^{10}$.  

String theory can be \emph{compactified} on a compact Riemannian manifold $X$ with a Ricci flat metric.  This means that we take $M^{10}=X\times M^d$, where $M^d$ is a Lorentzian manifold of dimension $d=10-\dim_{\mathbb{R}}(X)$.  By ``integrating out" the fields on $X$, we obtain an $d$-dimensional \emph{effective theory} on $M^d$, the physical spacetime of the theory.  The physical properties of the $d$-dimensional theory are determined by the geometry of $X$, so that calculations and theorems about the geometry of $X$ inform physics.  Conversely, ideas in physics such as dualities lead to non-trivial predictions about the geometry of $X$.  This two-way flow of information is at the core of the geometry-physics dictionary.

Consider $IIA[S\times E]$, IIA string theory compactified on $S\times E$.  Since $\dim_{\mathbb{R}}(S\times E)=6$, this is a 4-dimensional theory, and is in fact a \emph{4-dimensional $\mathcal{N}=4$ theory}. The $\mathcal{N}=4$ adjective describes the amount of \emph{supersymmetry}, as we now outline.

For simplicity, let's assume that our spacetime $M^d$ is $d$-dimensional Minkowski space.  Then the physical theory has a group of symmetries containing the isometry group of $M^d$.  At each $p\in M$ we have an induced Lie algebra of infinitesimal symmetries.  The $\mathcal{N}=4$ \emph{supersymmetry algebra} is a particular $\mathbb{Z}_2$-graded Lie algebra $\Fg=\Fg^0\oplus\Fg^1$ of infinitesimal symmetries, with $\Fg^0$ containing the infinitesimal isometries.  Here, $\mathcal{N}=4$ means $\dim(\Fg^1)=16$.\footnote{The $\mathbb{Z}_2$-graded Lie algebra $\Fg$ is not arbitrary but is constrained by physical principles.  The minimum value of $\dim(\Fg^1)$ in a 4-dimensional supersymmetric theory is 4, the dimension of a minimal real spin representation in signature $(3,1)$.  For any $\mathbb{Z}_2$-graded Lie algebra $\Fg=\Fg^0\oplus\Fg^1$, the even part $\Fg^0$ is an ordinary Lie algebra, and the odd part $\Fg^1$ is a $\Fg^0$-module.  In 4-dimensional minimal ($\mathcal{N}=1$) supersymmetry, $\Fg^1$ is the real spin representation $\Fs$.  In a 4-dimensional $\mathcal{N}=2$ theory we have $\Fg^1=\Fs^{\oplus 2}$ and in a 4-dimensional $\mathcal{N}=4$ theory we have $\Fg^1=\Fs^{\oplus4}$.  The dimension of $\Fg^1$ is called the number of supercharges of the theory.}  These supersymmetry algebras have precise mathematical definitions, see for example \cite{Freed}.  The amount of supersymmetry in a string compactification on a Calabi-Yau manifold $X$ is determined by the particular string theory used and the holonomy group of $X$. 
The holonomy group acts naturally on a fiber of the complexified spin bundle.  If the holonomy group acts trivially on a nonzero vector, this determines a covariantly constant spinor on $X$ which is used to construct a supersymmetry.
The holonomy group of $S\times E$ acts trivially on a 2-dimensional subspace, leading to $\mathcal{N}=4$ supersymmetry. By contrast, the holonomy group of a Calabi-Yau threefold $X$ only fixes a 1-dimensional space of spinors, so $\text{IIA}[X]$ only has half as much supersymmetry, $\mathcal{N}=2$.\footnote{
Another description is as follows.
Assume the Calabi--Yau threefold carries the action of an abelian variety of dimension $k$. Then
the corresponding $IIA$ theory is of type $\CN = 2^{k+1}$.}

Irreducible representations of $\Fg$ are completely classified.  Among these are the 1/2 BPS representations and 1/4 BPS representations.  To these respective representations correspond 1/2 (resp.\ 1/4) BPS states in the physical theory.

A key point is that reduced Donaldson-Thomas invariants of $S\times E$ can be directly related to 1/4 BPS invariants in the associated physical theory.  We will return to this point shortly.

There are other 4-dimensional $\mathcal{N}=4$ theories.  IIA and IIB theory have the same amount of supersymmetry in 10~dimensions.  It follows immediately that IIB$[S\times E]$ is also an $\mathcal{N}=4$ theory.

Now, the heterotic string in 10~dimensions has only half of the supersymmetry as IIA or IIB.  It follows that Het$[S\times E]$ is an $\mathcal{N}=2$ theory.  To get an $\mathcal{N}=4$ theory, we need to compactify the heterotic string on a manifold so that the holonomy acts trivially on the entire 4-dimensional space of spinors. An obvious choice is the flat 6-torus $T^6=(S^1)^6$ with trivial holonomy.  So Het$[T^6]$ is an $\mathcal{N}=4$ theory.

The assertion of heterotic-IIA duality is that 
\[ \text{IIA}[S\times E]= \text{Het}[T^6], \]
that is, these two 4-dimensional $\mathcal{N}=4$ theories are the same, albeit in a non-obvious way.  Also,
\[ \text{IIA}[S\times E]=\text{IIB}[S\times E] \]
by \emph{$T$-duality}.  These assertions have an enormous amount of content.  In the context of CHL models, these give predictions about their reduced DT invariants coming from calculations with no obvious relationship to DT theory or algebraic geometry.

To begin to extract some content, we next observe that the \emph{states} of a physical theory have \emph{charges}, which live in a \emph{charge lattice}.  The states also transform in a representation of the supersymmetry algebra as we have already mentioned.

Before discussing our 4-dimensional $\mathcal{N}=4$ theories, a more elementary example of a charge lattice is that the electric charges of the known elementary particles live in the electric charge lattice
\[
\Lambda_e=e\mathbb{Z}\subset\mathbb{R},
\]
where $e$ is the absolute value of the charge of the electron.  Charged particles interact with the photon, the force carrier which is described in Yang-Mills theory by a $U(1)$ \emph{gauge field}, identified with a connection on a principal $U(1)$ bundle on the 4-dimensional spacetime $M^4$. The \emph{electromagnetic field strength} is up to a scalar the curvature of the connection, $F\in\Omega^2_{M^4}$.  

In this formulation, Maxwell's equations take the simple form $dF=d*F=0$.  These equations are clearly invariant under the duality $F\to *F$, which underlies \emph{electric-magnetic duality}. When $F$ is expressed in terms of the electric field $\vec{E}$ and magnetic field $\vec{B}$, the duality transformation takes $\vec{E}$ to $-\vec{B}$ and $\vec{B}$ to $\vec{E}$.\footnote{In Lorentzian signature in 4~dimensions, we have $**=-\mathrm{Id}$ on 2-forms.} This means that if we would ever observe magnetic monopoles, their behavior in a magnetic field would be (with some sign differences) the same as the behavior of an electric monopole (charged particle) in an electric field, and the behavior of a magnetic monopole in an electric field would be the same as that of a charged particle in a magnetic field.  We say that an electrically charged particle has magnetic charge 0 and a magnetically charged particle has electric charge zero.  Magnetic charges are quantized (by the Dirac quantization condition), i.e.\ they also live in a rank~1 lattice.

In this theoretical framework, particles can have both electric and magnetic charge.  Such particles are called \emph{dyons}.  Their \emph{electromagnetic charge} lives in a rank~2 charge lattice, the direct sum of the electric and magnetic lattices.  Using the fundamental electric and magnetic charges to identify this lattice with $\BZ^2$, we can express the dyon charges as $(q,p)$, with $q$ units of electric charge and $p$ units of magnetic charge.  With this identification, the standard inner product on $\BZ^2$ provides a pairing on the electromagnetic lattice, which also has intrinsic physical meaning. Electric-magnetic duality extends an action of $\SL_2(\BZ)$ on the charge lattice $\BZ^2$.  We say that $\SL_2(\BZ)$ is the \emph{duality group}.  We emphasize that a duality transformation can transform \emph{all} of the fields in theory.  For example, the action of the duality transformation described by
\[
S=\left(
\begin{array}{cc}
0&-1\\
1&0
\end{array}
\right)
\] 
acts as
\[
\left(
\begin{array}{c}
\vec{E}\\
\vec{B}
\end{array}
\right)\mapsto
S\left(
\begin{array}{c}
\vec{E}\\
\vec{B}
\end{array}
\right)=
\left(
\begin{array}{c}
-\vec{B}\\
\vec{E}
\end{array}
\right)
\]
not only exchanges electric and magnetic charges (up to sign), $S(q,p)=(-p,q)$, but also exchanges electric monopoles with magnetic monopoles, and other physical quantities.  In extending these notions to string dualities, we sometimes relate physical quantities in one theory which have an algebro-geometric description to quantities in another theory which do not admit an algebro-geometric description.  

A Yang-Mills theory with gauge group $G=U(1)^r$ physically contains $r$ gauge fields and correspondingly has a rank $r$ electric charge lattice $\Lambda_e\simeq\mathbb{Z}^r$.  The components of the charge of a particle can be thought of as the electric charges of the particle with respect to the individual gauge fields.  Including magnetic charges, we get an electromagnetic charge lattice of rank $2r$.

In both IIA$[S\times E]$ and Het$[T^6]$, we have $G=U(1)^{28}$ (at generic points of the physical moduli space).  Without going in to details, one simply enumerates the fields in the 10~dimensional theory which appear as gauge fields in 4~dimensions after compactification.  This is a well-defined and simple computation in geometry.  In each case, we find 28 gauge fields, for very different reasons. 

The geometry-physics dictionary further identifies the electric charge lattice of IIA$[S\times E]$ with 
\begin{equation}\label{eqn:IIElec}
\Lambda_e=H^*(S,\BZ)\oplus U^2.
\end{equation}
The $U^2$ part of the lattice is associated with $E$ via \emph{momentum and winding modes} of the string wrapping the independent 1-cycles of $E$ (and can be identified with $H^*(E,\BZ)$). Thus $\Lambda_e\simeq E_8(-1)^2\oplus U^6$, the unique even self-dual lattice of signature $(6,22)$.

In this situation, the magnetic lattice $\Lambda_m$ is isomorphic to $\Lambda_e$.  Thus
\begin{equation}\label{eqn:IIEM}
\Lambda=\Lambda_e\oplus \Lambda_m=\left(H^*(S,\BZ)\oplus U^2\right)\oplus\left(H^*(S,\BZ)\oplus U^2\right)
\end{equation}
Identifying $\Lambda$ with $\Lambda_e\otimes\BZ^2$, the \emph{duality group} is $\mathrm{Isom}(\Lambda_e)\times\SL_2(\BZ)\simeq\mathrm{SO}(6,22,\BZ)\times\SL_2(\BZ)$, acting on $\Lambda$ in the obvious way. To each element $\sigma$ of the duality group, there is a (non-geometric) automorphism $f_\sigma$ of the physical theory, taking a BPS state $\rho$ with charge $Z(\rho)\in \Gamma$ to a BPS state $f_\sigma(\rho)$ with charge $\sigma \cdot Z(\rho)$. In this way, string theory reveals a much larger symmetry group than we are able to see in algebraic geometry proper, with powerful consequences for algebraic geometry.

The heterotic theory compactified on $T^6$ is a theory including $E_8\times E_8$ bundles on $T^6$.  In this case we have 
\begin{equation}\label{eqn:HetElec}
\Lambda_e=E_8(-1)^2\oplus U^6. 
\end{equation}
The $U^6$ part of the lattice is associated with momentum and winding modes of the string wrapping the independent 1-cycles of $T^6$.

Again, the magnetic lattice $\Lambda_m$ is isomorphic to $\Lambda_e$.  Thus
\begin{equation}\label{eqn:HetEM}
\Lambda=\Lambda_e\oplus \Lambda_m=
\left(E_8(-1)^2\oplus U^6\right)\oplus \left(E_8(-1)^2\oplus U^6\right).
\end{equation}
Remarkably, we immediately see that the respective electric charge lattices (\ref{eqn:IIElec}) and (\ref{eqn:HetElec}) of IIA$[S\times E]$ and Het$[T^6]$ are isomorphic.  Similarly, the corresponding electromagnetic charge lattices (\ref{eqn:IIEM}) and (\ref{eqn:HetEM}) are also isomorphic.  

We see that we obtain two 4-dimensional ${\mathcal N}=4$ theories with the same charge lattice, 
so it is a natural question to ask if these two theories are actually the same.  This question was asked more than 20 years ago (with more evidence than sketched above) and no contradictions have been found to date.  This is what is meant by heterotic-type II duality in our context.  We will refine the duality in the CHL context shortly, but for simplicity we continue to place our discussion in IIA$[S\times E]=$Het$[T^6]$ before passing to CHL.

We write the electromagnetic charges as $(Q,P)\in \Lambda_e\oplus \Lambda_m$.  The three quantities $Q^2, Q\cdot P, P^2$ are manifestly invariant under $\mathrm{Isom}(\Lambda_e)$ (and these quantities generate the ring of all invariants if $Q\wedge P$ is a primitive rank~2 lattice).  
DT invariants arise from D6-D2-D0 branes in IIA$[S\times E]$. A $Dp$-brane is a $p$-dimensional object moving in time.  BPS branes exist only for $p$ even in IIA and $p$ odd in IIB.  

A PT-pair $\CO\to\CF$ has a K-theory class.  Identifying K-theory with cohomology over the rationals we get components in $H_i(S\times E,\BQ)\simeq H^{6-i}(S\times E,\BQ)$ for $i\in\{6,2,0\}$ only.  This is the mathematical meaning of the D6-D2-D0 terminology.  After compactification of IIA on $S\times E$, we are left with a point particle in $M^4$ moving in time.  These particles are the charged BPS states in our theory.

The electric and magnetic charges $(Q,P)$ of the BPS states corresponding to $\DT^{S\times E}_{n,(\gamma,d)}$ have been spelled out in the physics literature. 
Let $e_1$ and $e_2$ denote the generators of the hyperbolic lattice as presented in Section~\ref{subsection_symplectic_auto}.  Then given a class $(\gamma,d)\in H_2(S\times E,\BZ)$, the electric and magnetic charges, and their invariants, are given by
\begin{equation}\label{eqn:SEcharges}
Q=(\gamma,ne_2,0),\quad P=(0,e_1+(d-1)e_2,0)
\end{equation}
with invariants
\[ Q^2=\gamma^2,\ P^2=2d-2,\ Q\cdot P=n.\]
The shift from $de_2$ to $(d-1)e_2$ in the component of the magnetic charge associated to $d\in H_2(E,\BZ)$ arises from the quotient by $E$ used in defining the reduced DT invariants.

For each $(Q,P)$, we can ask about the degeneracy\footnote{The degeneracy is an index in physics, defined as a supertrace in the relevant Hilbert space, roughly the difference between the dimensions of spaces of \emph{bosonic} and \emph{fermionic} states up to an omitted universal factor.} of 1/4-BPS states with charges $(Q,P)$.
If $Q \wedge P$ is a primitive rank~2 lattice,
the group $\mathrm{SO}(6,22;\BZ)$ acts transitively on the set of charges with
fixed $(Q^2,Q\cdot P,P^2)$. It follows that this index only depends on $Q^2,Q\cdot P,P^2$, so we write the degeneracy equivalently as $\mathsf{d}(Q,P)$ or $\mathsf{d}(Q^2/2,P^2/2,Q\cdot P)$.
We form the generating function
\[
Z(q,t,p)=\sum_{h = -1}^\infty
\sum_{k=-1}^\infty
\sum_{n\in\BZ}\mathsf{d}(h,k,n) t^h q^k (-p)^n.
\]



The degeneracies $\mathsf{d}(Q^2/2,P^2/2,Q\cdot P)$
have been computed using the IIB description.  The result is
\[
Z(q,t,p)=-\frac1{\chi_{10}(q,t,p)}.
\]
Passing back to the IIA description and using the charges (\ref{eqn:SEcharges}), we have
\[ \DT^{S\times E}_{n,(\beta_h,d)}=\mathsf{d}(h-1,d-1,n). \]
This leads immediately to the Igusa cusp form conjecture of DT theory on $S \times E$, proven in \cite{ObPix2,FM1}.  

A useful table for understanding the content of dualities appears in \cite[Table~3.1]{Cheng}, with conventions for electric charges and magnetic charges switched from ours.   We give two examples to show how far duality takes us outside of algebraic geometry.  For example, our D0-brane charge $n$ of DT theory corresponds to a \emph{momentum quantum number} of the string around one of the 1-cycles of $T^6$ in the heterotic theory.  Furthermore, even within IIA$[S\times E]$, we see that a $\mathrm{SO}(6,22;\BZ)$ transformation can take the D0-brane charge to non-geometric objects such as momentum quantum numbers around the 1-cycles of $E$.  This means that the full content of physical dualities cannot be understood within algebraic geometry proper.

We turn at last to our main interest, the CHL models.
Letting $X=(S\times E)/\BZ_N$, the CHL model is IIA$[X]$.  It is also a 4-dimensional $\mathcal{N}=4$ theory, with a heterotic dual Het$[T^6/\BZ_N]$.  The  electric charge lattice is \cite{NMW}
\begin{equation}\label{eqn:chle}
\Lambda_e=\left(H^*(S,\BZ)^g\right)^*\oplus U\oplus U\left(\frac1N\right)
\end{equation}
and the magnetic charge lattice is
\begin{equation}\label{eqn:chlm}
\Lambda_m=\Lambda_e^*=H^*(S,\BZ)^g\oplus U\oplus U\left(N\right).
\end{equation}
The electric and magnetic charge lattices are different in general, but we still have
\[ \Lambda = \Lambda_e \oplus \Lambda_m. \]
The degeneracies of several CHL models were determined in \cite{DJS2} including order 2 models, and the degeneracies for additional models are worked out in \cite{NMW} using IIB$[X]$ (which is dual to IIA$[X]$ via \emph{T-duality} on a particular 1-cycle of $E$ depending on the particular translation in $E$  used to construct the CHL model $X$).  The charges in the IIB theory are also described in \cite[Table~3.1]{Cheng}.  For the order two CHL model, the results are consistent with the twisted DT partition function but not with the untwisted DT partition function.  We conclude this appendix by explaining that there is no contraction with physics.

As in Section~\ref{subsection_symplectic_auto} of the main paper, let
\[ P = \frac{1}{N} \sum_{i=0}^{N-1} g^{i} : H^{\ast}(S,\BQ) \to H^{\ast}(S,\BQ) \]
be the projection operator. Then we have\footnote{If $\alpha = P(\alpha')$
for some $\alpha' \in H^{\ast}(S,\BZ)$ then $\langle \alpha, \beta\rangle = \langle \alpha', \beta \rangle$
for any $\beta \in (H^{\ast}(S,\BZ))^g$.
Hence $P( H^{\ast}(S,\BZ) ) \subset \left( H^{\ast}(S,\BZ)^g \right)^{\ast}$.
The converse follows since $\Lambda_{K3}^{\ast} \to
\left( H^{\ast}(S,\BZ)^g \right)^{\ast}$ is surjective.}
\[ P( H^{\ast}(S,\BZ) ) \cong \left( H^{\ast}(S,\BZ)^g \right)^{\ast}. \]

To specify a CHL model we fix a primitive vector $v \in H^1(E)$ of square zero and
let $\delta = \frac{1}{N} v \in \frac{1}{N} H^1(E)$.
The duality group of the CHL model, denoted $G_{\text{dual}}$, contains the product\footnote{
The exact duality group of the CHL model has not yet been fully determined.
In \cite{PV2} it is argued that the duality group $G_{\text{dual}}$ should be strictly bigger than $\Gamma_1(N) \times C(\hat{g})$, as it should contain the Fricke involution $(Q,P)\mapsto (-P,NQ)$. We will only consider the product $\Gamma_1(N) \times C(\hat{g})$ here. 
}
\[ \Gamma_1(N) \times C(\hat{g}) \subset G_{\text{dual}} \]
where
\[ \Gamma_1(N) = \left\{ \begin{pmatrix}a&b\\c&d\end{pmatrix} \middle| c\equiv 0, a,d \equiv 1 \text{ modulo } N \right\} \]
and $C(\hat{g})$ is the centralizer of the pair $\hat{g} = (g, \delta)$
in $\mathrm{SO}(6,22;\BZ)$, i.e.
\[ C(\hat{g}) =
\{ \varphi \in \mathrm{SO}(6,22;\BZ)\, |\, \varphi(\delta) = \delta, \ \varphi \circ g = g \circ \varphi \} \]
where we have extended the action of $g$ on $\Lambda_{K3}$
to an action on $\Lambda$ by letting it act as the identity on $U \oplus U(N)$.

Let now $(Q,P) \in \Lambda$ be a pair of electro-magnetic charges such that $Q \wedge P$ is a primitive vector in the lattice $\Lambda_e\wedge \Lambda_m := \Lambda_m^* \wedge \Lambda_m$. The dyon degeneracy $\mathsf{d}(Q,P)$ is invariant under the duality group and
we need to understand the orbits of $(Q,P)$ under the duality group.  In the case of $S\times E$, we have that $\mathrm{SO}(6,22;\BZ)$ acts transitively on the set of $(P,Q)$ with fixed $(Q^2,\ Q \cdot P,\ P^2)$ and $Q \wedge P$ primitive, hence the conclusion that the degeneracies are of the form $\mathsf{d}(Q^2/2,P^2/2,Q\cdot P)$.  These triples are not preserved by $\mathrm{SL}(2,\BZ)$, but this additional part of the duality group gives relations between the degeneracies.  For example, in the $S\times E$ case, the matrix $S\in \mathrm{SL}(2,\BZ)$ acts on charges via $S(Q,P)=(-P,Q)$, implying the relation $\mathsf{d}(Q^2/2,P^2/2,Q\cdot P)=\mathsf{d}(P^2/2,Q^2/2,-Q\cdot P)$.  This can be verified in DT theory as the symmetry $\chi_{10}^{-1}(q,t,p)=\chi_{10}^{-1}(t,q,p^{-1})$, which makes sense after a change in stability condition corresponding to an analytic continuation of the expansion of $\chi_{10}^{-1}$.

In the CHL case, we claim that there are distinct charges $(Q,P)$ with $Q \wedge P$ primitive
and with the same $(Q^2,\ Q \cdot P,\ P^2)$ which are not related by a duality transformation.
Hence the degeneracies need not be of the form $\mathsf{d}(Q^2/2,P^2/2,Q\cdot P)$, and more care is needed in drawing conclusions from a calculation in a dual physical model.  

For example, define the \emph{residue} of $(Q,P)$ to be
the class of $Q$ in $\Lambda_m^{\ast} / \Lambda_m$ where we have identified $\Lambda_e \equiv \Lambda_m^{\ast}$, i.e.
\[ r(Q,P) = [ Q ] \in \Lambda_m^{\ast} / \Lambda_m. \]
Then for $g = \binom{a\ b}{c\ d} \in \Gamma_1(N)$ we have
\[
r\left( g (Q,P) \right)
=
[ dQ - bP ] = [ Q ] = r(Q,P)
\]
where we have used that $N \Lambda_m^{\ast} \subset \Lambda_m$.
Moreover, since every $h \in C(\hat{g})$ arises from an automorphism of the unimodular lattice $\Gamma^{22,6}$ the map $h$ acts by the identity on the discriminant $\Lambda_m^{\ast} / \Lambda_m$. Hence
\[ r\left( h (Q,P) \right) = r\left( hQ,hP \right) = [hQ] = [Q] = r(Q,P). \]
Moreover, in the case $N=2$ the residue distinguishes between twisted and untwisted classes:
a primitive $\gamma \in P(H_2(S,\BZ))$ is untwisted (or twisted) depending on whether
its residue $r(\gamma)$ vanishes (or not).
A basic question is whether the residue of $(Q,P)$ is indeed invariant under the full duality group $G_{\text{dual}}$? More generally, we can ask:

\vspace{3pt}
\noindent \textbf{Problem.} Determine the full set of invariants
of the pair $(Q,P)$ for $Q \wedge P$ primitive  under the duality group.

\vspace{7pt}
The orbits of $(Q,P)$ with $Q \wedge P$ primitive under the duality group should correspond to the deformation classes of a pair of a CHL model together with a fixed ample primitive class on the K3.
Hence understanding the set of invariants is the first step towards
identifying the Donaldson--Thomas invariants of all CHL models.





\vspace{5pt}
\noindent \emph{Acknowledgements:} We would like to thank Nikita Nekrasov and Max Zimet for helpful correspondence.  S.K. was supported by NSF grants DMS-1502170 and DMS-1802242, together with DMS-1440140 while at MSRI during Spring 2018.

\end{document}